\documentclass[a4paper,11pt]{article}

\usepackage{amsfonts}
\usepackage{amssymb}
\usepackage{amsthm}
\usepackage{amsmath}
\usepackage{a4wide}
\usepackage{mathrsfs}
\usepackage{epsfig}
\usepackage{palatino}
\usepackage{tikz,fp,ifthen}
\usepackage{float}
\usepackage[colorinlistoftodos,textwidth=2.3cm]{todonotes} 
\newcommand{\pdfgraphics}{\ifpdf\DeclareGraphicsExtensions{.pdf,.jpg}\else\fi}
\usepackage{graphicx}

\usepackage{float}
\usepackage{color}
\definecolor{hanblue}{rgb}{0.27, 0.42, 0.81}
\definecolor{red}{rgb}{1.0, 0.0, 0.0}
\usepackage[colorlinks, citecolor=hanblue,linkcolor=blue, urlcolor = hanblue,hypertexnames=false]{hyperref}
\usepackage{cleveref}

\usepackage{mathtools}

\theoremstyle{plain}

\newtheorem{teo}{Theorem}[section]
\newtheorem{lemma}[teo]{Lemma}
\newtheorem{prop}[teo]{Proposition}
\newtheorem{cor}[teo]{Corollary}

\theoremstyle{definition}
\newtheorem{dfnz}[teo]{Definition}

\theoremstyle{remark}
\newtheorem{rem}[teo]{Remark}

\numberwithin{equation}{section}

\def\R{\mathbb R}

\newcommand{\e }{\varepsilon }
\newcommand{\ep }{\varepsilon }

\newcommand{\ga }{\gamma}

\newcommand{\intbar}{\etaathop{\int\etaakebox(-13.5,0){\rule[4pt]{.7em}{0.3pt}}
\kern-6pt}\nolimits}

\newcommand{\be}{\begin{equation}}
\newcommand{\ee}{\end{equation}}
\newcommand{\bea}{\begin{equation*}}
\newcommand{\eea}{\end{equation*}}

\def\R{{{\mathbb R}}}
\def\SS{{{\mathbb S}}}

\def\ders{\partial_s}

\def\pol{{\mathfrak{p}}}

\newcommand{\N}{\mathbb{N}}

\def\be{\begin{equation}}
\def\ee{\end{equation}}
\def\bea{\begin{eqnarray*}}
\def\bean{\begin{eqnarray}}
\def\eean{\end{eqnarray}}
\def\eea{\end{eqnarray*}}

\DeclarePairedDelimiter\scal{\langle}{\rangle} 

\newcommand{\de}{{\,\mathrm{d}}} 

\newcommand{\pa}{\partial}

\usepackage[maxbibnames=99,backend=biber, sorting=nyt, doi=false, url=false, isbn=false]{biblatex}
\begin{document}

\pdfgraphics 

\title{A survey of the elastic flow of curves and networks}

\author{Carlo Mantegazza \footnote{Dipartimento di Matematica e Applicazioni, Universit\`a di Napoli Federico II, Via Cintia, Monte S. Angelo
80126 Napoli, Italy} \and Alessandra Pluda \footnote{ Dipartimento di Matematica, Universit\`a di Pisa, Largo
    Bruno Pontecorvo 5, 56127 Pisa, Italy} \and Marco Pozzetta \footnotemark[2]}
\date{\today}

\maketitle

\begin{abstract}
\noindent We collect and present in a unified way several results in recent years about the elastic flow of curves and networks, trying to draw the state of the art of the subject. In particular, we give a complete proof of global existence and smooth convergence to critical
points of the solution of the elastic flow of closed curves in $\mathbb{R}^2$.
In the last section of the paper we also discuss a list of open problems.
\end{abstract}


\section{Introduction}

The study of \emph{geometric flows} is a very flourishing mathematical field and geometric evolution equations have been applied to a variety of topological, analytical and physical problems, giving in some cases very fruitful results. 
In particular, a great attention has been devoted to the analysis of 
harmonic map flow, mean curvature flow and Ricci flow.
With serious efforts from the members of the mathematical community the understanding of these topics gradually improved and it culminated with Perelman's proof of the Poincar\'{e} conjecture making use of the Ricci flow, completing Hamilton's program. The enthusiasm for such a marvelous result encouraged more and more researchers to investigate properties and applications of general geometric flows and the field branched out in various different directions, including \emph{higher order flows},
among which we mention the \emph{Willmore flow}.

In the last two decades a certain number of authors focused on
the one dimensional analog of the Willmore flow (see~\cite{kuschat1}): the elastic flow of curves
and networks.  The elastic energy of a regular and sufficiently smooth curve  $\gamma$  
is a linear combination of the $L^2$-norm of the curvature $\boldsymbol{\kappa}$ and the length, namely 
\begin{equation*}
\mathcal{E}\left(\gamma\right):=\int_{\gamma} \vert\boldsymbol{\kappa}\vert^{2}+\mu \,\mathrm{d}s\,.
\end{equation*} 
where $\mu\geq 0$.
In the case of networks
(connected sets composed of $N\in\mathbb{N}$ curves that meet at their endpoints
in junctions of possibly different order)
the functional is defined in a similar manner: 
one sum the contribution of each curve (see Definition~\ref{eef}).
Formally the elastic flow is the $L^2$ gradient flow of the functional $\mathcal{E}$
(as we show in Section~\ref{section:defi})
and the solutions of this flow are the object of our interest in the current paper.

To the best of our knowledge the problem was taken into account for the first time by Polden.
In his Doctoral Thesis~\cite[Theorem 3.2.3.1]{poldenthesis} he proved that, if we take as initial datum a smooth immersion of the circle in the plane, then there exists a smooth solution
to the gradient flow problem  for all positive times. Moreover, as times goes to infinity,
it converges along subsequences to a critical point of the  functional (either a circle, or
a symmetric figure eight or a multiple cover of one of these).
Polden was also able to prove that if the winding number of the initial curve is $\pm 1$
(for example the curve is embedded), then it converges to a unique circle
~\cite[Corollary 3.2.3.3]{poldenthesis}. 
In the early 2000s Dziuk, Kuwert and Sch\"{a}tzle generalize the global existence and subconvergence result to $\mathbb{R}^n$ and derive an algorithm to treat the flow and compute several numerical examples.
Later the analysis was extended to non closed curve, both with fixed endpoint and with non--compact branches. The problem for networks was first proposed in 
2012 by Barrett, Garcke and N\"{u}rnberg~\cite{bargarnu}.

Beyond the study of this specific problem
there are quite a lot of catchy variants.
For instance, as for a regular $C^2$ curve $\gamma:I\to\mathbb{R}^2$
it holds $\boldsymbol{k}=\partial_s \tau$,
where $\tau$ is the unit tangent vector and $\partial_s$ denotes derivative with respect to the arclength parameter $s$ of the curve, we can introduce the \emph{tangent indicatrix}: a scalar map $\theta:I\to\mathbb{R}$ such that 
$\tau=(\cos\theta,\sin\theta)$. 
Then we can write the elastic energy in terms of the angle spanned by the tangent vector.
By expressing the $L^2$ corresponding gradient flow by means of $\theta$
one get another geometric evolution equation. This is a second order gradient flow and it has been
first considered by~\cite{We93} and then further 
investigated by~\cite{LiYa18, LiYaSc15, NoPo, OkPoWh, We95}.

Critical points of total squared curvature subject to fixed length  are called \emph{elasticae}, or \emph{elastic curves}. 
Notice that for any $\mu>0$ the elasticae are (up to homothety) exactly the critical points of the energy
$\mathcal{E}$. Elasticae have been studied since Bernoulli and Euler as the elastic energy was used as a model for the bending energy of an elastic rod~\cite{Truesdell} and 
more recently Langer and Singer contributed to their classification~\cite{LaSi84, LaSi85}
(see also~\cite{DjHaMl, Li89}). \\
The $L^2$--gradient flow of $\int\vert\boldsymbol{\kappa}\vert^2\,\mathrm{d}s$ when the curve is subjected to fixed length is studied in~\cite{DaLiPo14,DaLiPo17,DzKuSc02,RuSp20}.

It is worth to mention also results about the Helfrich flow~ \cite{DaPo14, Wh15},
the elastic flow with constraints~\cite{Koiso,Ok07,Ok08}
and other fourth (or higher) order flows~\cite{AbBu19,AbBu20,McWhWu19,McWhWu17,WhWh}.

In the following table we collect some contributions on the elastic flow of curves (closed or open)
and networks. The first column concerns papers containing detailed proofs of short time existence results. The initial datum can be a function of a suitably chosen Sobolev space, or H\"{o}lder space, or the curves are smooth. In the second column we place the articles that show existence for all 
positive times or that describe obstructions to such a desired result. 
When the flow globally exists, it is natural to wonder about the behavior of the solutions for 
$t\to+\infty$. Papers that answer this question are in the third column.
The ambient space may vary from article to article: it can be $\mathbb{R}^2$, $\mathbb{R}^n$, or a Riemannian manifold.

\smallskip

\noindent
\begin{tabular}{|c|c|c|c|}
\hline
& Short Time Existence  & Long time Behavior  & Asymptotic Analysis \\
\hline     
  \small{closed curves} &  \cite{poldenthesis} & \cite{DzKuSc02}  \cite{poldenthesis} & \cite{DzKuSc02} \cite{MaPo20} \cite{poldenthesis} \cite{Po20} \\
\hline   
  \small{open curves  Navier b.c.}   &   \cite{NoOk14} &  \cite{NoOk14} & \cite{NoOk14} \cite{NoOk17} \\
\hline     
  \small{open curves, clamped b.c.}   & \cite{Sp17} & \cite{Ch12} & \cite{DaPoSp16} \cite{NoOk17}  \\  
 \hline 
     \small{non compact curves}   &  & \cite{NoOk14} & \cite{NoOk14}\\    
\hline     
  \small{networks}   & \cite{DaChPo20} \cite{GaMePl1} \cite{GaMePl2}  &  \cite{DaChPo19} \cite{GaMePl2} & \cite{DaChPo19} \\   
\hline     
\end{tabular}\vspace{0.2cm}
We refer also to the two recent PhD theses~\cite{menzelthesis,PozzettaThesis}.

\medskip

The aim of this expository paper is to arrange (most of) this material in a unitary form,
proving in full detail the results for the elastic flow of closed curves
and underlying the differences with the other cases.

For simplicity we restrict to the Euclidean plane as ambient space.
In Section~\ref{section:iniziale} we define the flow, deriving the motion
equation and the necessary boundary conditions for open curves and networks.
In the literature curves that meet at junctions of order at most three are usually considered, while here the order
of the junctions is arbitrary.

In Section~\ref{sec:ShortTimeExistence} we show short time existence and uniqueness (up to reparametrizations) for the elastic flow of closed curve, supposing that the initial datum is 
H\"{o}lder--regular (Theorem~\ref{geomexistence}).
The notion of $L^2$-gradient flow gives rise to a fourth order parabolic quasilinear  PDE, where  
the motion in tangential direction is not specified. To obtain a non--degenerate equation
we fix the tangential velocity, then getting first a \emph{special flow} (Definition ~\ref{def:SpecialFlow}).
We find a unique solution of the special flow (Theorem~\ref{existenceanalyticprob})
using a standard linearization procedure and a fixed point argument. Then a key point is to ensure that solving the special flow is enough to obtain a  solution to the original problem. How to overcome this issue is explained in 
Section~\ref{Sec:invariance}.
The short time existence result can be easily adapted to open curves (see Remark~\ref{generalizzazioneopen}),
but present some extra difficulties in the case of networks, that we explain in Remark~\ref{generalizzazionenetwork}.

One interesting feature following from the parabolic structure of the elastic flow is that solutions are smooth for any (strictly) positive times. We give the idea of two possible lines of proof of this fact and we refer to~\cite{GaMePl2}
and~\cite{DaChPo20} for the complete result.

Section~\ref{sec:LongTimeExistence} is devoted to the prove that the flow of either closed or open curves
with fixed endpoint 
exists globally in time (Theorem~\ref{teo:GlobalExistence}).
The situation for network is more delicate and it depends on the evolution of the length of the curves composing the network and on the angles formed by the tangent vectors of the curves
concurring at the junctions (Theorem~\ref{thm:LongTimeNetworks}).

In Section~\ref{sec:SmoothConvergence} we first show that, as time goes to infinity,
the solutions of the elastic flow of closed curve convergence 
along subsequences to stationary points of the elastic energy, up to translations and reparametrizations. We shall refer to this phenomenon as the \emph{subconvergence} of the flow. We then discuss how the subconvergence can be promoted to full convergence of the flow, namely to the existence of the full asymptotic limit as $t\to+\infty$ of the evolving flow, up to reparametrizations (Theorem~\ref{FullConvergence}). The proof is based on the derivation and the application of a \L ojasiewicz--Simon gradient inequality for the elastic energy.

We conclude the paper with a list of open problems.

\section{The elastic flow}\label{section:iniziale} 

A regular curve $\gamma$ is a continuous map $\gamma:[a,b]\to\R^2$ which is differentiable on $(a,b)$ and such that $\vert\partial_x\gamma\vert$ never vanishes on $(a,b)$.
Without loss of generality, from now on we consider $[a,b]=[0,1]$.

In the sequel 
we will abuse the word ``curve'' to refer both to the parametrization of a curve, the equivalence class of reparametrizations, or the support in $\mathbb{R}^2$.

We denote by $s$ the arclength parameter and
we will pass to the arclength parametrization of the curves when it is more convenient
without further mentioning. 
We will also extensively use the arclength measure $\mathrm{d}s$ when integrating
with respect to the volume element $\mu_g$ on $[0,1]$ induced by a regular rectifiable curve $\gamma$, namely,
given a $\mu_g$-integrable function $f$ on $[0,1]$ it holds
\begin{equation*}
\int_{[0,1]}f\,\mathrm{d}\mu_g=\int_0^1f(x)\vert\partial_x\gamma(x)\vert\,\mathrm{d}x
=\int_0^{\ell(\gamma)}f(x(s))\,\mathrm{d}s=:\int_{\gamma}f\,\mathrm{d}s\,,
\end{equation*}
where $\ell(\gamma)$ is the length of the curve $\gamma$.

\begin{dfnz}\label{network}
A planar \emph{network} $\mathcal{N}$ is a connected set in $\mathbb{R}^2$ given by a finite union of images of regular curves $\gamma^i:[0,1]\to\R^2$ 
that may have endpoints of order one fixed in the plane and 
curves that meet at junctions of different order $m\in\mathbb{N}_{\geq 2}$.\\
The \emph{order} of a junction $p\in\R^2$ is the number $\sum_i \{0,1\}\cap \sharp(\gamma^i)^{-1}(p)$.
\end{dfnz}

As special cases of networks we find:
\begin{itemize}
\item a single curve (either closed or not);
\item a network of three curves whose endpoints meet at two different triple junction
(the so-called \emph{Theta});
\item a network of three curves with one common endpoint
at a triple junction and the other three
endpoint of order one (the so called \emph{Triod}).
\end{itemize}

Notice that when it is more convenient, we will parametrize a closed curve as a map $\gamma:\mathbb{S}^1\to\mathbb{R}^2$.\\


In order to calculate the integral of an $N$-tuple $f=(f^1,\ldots,f^N)$ of functions along the network $\mathcal{N}$ 
composed of the $N$ curves $\gamma^i$ we adopt the notation
\begin{equation*}
\int_{\mathcal{N}} f\,\mathrm{d}s:=\sum_{i=1}^{N}\int_{\gamma^i}^{}f_{|\gamma^i}\,\mathrm{d}s
=\sum_{i=1}^N\int_0^1 f^i\vert\partial_x\gamma^i\vert\,\mathrm{d}x\,.
\end{equation*}
If $\mu=(\mu^1, \ldots,\mu^N)$ with $\mu^i\geq 0$, then the notation 
$\int_{\mathcal{N}} \mu f\,\mathrm{d}s$ stands for 
$\sum_{i=1}^N\int_0^1 \mu^i f^i\vert\partial_x\gamma^i\vert\,\mathrm{d}x$.\\

Let $\gamma:[0,1]\to\mathbb{R}^2$ be a regular curve and $f:(0,1)\to\mathbb{R}$ a Lebesgue measurable
function. For $p\in[1,\infty)$ we define
\begin{equation*}
\Vert f\Vert^p_{L^p(\mathrm{d}s)}:=\int_\gamma \vert f\vert^p\,\mathrm{d}s
=\int_0^1\vert f(x)\vert^p\vert\partial_x\gamma(x)\vert\,\mathrm{d}x
\end{equation*}
and
\begin{equation*}
L^p(\mathrm{d}s):=\left\lbrace f:(0,1)\to\mathbb{R}\;\text{Lebesgue measurable with}\;
\Vert f\Vert^p_{L^p(\mathrm{d}s)}<+\infty\right\rbrace\,.
\end{equation*}
We will also use the $L^\infty$--norm
\begin{equation*}
\lVert f^i\rVert_{L^\infty(\mathrm{d}s)}:=
\mathrm{ess\; sup}_{L^\infty(\mathrm{d}s)} \,\vert f^i\vert\,.
\end{equation*}
Whenever we are considering continuous functions, we identify the supremum norm with the $L^\infty$ norm and denote it by $\left\lVert\cdot\right\rVert_\infty$.

We remark here that for sake of notation we will simply write $\Vert\cdot\Vert_{L^p}$
instead of $\lVert \cdot \rVert_{L^p(\de s)}$
both for $p\in [1,\infty)$ and $p=\infty$ whenever there is no risk of confusion.

We will analogously write
\begin{equation*}
\lVert f\rVert_{L^p}:=\sum_{i=1}^{N}\lVert f^i\rVert_{L^p(\mathrm{d}s)}\quad\text{for all}\; p\in [1,\infty)\quad\text{and}\quad
\lVert f\rVert_{L^\infty}:=\sum_{i=1}^{N}\lVert f^i\rVert_
{L^\infty(\mathrm{d}s)}\,,
\end{equation*}
for an $N$-tuple of functions $f$ along a network $\mathcal{N}$.

Assuming that $\gamma^i$ is of class $H^2$, we denote by $\boldsymbol{\kappa}^i:=\partial_s^2\gamma^i$ the curvature vector to the curve $\gamma^i$, which is defined at almost every point and the curvature is nothing but $\kappa^i:=\vert\boldsymbol{\kappa}^i\vert$.
We recall that in the plane we can write the curvature vector as $\boldsymbol{\kappa}^i=k^i	\nu^i$
where $\nu^i$ is the counterclockwise rotation of $\frac{\pi}{2}$ of the unit tangent vector $\tau^i:=\vert\partial_x\gamma^i\vert^{-1}(\partial_x\gamma^i)$ to a curve $\gamma^i$ and then $k^i$ is the \emph{oriented curvature}.

\begin{dfnz}\label{elasticenergy}
Let $\mu^i \geq 0$ be fixed for 
$i\in \{1, \ldots, N\}$.
The \emph{elastic energy functional} $\mathcal{E}_\mu$
of a network $\mathcal{N}$ given by $N$ curves $\gamma^i$ of class $H^2$ is defined by
\begin{equation}\label{eef}
\mathcal{E}_\mu\left(\mathcal{N}\right):= 
 \int_{\mathcal{N}} \vert \boldsymbol{\kappa}\vert^{2}\,\mathrm{d}s
+\mu\, \mathrm{L}(\mathcal{N})
:=\sum_{i=1}^N\left(\int_{\mathcal{N}^{i}}  (k^i)^{2} \,\mathrm{d}s
+\mu^i \, \ell (\gamma^i) \right)\,,
\end{equation}
and $\mu \mathrm{L}(\mathcal{N})$ is named weighted  
global length of the network $\mathcal{N}$.
\end{dfnz}

\subsection{First variation of the elastic energy}\label{sec:FirstVariation}

The computation of the first variation
has been carried several times in full details in the literature, 
both in the setting of closed curves or networks. We refer for instance to~\cite{bargarnu, MaPo20}. 

Let $N\in\mathbb{N}$, $i\in\{1,\ldots,N\}$.
Consider a network $\mathcal{N}$ composed of $N$ curves, parametrized by $\gamma^i:[0,1]\to\mathbb{R}^2$ of class $H^4$.
In order to compute the first variation of the energy we can suppose that the curves meet at one junction, which is of order $N$
and $\gamma^i(1)$ is some fixed point in $\R^2$ for any $i$. That is
\begin{equation*}
\gamma^1(0)=\ldots=\gamma^N(0)\,, \quad \gamma^i(1)=P^i\in\mathbb{R}^2\,.
\end{equation*}
The case of networks with other possible topologies can be easily deduced from the presented one. 
We consider a variation $\gamma^i_\varepsilon=\gamma^i+\varepsilon\psi^i$
of each curve $\gamma^i$ of $\mathcal{N}$ 
with $\varepsilon\in\mathbb{R}$ and 
$\psi^i:[0,1]\to\mathbb{R}^2$ of class $H^2$.
We denote by $\mathcal{N}_\varepsilon$ the network composed
of the curves $\gamma^i_\varepsilon$, which are regular whenever $|\e|$ is small enough.
We need to impose that the structure of the network $\mathcal{N}$ is preserved in the variation:
we want the network $\mathcal{N}_\varepsilon$ to still have one junction of order $N$
and we want to preserve the position of the other endpoints
$\gamma^i_\varepsilon(1)=P^i$. To this aim we require
\begin{equation*}
\psi^1(0)=\ldots=\psi^N(0)\,,\qquad \psi^i(1)=0
\qquad\forall\,i\in\{1,\ldots,N\}\,.
\end{equation*}

By definition of the elastic energy functional of networks, we have
\begin{equation*}
\mathcal{E}_\mu(\mathcal{N}_\varepsilon)
=\sum_{i=1}^N\int_{\gamma^i_\varepsilon}(k^i_\varepsilon)^2+\mu^i\,\mathrm{d}s
=\sum_{i=1}^N\int_{\gamma^i_\varepsilon}\vert\boldsymbol{\kappa}^i_\varepsilon\vert^2
+\mu^i\,\mathrm{d}s\,.
\end{equation*}
We introduce the operator $\partial_s^\perp$ (that acts on a vector field $\varphi$)
defined as the normal component of $\partial_s\varphi$ along the curve $\gamma$, that is
$\partial_s^\perp\varphi=\partial_s\varphi
-\left\langle \partial_s\varphi,\partial_s\gamma\right\rangle\partial_s\gamma$.
Then a direct computation  yields the following identities:
\begin{equation}\label{eq:VariazioniOggettiGeometrici}
\begin{split}
    \partial_\e\de s_\e &= \scal{\partial_s\psi^i,\tau^i_\e} \de s_\e
    = \left( \partial_s\scal{\psi^i,\tau^i_\e} - \scal{\psi^i,\boldsymbol{\kappa}^i_\e} \right) \de s_\e , \\
    \partial_\e\partial_s - \partial_s\partial_\e &= \left( \scal{\boldsymbol{\kappa}^i_\e, \psi^i}- \partial_s \scal{\tau^i_\e, \psi^i} \right) \partial_s,\\
    \partial_\e \tau^i_\e &= \partial_s^\perp (\psi^i)^\perp + \scal{\tau^i_\e, \psi^i}\boldsymbol{\kappa}^i_\e, \\
    \partial_\e \boldsymbol{\kappa}^i_\e &= (\partial_s^\perp)^2 (\psi^i)^\perp - \scal{\partial_s^\perp(\psi^i)^\perp, \boldsymbol{\kappa}^i_\e}\tau^i_\e + \scal{\tau^i_\e, \psi^i}\partial_s\boldsymbol{\kappa}^i_\e + \scal{\boldsymbol{\kappa}^i_\e,\psi^i}\boldsymbol{\kappa}^i_\e,
\end{split}
\end{equation}
for any $i$ on $(0,1)$, where $s$ is the arclength parameter of $\gamma_\e$ for any $\e$.
Therefore, evaluating at $\e=0$, we obtain
\begin{align}\label{primostep}
 \frac{d}{d\varepsilon}\mathcal{E}_\mu(\mathcal{N}_\varepsilon)\Big\vert_{\varepsilon=0}
 & =\sum_{i=1}^{N} \left[\int_{\gamma^i} 2 \langle \boldsymbol{\kappa}^{i},
 \partial^2_{s} \psi^{i} \rangle \, \mathrm{d}s   
 + \int_{\gamma^i} (-3 |\boldsymbol{\kappa}^i|^2+\mu^i)
\left\langle \tau^{i},\partial_s\psi^{i}\right\rangle \,\mathrm{d}s \right]\,.
\end{align}
Moreover, denoting by $\partial_s^\perp(\cdot):
=\partial_s(\cdot)-\langle\partial_s (\cdot),\tau \rangle\tau$,
we have
\begin{align*}
\partial_s \boldsymbol{\kappa}^i&
=\partial^\perp_s \boldsymbol{\kappa}^i- |\boldsymbol{\kappa}^i|^2\tau^i\,,\\
\partial_s^2 \boldsymbol{\kappa}^{i}&= (\partial_s^\perp)^2 \boldsymbol{\kappa}^{i} 
-3 \langle \partial_s \boldsymbol{\kappa}^i, \boldsymbol{\kappa}^i\rangle 
\partial_s \gamma^i - |\boldsymbol{\kappa}^i|^2 \boldsymbol{\kappa}^i\,,
\end{align*}
then, using these identities and 
integrating~\eqref{primostep} by parts twice, one gets
\begin{align}
 \frac{d}{d\varepsilon}\mathcal{E}_\mu(\mathcal{N}_\varepsilon) \Big\vert_{\e=0}
 =&\sum_{i=1}^{N} 
 \int_{\gamma^i}  \left\langle
 2 (\partial^\perp_s)^2 \boldsymbol{\kappa}^{i}+ |\boldsymbol{\kappa}^i|^2 \boldsymbol{\kappa}^{i} 
 -\mu^i \boldsymbol{\kappa}^{i},
\psi^{i} \right\rangle \,\mathrm{d}s\nonumber\\
 +&\sum_{i=1}^{N} \left[ 2 \left. \langle \boldsymbol{\kappa}^{i},
 \partial_s \psi^{i}\rangle \right|_0^1 
+  \left. \langle -2\partial^\perp_s \boldsymbol{\kappa}^{i}
- |\boldsymbol{\kappa}^i|^2\tau^i+\mu^i\tau^i, \psi^{i}\rangle \right|_0^1 
 \right] \label{firtsvar} \\
 =&\sum_{i=1}^{N} 
 \int_{\gamma^i}  \left\langle
 2 (\partial^\perp_s)^2 \boldsymbol{\kappa}^{i}+ |\boldsymbol{\kappa}^i|^2 \boldsymbol{\kappa}^{i} 
 -\mu^i \boldsymbol{\kappa}^{i},
\psi^{i} \right\rangle \,\mathrm{d}s \nonumber\\
 +&\sum_{i=1}^{N} 2  \langle \boldsymbol{\kappa}^{i}(1),
 \partial_s \psi^{i}(1)\rangle - 2  \langle \boldsymbol{\kappa}^{i}(0),
 \partial_s \psi^{i}(0)\rangle \nonumber\\
+& \left\langle\left( \sum_{i=1}^{N}-2\partial^\perp_s \boldsymbol{\kappa}^{i}(0)
- \vert\boldsymbol{\kappa}^i(0)\vert^2\tau^i(0)+\mu^i\tau^i(0)\right), \psi^{1}(0)\right\rangle\,.\label{boundarypart}
\end{align}
As we chose arbitrary fields $\psi^i$, we can split 
$\partial_s\psi^i$ into normal and tangential components as
\begin{align*}
\partial_{s}\psi^{i}&=\partial_{s}^\perp\psi^{i}+\partial_{s}^\parallel\psi^{i}
=\left\langle \partial_{s}\psi^i,\nu^i\right\rangle\nu^i
+\left\langle \partial_{s}\psi^i,\tau^i\right\rangle\tau^i
=:\left(\psi^i_{s}\right)^\perp\nu^i+\left(\psi^i_{s}\right)^\parallel \tau^i\,.
\end{align*}
This allows us to write
\begin{align*}
 \left\langle \boldsymbol{\kappa}^{i},
 \partial_s \psi^{i}\right\rangle
 &= \left\langle k^i\nu^i,
\left(\psi^i_{s}\right)^\perp\nu^i+\left(\psi^i_{s}\right)^\parallel \tau^i\right\rangle
=k^i  \left(\psi^i_{s}\right)^\perp\,, 
\end{align*}
and we can then partially reformulate the first variation in terms of the oriented curvature and its derivatives:
\begin{align}\label{riformulazione}
& \frac{d}{d\varepsilon}\mathcal{E}_\mu(\mathcal{N}_\varepsilon) \Big\vert_{\e=0}
 =\sum_{i=1}^{N} \int_{\gamma^i} 
 \left(2\partial_ s^2 k^{i} +(k^i)^3-\mu^i k^i\right) \left(\psi^i\right)^\perp\,\mathrm{d}s\nonumber\\
 &+2\sum_{i=1}^N \left.  k^i  \left(\psi^i_{s}\right)^\perp  \right|_0^1 
+ \left\langle\left( \sum_{i=1}^{N}-2\partial^\perp_s \boldsymbol{\kappa}^{i}(0)
- \vert\boldsymbol{\kappa}^i(0)\vert^2\tau^i(0)+\mu^i\tau^i(0)\right), \psi^{1}(0)\right\rangle\,.
\end{align}

\subsection{Second variation of the elastic energy}\label{sec:SecondVariation}

In this part we compute the second variation of the elastic energy functional $\mathcal{E}_\mu$. We are interested only in showing its structure and analyze some properties, instead of computing it explicitly (for the full formula of the second variation we refer to~\cite{DaPoSp16} and~\cite{Po20}). In fact, we will exploit the properties of the second variation only in the proof of the smooth convergence of the elastic flow of closed curves in~\Cref{sec:SmoothConvergence}. In particular, we we will not need to carry over boundary terms in the next computations.

\medskip

Let $\gamma:(0,1)\to\R^2$ be a smooth curve and let $\psi:(0,1)\to \R^2$ be a vector field in $H^4(0,1)\cap C^0_c(0,1)$, that is, $\psi$ identically vanishes out of a compact set contained in $(0,1)$. In this setting, we can think of $\gamma$ as a parametrization of a part of an arc of a network or of a closed curve. We are interested in the second variation
\[
\frac{d^2}{d\e^2} \mathcal{E}_\mu(\gamma+\e\psi)\Big\vert_{\e=0}\,.
\]
By~\eqref{firtsvar} we have
\[
\begin{split}
    \frac{d^2}{d\e^2} \mathcal{E}_\mu(\gamma+\e\psi)\Big\vert_{\e=0}
    &= \frac{d}{d\e}\Big\vert_{\e=0} \int_{\gamma_\e} \left\langle
    2(\partial_s^\perp)^2\boldsymbol{\kappa}_\e + |\boldsymbol{\kappa}_\e|^2 \boldsymbol{\kappa}_\e - \mu \boldsymbol{\kappa}_\e
    , \psi\right\rangle \de s_\e\,,
\end{split}
\]
where $\boldsymbol{\kappa}_\e$ is the curvature vector of $\gamma_\e=\gamma+\e\psi$, for any $\e$ sufficiently small.

We further assume that $\gamma$ is a critical point for $\mathcal{E}_\mu$ and that $\psi$ is normal along $\gamma$. Then
\[
\frac{d^2}{d\e^2} \mathcal{E}_\mu(\gamma+\e\psi)\Big\vert_{\e=0}
= \int_\gamma \left\langle \partial_\e  \big\vert_{\e=0} \left(
    2(\partial_s^\perp)^2\boldsymbol{\kappa}_\e + |\boldsymbol{\kappa}_\e|^2 \boldsymbol{\kappa}_\e - \mu \boldsymbol{\kappa}_\e
    \right), \psi\right\rangle \de s\,.
\]
Using~\eqref{eq:VariazioniOggettiGeometrici}, if $\phi_\e$ is a normal vector field along $\gamma_\e$ for any $\e$ and we denote $\phi\coloneqq \phi_0$, a direct computation shows that
\[
\partial_\e |_{_{\e=0}} \partial_s^\perp \phi_\e - \partial_s^\perp \partial_\e|_{_{\e=0}} \phi_\e 
= \scal{\psi, \boldsymbol{\kappa}} \partial_s^\perp \phi - \scal{\partial_s^\perp \phi, \partial_s^\perp \psi} \tau + \scal{\phi,\boldsymbol{\kappa}}\partial_s^\perp \psi\,.
\]
Hence $\partial_\e |_{_{\e=0}} (\partial_s^\perp)^2 \boldsymbol{\kappa}_\e$ can be computed applying the above commutation rule twice, first with $\phi_\e = \partial_s^\perp \boldsymbol{\kappa}_\e$ and then with $\psi_\e = \boldsymbol{\kappa}_\e$. One easily obtains
\[
\partial_\e |_{_{\e=0}} (\partial_s^\perp)^2 \boldsymbol{\kappa}_\e 
= (\partial_s^\perp)^4 \psi + \Omega(\psi)\,,
\]
where $\Omega(\psi)\in L^2(\mathrm{d}s)$ is a normal vector field along $\gamma$, depending only on $k,\psi$ and their ``normal derivatives'' $\partial_s^\perp$ up to the third order. Moreover the dependence of $\Omega$ on $\psi$ is linear. For further details on these computations we refer to~\cite{MaPo20}.

Using~\eqref{eq:VariazioniOggettiGeometrici} it is immediate to check that $\partial_\e |_{_{\e=0}} \left( |\boldsymbol{\kappa}_\e|^2 \boldsymbol{\kappa}_\e - \mu \boldsymbol{\kappa}_\e \right)$ yields terms that can be absorbed in $\Omega(\psi)$. Therefore we conclude that
\[
\frac{d^2}{d\e^2} \mathcal{E}_\mu(\gamma+\e\psi)\Big\vert_{\e=0}
= \int_\gamma \left\langle 2 (\partial_s^\perp)^4\psi + \Omega(\psi) , \psi \right\rangle \de s\,.
\]
By polarization, we see that the second variation of $\mathcal{E}_\mu$ defines a bilinear form $\delta^2\mathcal{E}_\mu(\varphi,\psi)$ given by
\[
\delta^2\mathcal{E}_\mu(\varphi,\psi) = \int_\gamma \left\langle 2 (\partial_s^\perp)^4\varphi + \Omega(\varphi) , \psi \right\rangle \de s\,,
\]
for any normal vector field $\varphi,\psi$ of class $H^4\cap C^0_c$ along $\gamma$, which is a smooth critical point of $\mathcal{E}_\mu$.

\subsection{Definition of the flow}\label{section:defi}

In this section we define the elastic flow for curves and networks.
We formally derive it as the $L^2$--gradient flow of the elastic energy functional~\eqref{eef}.
We need to derive the normal velocity defining the flow. The reasons why a gradient flow is defined in term of a normal velocity are related to the invariance under reparametrization of the energy functional and we will come back on this point more deeply in~~\Cref{Sec:invariance}.

\medskip

The analysis of the boundary terms appeared in the computation of the first variation play an important role in the definition of the flow. Indeed, a correct definition of the flow depends on the fact that the velocity defining the evolution should be the opposite of the ``gradient'' of the energy. Hence we need to identify such a gradient from the formula of the first variation and, in turn, analyze the boundary terms appearing.

Suppose first that the network is composed only 
of one closed curve $\gamma\in C^{\infty}([0,1],\mathbb{R}^2)$.
This means that for every $k\in\mathbb{N}$ we have $\partial_x^k\gamma(0)=\partial_x^k\gamma(1)$ and $\gamma$ can be seen as a smooth periodic function on $\R$. Then a variation field $\psi$ is just a periodic function as well and no further boundary constraints are needed and then the boundary terms in~\eqref{firtsvar} are automatically zero.
Then~\eqref{boundarypart} reduces to 
\begin{equation*}
 \frac{d}{d\varepsilon}\mathcal{E}_\mu(\gamma_\varepsilon)_{|\varepsilon=0}
=\int_{\gamma}  \left\langle
 2 (\partial^\perp_s)^2 \boldsymbol{\kappa}+ |\boldsymbol{\kappa}|^2 \boldsymbol{\kappa}
 -\mu \boldsymbol{\kappa},
\psi \right\rangle \,\mathrm{d}s \,.
\end{equation*}
We  have formally written  
the directional derivative of $\mathcal{E}_\mu$ of each curve
in the direction $\psi$ as the $L^2$--scalar product
of $\psi$ and the vector $2 (\partial^\perp_s)^2 \boldsymbol{\kappa}
+|\boldsymbol{\kappa}|^2 \boldsymbol{\kappa} -\mu \boldsymbol{\kappa}$. Hence we can understand 
$2 (\partial^\perp_s)^2 \boldsymbol{\kappa}
+|\boldsymbol{\kappa}|^2 \boldsymbol{\kappa} -\mu \boldsymbol{\kappa}$
to be the gradient of $\mathcal{E}_\mu$.
We then set the normal velocity driving the flow to be the opposite of such a gradient, that is
\begin{equation}\label{motionequation}
(\partial_t\gamma)^\perp=-2 (\partial^\perp_s)^2 \boldsymbol{\kappa}
-|\boldsymbol{\kappa}|^2 \boldsymbol{\kappa} +\mu \boldsymbol{\kappa}\,,
\end{equation}
where, again, $(\cdot)^\perp$ denotes the 
normal component of the velocity $\partial_t\gamma$ of the curve $\gamma$:
\begin{equation*}
(\partial_t\gamma)^\perp=\partial_t\gamma-\left\langle\partial_t\gamma,\tau\right\rangle\tau\,.
\end{equation*}
In $\mathbb{R}^2$ it is possible to express  
the evolution equation in terms of the scalar curvature:
\begin{align*}
\left\langle\partial_t\gamma,\nu\right\rangle\nu=(\partial_t\gamma)^\perp
=2 (\partial^\perp_s)^2 \boldsymbol{\kappa}
+|\boldsymbol{\kappa}|^2 \boldsymbol{\kappa}-\mu \boldsymbol{\kappa}
=\left(2\partial_ s^2 k +(k)^2k-\mu k\right)\nu\,.
\end{align*}
This last equality can be directly deduced from~\eqref{riformulazione}.
In this way we have derived an equation that describe
the normal motion of each curve.

\medskip

We pass now to consider,
exactly as in Section~\ref{sec:FirstVariation},  
a network composed of $N$ curves, parametrized by $\gamma^i:[0,1]\to\mathbb{R}^2$
with $i\in\{1,\ldots,N\}$, that meet at one junction of order $N$ at $x=0$
and have the endpoints at $x=1$ fixed in $\mathbb{R}^2$.
We denote by $\mathcal{N}_\varepsilon$ the network composed
of the curves $\gamma^i_\varepsilon=\gamma^i+\varepsilon\psi^i$
with $\psi^i:[0,1]\to\mathbb{R}^2$ such that
\begin{equation*}
\psi^1(0)=\ldots=\psi^N(0)\,,\qquad \psi^i(1)=0\quad \forall\,i\in\{1,\ldots,N\}\,.
\end{equation*}

Since the energy of a network is defined as the sum of
of the energy of each curve, it is reasonable to define
the gradient of $\mathcal{E}_\mu$ as the sum of the gradient
of the energy of each curve composing the network,
that we have identified with the vectors
$2 (\partial^\perp_s)^2 \boldsymbol{\kappa}^i
+|\boldsymbol{\kappa}^i|^2 \boldsymbol{\kappa}^i-\mu^i \boldsymbol{\kappa}^i$.
Hence, a network is a critical point of the energy when the the vectors
$2 (\partial^\perp_s)^2 \boldsymbol{\kappa}^i
+|\boldsymbol{\kappa}^i|^2 \boldsymbol{\kappa}^i-\mu^i \boldsymbol{\kappa}^i$ vanish and the boundary terms in~\eqref{boundarypart} 
are zero.
Depending on the boundary constraints imposed on the network, i.e., its topology or possible fixed endpoints,
we aim now to characterize the set of networks
fulfilling  boundary conditions
that imply
\begin{equation*}
\sum_{i=1}^{N} \left[ 2 \left. \langle \boldsymbol{\kappa}^{i},
 \partial_s \psi^{i}\rangle \right|_0^1 
+  \left. \langle -2\partial^\perp_s \boldsymbol{\kappa}^{i}
- |\boldsymbol{\kappa}^i|^2\tau^i+\mu^i\tau^i, \psi^{i}\rangle \right|_0^1 
 \right]=0\,.
\end{equation*}
Let us discuss the main possible cases of boundary conditions separately.

\medskip

\noindent \textit{Curve with constraints at the endpoints}

As we have mentioned before, if the network is composed 
of one curve, but this curve is not closed, then we fix its endpoint,
namely $\gamma(0)=P\in\mathbb{R}^2$ and $\gamma(1)=Q\in\mathbb{R}^2$.
As already shown in 
in Section \ref{sec:FirstVariation},  to 
maintain the position of the endpoints, 
we require
$\psi(0)=\psi(1)=0$, that automatically implies
\begin{equation*}
\left. \langle -2\partial^\perp_s \boldsymbol{\kappa}^{i}
- |\boldsymbol{\kappa}^i|^2\tau^i+\mu^i\tau^i, \psi^{i}\rangle \right|_0^1 =0\,,
\end{equation*}
in the computation of the first variation.
On the other hand we are free to chose $\partial_s\psi$ as test fields in the first variation. Suppose for example that $\partial_s\psi(0)=\nu$
(where $\nu$ is the unit normal vector to the curve $\gamma$)
and $\partial_s\psi(1)=0$, then from~\eqref{boundarypart}
we obtain $k(0)=0$ and so $\boldsymbol{k}(0)=0$.
Interchanging the role of $\partial_s\psi(0)$ and $\partial_s\psi(1)$
we have $k(1)=\boldsymbol{k}(1)=0$.

Hence we end up with the following set of conditions
\begin{equation*}
\begin{cases}
\gamma(0)=P\\
\gamma(1)=Q\\
\boldsymbol{\kappa}(0)=\boldsymbol{\kappa}(1)=0\,,
\end{cases}
\end{equation*}
known in the literature as \emph{natural or Navier} boundary conditions.

\medskip

However, 
since the elastic energy functional is a functional of the second order,
it is legitimate to impose also that the unit tangent vectors at the endpoint of the curve are
fixed, namely that the curve is \emph{clamped}.
Hence we now have $\gamma(0)=P, \gamma(1)=Q, \tau(0)=\tau_0,\tau(1)=\tau_1$ as constraints.
This time these boundary conditions affects 
the class of test function requiring $\partial_s\psi(0)=\partial_s\psi(1)=0$, that, together with
$\psi(0)=\psi(1)=0$, automatically set~\eqref{boundarypart} to zero.

\medskip

\noindent \textit{Networks}

We can consider without loss of generality that the structure of a network is as described in Section~\ref{sec:FirstVariation}. Indeed
boundary conditions for a
other possible topologies can be easily deduces from this case.

The possible boundary condition at $x=1$ are nothing but what we just described for
a single curve with constraints at the endpoints.
Thus we focus on the junction
$O=\gamma^1(0)=\ldots=\gamma^N(0)$. We can distinguish two sub cases

\medskip

\noindent \textit{Neumann (so-called natural or Navier) boundary conditions}

In this case we only require the network not to change its topology in a first variation. Letting first $\psi^i(0)=0$ for any $i$, it remains the boundary term
\begin{equation*}
\sum_{i=1}^{N}  \langle \boldsymbol{\kappa}^{i}(0),
 \partial_s \psi^{i}(0)\rangle  =0,
\end{equation*}
where the test functions $\psi^i$ appear differentiated. 
We can choose $\partial_s\psi^{1}(0)=\nu^1(0)$ and $\partial_s\psi^i(0)=0$
for every $i\in\{2,\ldots,N\}$.
This implies $\boldsymbol{\kappa}^1(0)=0$.
Then, because of the arbitrariness of the choice of $i$ we obtain:
\begin{equation}\label{curvaturenulle}
\boldsymbol{\kappa}^i(0)=0\,,
\end{equation}
for any $i\in\{1, \ldots, N\}$.

It remains to consider the last term of~\eqref{boundarypart}.
Taking into account the just obtained condition~\eqref{curvaturenulle}, by arbitrariness of $\psi^1(0)=...=\psi^N(0)$ it
reads
\begin{equation*}
\sum_{i={1}}^{N} \left(-2\partial^\perp_s \boldsymbol{\kappa}^{i}(0)
+\mu\tau^i(0)\right)=0\,,
\end{equation*}

\medskip

\noindent \textit{Dirichlet (so-called clamped) boundary conditions}

As discussed above, also in the case of a network we can impose a condition on the tangent of the curves at their endpoints. As we saw in the clamped curve case, from the variational point of 
this extra condition  involves the unit tangent vectors. Then an extra property 
on $\partial_{s}\psi^{i}$ is expected. 

At the junction 
we require the following $(N-1)$ conditions:
\begin{equation*}
\left\langle \tau^{i_1}(0),\tau^{i_2}(0)\right\rangle=c^{1,2}\,, 
\ldots\,,\, 
\left\langle  \tau^{i_{N-1}}(0),\tau^{i_N}(0)\right\rangle=c^{N-1,N}\,,
\end{equation*}
that is, the angles between tangent vectors are fixed.
We need that also the variation $\mathcal{N}_{\varepsilon}$ satisfies the same
\begin{equation*}
\left\langle \tau_\varepsilon^{i_1}(0),\tau_\varepsilon^{i_2}(0)\right\rangle=c^{1,2}\,, 
\ldots\,,\, 
\left\langle  \tau_\varepsilon^{i_{N-1}}(0),\tau_\varepsilon^{i_N}(y_N)\right\rangle=c^{N-1,N}\,,
\end{equation*}
for any $|\e|$ small enough. This means that for every $i,j\in\{1,\ldots, N\}$ we need that 
\begin{equation*}
\frac{d}{d\varepsilon}\left\langle\tau_\varepsilon^{i}(0),\tau_\varepsilon^{j}(0)\right\rangle=0\,,
\end{equation*}
that implies
\begin{align*}
0&=\frac{{d}}{{d}\varepsilon}\left\langle\tau_\varepsilon^{i}(0),\tau_\varepsilon^{j}(0)\right\rangle
\Big\vert_{\varepsilon=0}=
\left\langle
\partial_{s}^\perp\psi^i(0), \tau^j(0)\right\rangle
+\left\langle \tau^i(0),\partial_{s}^\perp\psi^j(0)\right\rangle\\
&=({\psi}_{s}^{i})^\perp(0)\left\langle\nu^i(0),\tau^j(0) \right\rangle
+({\psi}_{s}^{j})^\perp(0)\left\langle \tau^i(0), \nu^j(0)\right\rangle\\
&=({\psi}_{s}^{i})^\perp(0)\left\langle\nu^i(0),\tau^j(0) \right\rangle
-({\psi}_{s}^{j})^\perp(0)\left\langle\nu^i(0),\tau^j(0) \right\rangle\,.
\end{align*}
where we used the notation $\left(\psi^i_{s}\right)^\perp:=\left\langle \partial_{s}\psi^i,\nu^i\right\rangle$.
So we impose 
\begin{equation}\label{proprieta2}
({\psi}_{s}^{1})^\perp(0)=\ldots=({\psi}_{s}^{N})^\perp(0)\,.
\end{equation}
Then the first boundary term of~\eqref{boundarypart} reduces to
\begin{equation*}
2\langle({\psi}_{s}^{1})^\perp(0),\sum_{i=1}^{N}
k^i(0)\rangle\,.
\end{equation*}
Hence we find the following boundary conditions:
\begin{align*}
\sum_{i=1}^{N}\ k^i(0)=0\,,\qquad \sum_{i=1}^{N}-2\partial^\perp_s \boldsymbol{\kappa}^{i}(0)
- \vert\boldsymbol{\kappa}^i(0)\vert^2\tau^i(0)+\mu^i\tau^i(0)=0\,.
\end{align*}

\medskip

In the end, whenever the network is composed of $N$ curves
we have a system of $N$ equations (not coupled)
that are quasilinear and of fourth order in the parametrizations 
of the curves with coupled boundary conditions. 

\medskip

We now need to briefly introduce the H\"{o}lder spaces that will appear in the 
definition of the flow. 
 
Let $N\in\mathbb{N}$, 
consider a network $\mathcal{N}$ composed of $N$ curves
with endpoints of order one fixed in the plane and the
curves that meet at junctions of different order $m\in\mathbb{N}_{\geq 2}$.
As we have already said each curve of  $\mathcal{N}$ 
is parametrized by $\gamma^i:[0,1]\to\mathbb{R}^2$.
Let $\alpha\in (0,1)$.
We denote $\gamma:=(\gamma^1,\ldots,\gamma^N)\in(\mathbb{R}^2)^N$
and
\begin{equation*}
\mathrm{I}_N:=C^{4+\alpha}\left([0,1];(\mathbb{R}^2)^N\right)\,.
\end{equation*}

We will make and extensive use of 
parabolic H\"{o}lder spaces (see also~\cite[\S 11, \S 13]{solonnikov2}).
For $k\in \{0,1,2,3,4\}$, $\alpha\in (0,1)$ the parabolic H\"older space
$$
C^{\frac{k+\alpha}{4}, k+\alpha}([0,T]\times[0,1])
$$
is the space 
of all functions $u:[0,T]\times [0,1]\to\mathbb{R}$ that have continuous derivatives $\partial_t^i\partial_x^ju$ where $i,j\in\mathbb{N}$ are such that $4i+j\leq k$ for which the norm
\begin{equation*}
\left\lVert u\right\rVert_{C^{\frac{k+\alpha}{4},k+\alpha}}:=\sum_{4i+j=0}^k\left\lVert\partial_t^i\partial_x^ju\right\rVert_\infty
+\sum_{4i+j=k}\left[\partial_t^i\partial_x^ju\right]_{0,\alpha}+\sum_{0<k+\alpha-4i-j<4}\left[\partial_t^i\partial_x^ju\right]_{\frac{k+\alpha-4i-j}{4},0}
\end{equation*}
is finite. 
We recall that for a function $u:[0,T]\times [0,1]\to\mathbb{R}$, for $\rho\in (0,1)$ the semi--norms $[ u]_{\rho,0}$ and $[ u]_{0,\rho}$ are defined as
$$
[ u]_{\rho,0}:=\sup_{(t,x), (\tau,x)}\frac{\vert u(t,x)-u(\tau,x)\vert}{\vert t-\tau\vert^\rho}\,,
$$
and
$$
[ u]_{0,\rho}:=\sup_{(t,x), (t,y)}\frac{\vert u(t,x)-u(t,y)\vert}{\vert x-y\vert^\rho}\,.
$$
Moreover
the space $C^{\frac{\alpha}{4},\alpha}\left([0,T]\times[0,1]\right)$ is equal to the space
$$
C^{\frac{\alpha}{4}}\left([0,T];C^0([0,1])\right)\cap C^0\left([0,T];C^\alpha([0,1])\right)\,,
$$
with equivalent norms.

We also define the spaces
$C^{\frac{k+\alpha}{4}, k+\alpha}([0,T]\times\{0,1\}, \mathbb{R}^m)$
to be 
$C^{\frac{k+\alpha}{4}}([0,T], \mathbb{R}^{2m})$
via the isomorphism
$f\mapsto (f(t,0),f(t,1))^t$.

\begin{dfnz}[Elastic flow]\label{Def:elasticflow}
Let $N\in\mathbb{N}$ and let 
$\mathcal{N}_0$  be an 
initial network composed of $N$ curves parametrized
by $\gamma_0=(\gamma_0^1,\ldots,\gamma_0^N)\in \mathrm{I}_N$,
 (possibly) with 
endpoints of order one  and (possibly) with 
curves that meet at junctions of different order $m\in\mathbb{N}_{\geq 2}$.
Then a time dependent family of networks 
$\mathcal{N}(t)_{t\in [0,T]}$ is a solution to the \emph{elastic flow}
in the time interval $[0,T]$ with $T>0$
if there exists a parametrization 
$$
\gamma(t,x)=\left(\gamma^1(t,x), \ldots,\gamma^N(t,x)\right)\in 
C^{\frac{4+\alpha}{4}, 4+\alpha}\left([0,T]\times [0,1];(\mathbb{R}^2)^N\right)\,,
$$
with $\gamma^i$ regular,
and such that for every $t\in [0,T], x\in [0,1]$ and $i\in\{1,\ldots,N\}$
the system 
\begin{equation}\label{evolutionlaw}
\begin{cases}
(\partial_t\gamma^i)^\perp=\left(-2\partial_s^2k^i-(k^i)^3+k^i\right)\nu^i\\
\gamma^i(0,x)=\gamma_0^i(x)
\end{cases}
\end{equation}
is satisfied.
Moreover the system is coupled with suitable boundary conditions as follows, corresponding to the possible cases of boundary conditions discussed in the formulation of the first variation.

\begin{itemize}
\item If $N=1$ and the curve $\gamma_0$ is closed
we require $\gamma(t,x)$ to be closed and we impose periodic boundary conditions.

\item If $N=1$ and the curve  $\gamma_0$ is not closed with
$\gamma_0(0)=P\in\mathbb{R}^2$, $\gamma_0(1)=Q\in\mathbb{R}^2$ 
and we want to impose natural boundary conditions we require
\begin{equation}\label{navier-endpoint}
\begin{cases}
\gamma(t,0)=P\\
\gamma(t,1)=Q\\
\boldsymbol{\kappa}(t,0)=\boldsymbol{\kappa}(t,1)=0\,.
\end{cases}
\end{equation}

\item If $N=1$ and the curve  $\gamma_0$ is not closed with
$\gamma_0(0)=P\in\mathbb{R}^2$, $\gamma_0(1)=Q\in\mathbb{R}^2$
and we want to impose clamped boundary conditions, we require
\begin{equation}\label{clamped-endpoint}
\begin{cases}
\gamma(t,0)=P\\
\gamma(t,1)=Q\\
\tau(t,0)=\tau_0\\
\tau(t,1)=\tau_1\,.
\end{cases}
\end{equation}

\item  If $N$ is arbitrary and $\mathcal{N}_0$ has one multipoint
$$
\gamma_0^{i_1}(y_1)=\ldots=\gamma_0^{i_m}(y_m)\,,
$$
with $(i_1,y_1),\ldots,(i_m,y_m)\in \{1,\ldots,N\}\times \{0,1\}$
and we want to impose natural boundary conditions, 
for every $j\in\{1,\ldots,m\}$ we require 
\begin{equation}\label{flussonatural}
\begin{cases}
\kappa^{i_j}(t,y)=0\\
\sum_{j=1}^{m} \left(-2\partial^\perp_s \boldsymbol{\kappa}^{i_j}+\mu^{i_j}\tau^{i_j}\right)(t,y_j)=0\,.
\end{cases}
\end{equation}

\item If $N$ is arbitrary and $\mathcal{N}_0$ has one multipoint
$$
\gamma_0^{i_1}(y_1)=\ldots=\gamma_0^{i_m}(y_m)\,,
$$
with $(i_1,y_1),\ldots,(i_m,y_m)\in \{1,\ldots,N\}\times \{0,1\}$
where  we want to impose clamped boundary conditions, we require 
\begin{equation}\label{flussoclamped}
\begin{cases}
\left\langle \tau^{i_1}(y_1),\tau^{i_2}(y_2)\right\rangle=c^{1,2}\\
\ldots\\
\left\langle \tau^{i_{m-1}}(y_{m-1}),\tau^{i_m}(y_m)\right\rangle=c^{m-1,m}\\
    \sum_{j=1}^m k^{i_j}=0\\
	\sum_{j=1}^m \left(-2\partial^\perp_s \boldsymbol{\kappa}^{i_j}- |\boldsymbol{\kappa}^{i_j}(y_i)|^2\tau^{i_j}(y_i)
	+\mu^{i_j}\tau^{i_j}(y_i)\right)=0\,.
\end{cases}
\end{equation}
\end{itemize}
\end{dfnz}

Clearly in the case of network with several junctions and 
endpoints of order one fixed in the plane,
one has to impose different boundary conditions 
(chosen among~\eqref{navier-endpoint},~\eqref{clamped-endpoint},
~\eqref{flussonatural} and~\eqref{flussoclamped}) at each junctions and endpoint.

We give a name to the boundary conditions appearing in the 
definition of the flow. 
When there is a multipoint 
$$
\gamma_0^{i_1}(y_1)=\ldots=\gamma_0^{i_m}(y_m)\,,
$$
with $(i_1,y_1),\ldots,(i_m,y_m)\in \{1,\ldots,N\}\times \{0,1\}$
we shortly refer to:
\begin{itemize}
\item $\gamma_0^{i_1}(t,y_1)=\ldots=\gamma_0^{i_m}(t,y_m)$
as \emph{concurrency condition};
\item $\left\langle \tau^{i_1}(y_1),\tau^{i_2}(y_2)\right\rangle=c^{1,2}\,,
\ldots,
\left\langle \tau^{i_{m-1}}(y_{m-1}),\tau^{i_m}(y_m)\right\rangle=c^{m-1,m}$ as \emph{angle conditions};
\item either $k^{i_j}(t,y)=0$ for every $j\in\{1,\ldots,m\}$
or $\sum_{j=1}^m k^{i_j}=0$ as \emph{curvature conditions};
\item $\sum_{j=1}^m \left(-2\partial^\perp_s \boldsymbol{\kappa}^{i_j}- |\boldsymbol{\kappa}^{i_j}(y_i)|^2\tau^{i_j}(y_i)
	+\mu^{i_j}\tau^{i_j}(y_i)\right)=0$ as \emph{third order condition}.
	
When we have an endpoint of order one 
	we refer to the condition involving the tangent
	vector as \emph{angle condition} and the curvature as \emph{curvature condition}.
\end{itemize}

\begin{rem}
In system~\eqref{evolutionlaw} only the normal component of the velocity is prescribed. This does not mean that the tangential velocity
is necessary zero. We can equivalently write the motion equations as
\begin{equation*}
    \partial_t\gamma^i=V^i\nu^i+T^i\tau^i\,,
\end{equation*}
where $V^i=-2\partial_s^2k^i-(k^i)^3+k^i$ and $T^i$ are some at least
continuous functions.
In the case of a single closed curve or a single curve with
fixed endpoint we can  impose $T\equiv 0$ (see Section~\ref{Sec:invariance}).
\end{rem}

\begin{dfnz}[Admissible initial network]\label{Def:admissible-initial-net}
A network $\mathcal{N}_0$ of $N$ regular curves
parametrized by $\gamma=(\gamma^1, \ldots,\gamma^N)$, 
$\gamma^i:[0,1]\to\mathbb{R}^2$ with $i\in\{1, \ldots, N\}$
possibly with $\ell$ 
endpoints of order one $\{\gamma^j(y_j)\}$
for some  $(j,y_j)\in \{1, \ldots, N\}\times \{0,1\}$,
and possibly with curves that meet 
at $k$ different junctions $\{O^p\}$ of order $m\in\N_{\ge2}$
at $O^p=\gamma^{p_1}(y_1)=\ldots=\gamma^{p_m}(p_m)$ 
for some $(p_i,y_i)\in \{1, \ldots, N\}\times \{0,1\}, p\in\{1, \ldots,k\}$
forming angles $\alpha^{p_i,p_{i+1}}$ between $\nu^{p_i}$ and $\nu^{p_{i+1}}$ is an admissible initial network 
if 
\begin{itemize}
    \item[i)] the parametrization $\gamma$ belongs to $\mathrm{I}_N$;
    \item[ii)] $\mathcal{N}_0$ satisfies all the boundary condition
imposed in the system: concurrency, angle, curvature and third order 
conditions;
     \item[iii)] at each endpoint $\gamma^j(y_j)$ of order one  it  holds
$$
2\partial_s^2 k^j(y_j)+(k^j)^3(y_j)-\mu^i k^j(y_j)=0\,;
$$
 \item[iv)] the initial datum fulfills the non--degeneracy condition:
at each junction
$$
\mathrm{span}\{\nu^{p_1}, \ldots,\nu_0^{p_m}\}=\mathbb{R}^2\,;
$$   
\item[v)] at each junction $\gamma^{p_1}(y_1)=\ldots=\gamma^{p_m}(y_m)$ 
where at least three curves concur, 
consider two consecutive unit normal vectors 
$\nu^{p_i}(y_i)$ and $\nu^{p_k}(y_k)$
such that
$\mathrm{span}\{\nu^{p_i}(y_i),\nu^{p_k}(y_k)\}=\R^2$.
Then for every $j\in\{1,\ldots,m\}$, $j\neq i$, $j\neq k$
we require
$$
\sin\theta^i V^i(y_i)+\sin\theta^k V^k(y_k)
+\sin\theta^j V^j(y_j)=0\,,
$$
where $\theta^i$ is the angle between $\nu^{p_k}(y_k)$
and $\nu^{p_j}(y_j)$, 
$\theta^k$ between $\nu^{p_j}(y_j)$
and $\nu^{p_i}(y_i)$
and $\theta^{j}$  between $\nu^{p_i}(y_i)$
and $\nu^{p_k}(y_k)$.
\end{itemize}
\end{dfnz}

\begin{rem}
The conditions ii)-iii)--v)
on the initial network are the 
so--called \emph{compatibility conditions}.
Together with the non--degeneracy condition, 
these conditions concern the boundary of the network,
and so they are not required in the case of one single closed curve.
\end{rem}

\begin{rem}
We refer to the conditions iii) and v) as 
fourth order compatibility conditions. 
We explain here how one derives condition v)
in the case of a junction 
$\gamma^1(0)=\ldots=\gamma^m(0)$.
Differentiating in time the concurrency condition
we get $\partial_t\gamma^1(0)
=\ldots=\partial_t\gamma^m(0)$, or, in terms
of the normal and tangential velocities
$V^1(0)\nu^1(0)+T^1(0)\tau^1(0)=\ldots
=V^m(0)\nu^m(0)+T^m(0)\tau^m(0)$.

Without loss of generality we suppose that the
concurring curves are labeled in a counterclockwise sense
and that $\mathrm{span}\{\nu^1(0),\nu^2(0)\}=\mathbb{R}^2$.
Then for every $j\in\{3,\ldots,m\}$ we have
$$
\sin\theta^1\nu^1(0)+\sin\theta^2\nu^2(0)
+\sin\theta^j\nu^j(0)=0\,,
$$
where $\theta^1$ is the angle between $\nu^2(0)$
and $\nu^j(0)$, 
$\theta^2$ between $\nu^j(0)$
and $\nu^1(0)$
and $\theta^j$  between $\nu^1(0)$
and $\nu^2(0)$.
Then 
\begin{align*}
\sin\theta^1V^1(0)&=
\left\langle V^1(0)\nu^1(0)+T^1(0)\tau^1(0),
\sin\theta^1\nu^1(0)\right\rangle\\
&=
\left\langle V^2(0)\nu^2(0)+T^2(0)\tau^2(0),
-\sin\theta^2\nu^2(0)-\sin\theta^j\nu^j(0)\right\rangle\\
&=
-\sin\theta^2 V^2(0)+
\left\langle V^2(0)\nu^2(0)+T^2(0)\tau^2(0),
-\sin\theta^j\nu^j(0)\right\rangle\\
&=
-\sin\theta^2 V^2(0)+
\left\langle V^j(0)\nu^j(0)+T^j(0)\tau^j(0),
-\sin\theta^j\nu^j(0)\right\rangle\\
&=-\sin\theta^2 V^2(0)-\sin\theta^j V^j(0)\,.
\end{align*}
Hence for every $j\in\{3,\ldots,m\}$ 
we obtained 
$\sin\theta^1V^1(0)+\sin\theta^2 V^2(0)
+\sin\theta^j V^j(0)=0$.
\end{rem}

\begin{rem}
To prove existence of solutions of class $C^{\frac{4+\alpha}{4}, 4+\alpha}$ to the elastic flow 
of networks it is necessary 
to require the fourth order compatibility 
conditions for the initial datum.
This conditions may sound not very natural because it 
does not appear among the boundary conditions imposed 
in the system.
It is actually possible not to ask for it by 
defining the elastic flow of networks
in a Sobolev setting. 
The price that we have to pay is that 
in such a case
a solution will be slightly less regular 
(see~\cite{GaMePl2, menzelthesis} for details). 
On the opposite side, if we want a smooth solution till $t=0$
one has to impose many more conditions.
These properties, the compatibility conditions of \emph{any order},  are derived 
repeatedly differentiating in time the boundary conditions 
and using the motion equation
to substitute time derivatives with space derivatives
(see~\cite{DaChPo19, DaChPo20}).
\end{rem}

\subsection{Invariance under reparametrization}\label{Sec:invariance}

It is very important to remark the consequences of the invariance under reparametrization of the energy functional on the resulting gradient flow. These effects actually occur whenever the starting energy is \emph{geometric}, i.e., invariant under reparamentrization.
To be more precise, 
let us say that the time dependent 
family of closed curves 
parametrized by 
$\gamma:[0,T]\times\mathrm{S}^1\to \mathbb{R}^2$ is a smooth solution to the elastic flow
	\begin{equation}\label{eq:CauchyCanonica}
	\begin{cases}
 \partial_t\gamma(t,x) = V_\gamma(t,x)\nu_\gamma(t,x) \,,\\
	\gamma(0,\cdot)=\gamma_0(\cdot)\,,
	\end{cases}
	\end{equation}
and the driving velocity $\partial_t\gamma$ is normal along $\gamma$.
If $\chi:[0,T]\times \mathrm{S}^1\to \mathrm{S}^1$ with $\chi(t,0)=0$ and $\chi(t,1)=1$
is a smooth one--parameter family of diffeomorphism
and  $\sigma(t,x)\coloneqq \gamma(t,\chi(t,x))$, then  it is immediate to check that $\sigma$ solves
	\[
	\begin{cases}
	\partial_t \sigma (t,x) = V_\sigma(t,x)\nu_\sigma(t,x) + W(t,x)\tau_\sigma(t,x),\\
	\sigma(0,\cdot) = \gamma_0(\chi(0,\cdot))\,,
	\end{cases}
	\]
and $W$ can be computed explicitly in terms of $\chi$ and $\gamma$. More importantly, one has that $V_\sigma(t,x)\nu_\sigma(t,x) =  V_\gamma(t,\chi(t,x))\nu_\gamma(t,\chi(t,x))$.	
Since  $W(t,x)\tau_\sigma(t,x)$ is a tangential term,
$\sigma$ itself is a solution to the elastic flow. Indeed its normal driving velocity $\partial_t^\perp\sigma$ is the one defining the elastic flow on $\sigma$.
This is the reason why the definition of the elastic flow is given in terms of the normal velocity of the evolution only.

\medskip

In complete analogy, if $\beta:[0,T)\times  \mathrm{S}^1 \to \mathbb{R}^2$ is given, it is smooth and solves
	\[
	\begin{cases}
	\partial_t \beta(t,x)= V_\beta(t,x)\nu_\beta(t,x) + w(t,x)\tau_\beta(t,x),\\
	\beta(0,\cdot)= \gamma_0 (\chi_0(\cdot))\,,
	\end{cases}
	\]
	where $\chi_0: \mathrm{S}^1\to  \mathrm{S}^1$ is a diffeomorphism, then letting $\psi:[0,T]\times  \mathrm{S}^1\to  \mathrm{S}^1$ be the smooth solution of
	\[
	\begin{cases}
	\partial_t\psi(t,x) = - |(\partial_x\beta)(t,\psi(t,x))|^{-1} w(t,\psi(t,x)),\\
	\psi(0,\cdot)=\chi_0^{-1}(\cdot)\,,
	\end{cases}
	\]
	it immediately follows that $\gamma(t,x)\coloneqq \beta(t,\psi(t,x))$ solves~\eqref{eq:CauchyCanonica}.

\medskip

Something similar holds true also in the general case of networks.
First of all it is easy to check the all possible boundary conditions
are invariant under reparametrizations (both at the multiple junctions and at the endpoints of order one). 
Concerning the velocity, we cannot impose the tangential velocity
to be zero as in~\eqref{eq:CauchyCanonica}, but it remains
true that if a time dependent family of networks
parametrized by $\gamma=(\gamma^1,\ldots,\gamma^N)$
with $\gamma^i:[0,T]\times [0,1]\to\mathbb{R}^2$ is a solution to the 
elastic flow, then $\sigma=(\sigma^1,\ldots,\sigma^N)$
defined by $\sigma^i(t,x)=\gamma^i(t, \chi^i(t,x))$
with $\chi^i:[0,T]\times [0,1]\to [0,1]$ a time dependent family
of diffeomorphisms such that $\sigma(t,0)=0$ and $\sigma(t,1)=1$
(together with suitable conditions on $\partial_x\sigma(t,0)$, $\partial_{x}^2\sigma(t,0)$ and so on)
is still a solution to the elastic flow of networks.
Indeed the velocity of $\gamma^i$ and $\sigma^i$
differs only by a tangential component.

\begin{rem}\label{tang-in-funz-norm}
We want to stress that at the junctions the tangential is
velocity determined by the normal velocity.

Consider a junction of order $m$
$$
\gamma^1(t,0)=\ldots=\gamma^m(t,0)\,.
$$
Differentiating in time yields  $\partial_t\gamma^1(t,0)=\ldots
=\partial_t\gamma^m(t,0)$ that, in terms of the normal and tangential motion $V$ and $T$ reads as
\begin{equation*}
V^j\nu^j+T^j\tau^j=V^{j+1}\nu^{j+1}+T^{j+1}\tau^{j+1}\,,
\end{equation*}
where $j\in\{1, \ldots, m\}$ with $m+1:=1$ and the argument $(t,0)$ is omitted from now on.
Testing these identities with the unit tangent 
vectors $\tau^j$ leads to the system:
\begin{footnotesize}
	\begin{equation*}
	\begin{pmatrix}
	 1 & -\cos\alpha^{1,2} & 0 & 0 &\ldots & 0\\
	 0 & 1 & -\cos\alpha^{2,3} & 0 & \ldots &  0 \\
	 0 & 0 & 1 & -\cos\alpha^{3,4} &  \ldots &  0 \\
	 \vdots & \vdots &  \vdots &  \vdots &  \vdots &  \vdots \\
	 0 & 0 & 0 & \ldots & 1 & -\cos\alpha^{m-1,m} \\
	 -\cos\alpha^{m,1}  & 0 & 0 & \ldots & 0 & 1 
	\end{pmatrix}
	\begin{pmatrix}
	T^1\\ T^2 \\ T^3 \\ \vdots \\ T^{m-1} \\ T^m
	\end{pmatrix}
	=\begin{pmatrix}
	-\sin\alpha^{1,2} V^2 \\ 
	-\sin\alpha^{2,3} V^3 \\ 
	-\sin\alpha^{3,4} V^4 \\ 
     \vdots \\
	-\sin\alpha^{m-1,m} V^m \\ 
	-\sin\alpha^{m,1} V^1
	\end{pmatrix} \,.
	\end{equation*}
\end{footnotesize}
	We call $M$ the $m\times m$--matrix of the coefficients and $R_1,\ldots, R_m$ its rows.

It is easy to see that
\begin{equation*}
\det(M)=1- \cos\alpha^{m,1}\cos\alpha^{1,2}
	\ldots\cos\alpha^{m-2,m-1}\cos\alpha^{m-1,m}\,,
\end{equation*}
that is different from zero till the non--degeneracy condition is satisfied.
Then the system has a unique solution and so
each  $T^i(t)$ can be expressed as a linear combination of $V^1(t),\ldots, V^m(t)$. 
\end{rem}

\begin{rem}
The previous observations clarify the fact that 
the only meaningful notion of uniqueness for a geometric flow like the elastic one is thus uniqueness up to
reparametrization.
\end{rem}

\medskip

We can actually take advantage of the invariance by reparametrization of 
the problem to reduce system~\eqref{evolutionlaw} to a 
non--degenerate system of quasilinear PDEs.
Consider the flow of one curve $\gamma$. As we said before, the normal velocity is a geometric quantity,
namely $\partial_t\gamma^\perp= V\nu=-2\partial_s^2 k\nu-k^{3}\nu+\mu k\nu$.
Computing this quantity in terms of the parametrization $\gamma$ we get
\begin{align}\label{Apart}
&-V\nu=2\partial_s^2 k\nu+k^{3}\nu-\mu k\nu\nonumber \\
&=2\frac{\partial_x^4 \gamma}{\left|\partial_x\gamma\right|^{4}}
-12\frac{\partial_x^3 \gamma\left\langle \partial_x^2 \gamma,\partial_x\gamma\right\rangle }{\left|\partial_x\gamma\right|^{6}}
-5\frac{\partial_x^2 \gamma\left|\partial_x^2 \gamma\right|^{2}}{\left|\partial_x\gamma\right|^{6}}
-8\frac{\partial_x^2 \gamma\left\langle \partial_x^3 \gamma,\partial_x\gamma\right\rangle }
{\left|\partial_x\gamma\right|^{6}}
+35\frac{\partial_x^2 \gamma\left\langle \partial_x^2 \gamma,\partial_x\gamma\right\rangle ^{2}}
{\left|\partial_x\gamma\right|^{8}}\nonumber\\
&\indent +\left\langle -2\frac{\partial_x^4 \gamma}{\left|\partial_x\gamma\right|^{4}}+12\frac{\partial_x^3 \gamma\left\langle 
\partial_x^2 \gamma,\partial_x\gamma\right\rangle }{\left|\partial_x\gamma\right|^{6}}+5\frac{\partial_x^2 \gamma\left|
\partial_x^2 \gamma\right|^{2}}{\left|\partial_x\gamma\right|^{6}}+8\frac{\partial_x^2 \gamma\left\langle 
\partial_x^3 \gamma,\partial_x\gamma\right\rangle }{\left|\partial_x\gamma\right|
^{6}}-35\frac{\partial_x^2 \gamma\left\langle \partial_x^2 \gamma,\partial_x\gamma\right\rangle ^{2}}{\left|\partial_x\gamma\right|^{8}},\tau\right\rangle 
\tau\nonumber\\
&\indent -\mu\frac{\partial_x^2 \gamma}{\left|\partial_x\gamma\right|^{2}}
+\left\langle \mu\frac{\partial_x^2 \gamma}{\left|\partial_x\gamma\right|^{2}},
\tau\right\rangle \tau\,.
\end{align}

We define 
\begin{align}\label{Tang}
\overline{T}&:=\left\langle -2\frac{\partial_x^4 \gamma}{\left|\partial_x\gamma\right|^{4}}
+12\frac{\partial_x^3 \gamma\left\langle \partial_x^2 \gamma,\partial_x\gamma\right\rangle }
{\left|\partial_x\gamma\right|^{6}}
+5\frac{\partial_x^2 \gamma\left|\partial_x^2 \gamma\right|^{2}}{\left|\partial_x\gamma\right|^{6}}
\right. \nonumber \\
&\indent \left. +8\frac{\partial_x^2 \gamma\left\langle \partial_x^3 \gamma,\partial_x\gamma\right\rangle }
{\left|\partial_x\gamma\right|^{6}}
-35\frac{\partial_x^2 \gamma\left\langle \partial_x^2 \gamma,\partial_x\gamma\right\rangle ^{2}}
{\left|\partial_x\gamma\right|^{8}}
+\mu\frac{\partial_x^2 \gamma}{\left|\partial_x\gamma\right|^{2}},\tau\right\rangle \,.
\end{align}

We can insert this choice of the tangential component of the velocity in
the motion equation, which becomes
\begin{align*}
 \partial_t\gamma&=V\nu+\overline{T}\tau 
    =-\frac{2}{\vert\partial_x\gamma\vert^4}\partial_x^4 \gamma
    +f(\partial_x\gamma,\partial_x^2 \gamma,\partial_x^3 \gamma)\\
   & =-2\frac{\partial_x^4 \gamma}{\left|\partial_x\gamma\right|^{4}}
+12\frac{\partial_x^3 \gamma\left\langle \partial_x^2 \gamma,\partial_x\gamma\right\rangle }{\left|\partial_x\gamma\right|^{6}}
+5\frac{\partial_x^2 \gamma\left|\partial_x^2 \gamma\right|^{2}}{\left|\partial_x\gamma\right|^{6}}
+8\frac{\partial_x^2 \gamma\left\langle \partial_x^3 \gamma,\partial_x\gamma\right\rangle }
{\left|\partial_x\gamma\right|^{6}}
-35\frac{\partial_x^2 \gamma\left\langle \partial_x^2 \gamma,\partial_x\gamma\right\rangle ^{2}}
{\left|\partial_x\gamma\right|^{8}}
+\mu\frac{\partial_x^2 \gamma}{\left|\partial_x\gamma\right|^{2}}\,.
\end{align*}

Considering now the boundary conditions:
up to reparametrization
the clamped condition $\tau(t,0)=\tau_0$
can be reformulated as $\partial_x\gamma(t,0)=\tau_0$ and
the curvature condition $k=\boldsymbol{\kappa}=0$
as $\partial_x^2 \gamma(t,0)=0$. We can then extend this discussion to the flow of general networks, in order to define the so-called special flow.

\begin{dfnz}[Admissible initial parametrization]\label{DEf:admissible-initial-para}
We say that $\varphi_0=(\varphi_0^1, \ldots, \varphi_0^N)$
is an admissible parametrization for the special flow
if 
\begin{itemize}
    \item the functions $\varphi_0^i$ are of class $C^{4+\alpha}([0,1];\mathbb{R}^2)$;
    \item $\varphi_0=(\varphi_0^1, \ldots, \varphi_0^N)$ satisfies 
    all the boundary conditons imposed in the system;
    \item at each endpoint of order one it holds $V^i=0$ and 
    $\overline{T}^i=0$ for any $i$;
    \item at each junction it holds 
    \begin{equation*}
        V^i\nu^i+\overline{T}^i\tau^i=V^j\nu^j+\overline{T}^j\tau^j
        \end{equation*}
        for any $i,j$;
    \item at each junction the non--degeneracy condition is satisfied;
\end{itemize}
where $\overline{T}^i$ is defined as in~\eqref{Tang} for any $i$ and $j$.
\end{dfnz}

\begin{dfnz}[Special flow]\label{def:SpecialFlow}
Let $N\in\mathbb{N}$ and let $\varphi_0=(\varphi_0^1,\ldots,\varphi_0^N)$
be an admissible initial parametrization in the sense of Definition~\ref{DEf:admissible-initial-para}
(possibly) with endpoints of order one and (possibly)
with junctions of different orders $m\in\mathbb{N}_{\geq 2}$.
Then a time dependent family of parametrizations
$\varphi_{t\in [0,T]}$, $\varphi=(\varphi^1,\ldots,\varphi^N)$ is a solution to the special flow
if and only if for every $i\in\{1,\ldots,N\}$ the functions
$\varphi^i$ are of class 
$C^{\frac{4+\alpha}{4}, 4+\alpha}([0,T]\times [0,1];\mathbb{R}^2)$,
for every $(t,x)\in [0,T]\times [0,1]$ it holds
$\partial_x\varphi(x)\neq 0$ and the system 
\begin{equation}\label{specialflow}
    \begin{cases}
  \partial_t\varphi^i=V^i\nu^i+\overline{T}^i\tau^i\\
    \varphi^i(0,x)=\varphi_0(x)
    \end{cases}
\end{equation}
is satisfied, where $\overline{T}^i$ is defined as in~\eqref{Tang} for any $i$.
Moreover the following boundary conditions are imposed:
\begin{itemize}
    \item if $N=1$ and $\varphi_0$ is a closed curve, then we impose periodic boundary conditions;
    \item if $N=1$ and $\varphi_0(0)=P,\varphi_0(1)=Q$, we can require
    either 
\begin{equation*}
\begin{cases}
\varphi^1(t,0)=P\\
\varphi^1(t,1)=Q\\
\partial_x^2\varphi(t,0)=\partial_x^2\varphi^1(t,1)=0\,,
\end{cases}
\end{equation*}
or 
\begin{equation*}
\begin{cases}
\varphi^1(t,0)=P\\
\varphi^1(t,1)=Q\\
\partial_x\varphi^1(t,0)=\tau_0\\
\partial_x\varphi^1(t,1)=\tau_1\,.
\end{cases}
\end{equation*}
\item  if $N$ is arbitrary and $\mathcal{N}_0$ has one multipoint
$$
\gamma_0^{i_1}(y_1)=\ldots=\gamma_0^{i_m}(y_m)\,,
$$
with $(i_1,y_1),\ldots,(i_m,y_m)\in \{1,\ldots,N\}\times \{0,1\}$
we can impose either 
\begin{equation*}
\begin{cases}
\partial_x^2\varphi^{i_j}(t,y)=0\quad\text{for every}\, j\in\{1,\ldots,m\}\\
\sum_{j=1}^{m} \left(-2\partial^\perp_s \boldsymbol{\kappa}^{i_j}+\mu^{i_j}\tau^{i_j}\right)(t,y_j)=0\,,
\end{cases}
\end{equation*}
or
\begin{equation*}
\begin{cases}
\left\langle \tau^{i_1}(y_1),\tau^{i_2}(y_2)\right\rangle=c^{1,2}\\
\ldots\\
\left\langle \tau^{i_{m-1}}(y_{m-1}),\tau^{i_m}(y_m)\right\rangle=c^{m-1,m}\\
    \sum_{j=1}^m k^{i_j}=0\\
   \left\langle \partial_x^2\varphi^{i_j}(t,y),\partial_x\varphi^{i_j}(t,y)\right\rangle=0\quad\text{for every}\, j\in\{1,\ldots,m\}\\
	\sum_{j=1}^m  \left(-2\partial^\perp_s \boldsymbol{\kappa}^{i_j}- |\boldsymbol{\kappa}^{i_j}(y_j)|^2\tau^{i_j}(y_j)
	+\mu^{i_j}\tau^{i_j}(y_{i_j})\right)=0\,.
\end{cases}
\end{equation*}
\end{itemize}
\end{dfnz}

\begin{lemma}\label{lem:SpecialImplicaElastic}
Let $\varphi_0=(\varphi_0^1, \ldots,\varphi_0^N)$ be an admissible initial parametrization and 
$\varphi_{t\in [0,T]}$,  $\varphi=(\varphi^1, \ldots,\varphi^N)$ be a solution 
to the special flow.
Then $\mathcal{N}_t=\cup_{i=1}^N\varphi^i(t,[0,1])$ is a solution
of the elastic flow of networks with initial datum 
$\mathcal{N}_0:=\cup_{i=1}^N\varphi^i_0([0,1])$.
\end{lemma}
\begin{proof}
We show that 
$\mathcal{N}_0$ is an admissible initial networks. 
Conditions i) and iv) are clearly satisfied, together with condition ii) because at the endpoints of order one
$0=V^i=2\partial_s^2k^i+(k^i)^3-\mu^i k^i$.
Also condition iii) it is easy to get: 
$\partial_x^2\varphi(y)=0$ implies $k(y)=0$,
$\partial_x\varphi(y)=\tau^*$ implies $\tau=\tau^*$
and all the other conditions are already satisfied by the special flow.
At each junction $\gamma^{1}(y_1)=\ldots=\gamma^{m}(y_m)$ 
of order at least three we 
consider two consecutive unit normal vectors 
$\nu^{i}(y_i)$ and $\nu^{k}(y_k)$
such that
$\mathrm{span}\{\nu^{i}(y_i),\nu^{k}(y_k)\}=\mathbb{R}^2$.
For every $j\in\{1,\ldots,m\}$, $j\neq i$, $j\neq k$
we call $\theta^i$ the angle between $\nu^k(0)$
and $\nu^j(0)$, 
$\theta^k$ between $\nu^j(0)$
and $\nu^i(0)$
and $\theta^j$  between $\nu^i(0)$
and $\nu^k(0)$ and we recall that it holds
\begin{align}
V^i\nu^i+\overline{T}^i\tau^i
=V^j\nu^j+\overline{T}^j\tau^j\,,\label{primauguaglianza}\\
V^i\nu^i+\overline{T}^i\tau^i
=V^k\nu^k+\overline{T}^k\tau^k\,.\label{secondauguaglianza}
\end{align}
By testing~\eqref{primauguaglianza} 
by $\sin \theta^j\tau^k$ and by 
$\cos \theta^j\nu^k$ and summing, we get
\begin{equation}\label{id1}
    V^i=\cos\theta^k V^j-\sin\theta^k \overline{T}^j\,.
\end{equation}
If instead we test~\eqref{primauguaglianza} 
by $\cos \theta^j\tau^k$ and by 
$\sin \theta^j\nu^k$
and we subtract the second equality to the first one,
it holds 
\begin{equation}\label{id2}
    \overline{T}^i=\cos\theta^k \overline{T}^j+\sin\theta^k V^j\,.
\end{equation}
Similarly, by testing~\eqref{secondauguaglianza}
by $\cos \theta^k\nu^j$ and by $\sin\theta^k\tau^j$ 
and subtracting the second identity to the first we have
\begin{equation}\label{id3}
    V^i=\cos\theta^j V^k+\sin\theta^j \overline{T}^k\,.
\end{equation}
Finally we test~\eqref{secondauguaglianza}
by $\cos \theta^k\tau^j$ and by $\sin\theta^k\nu^j$ 
and sum, obtaining 
\begin{equation}\label{id4}
    \overline{T}^i=\cos\theta^j \overline{T}^k-\sin\theta^j V^k\,.
\end{equation}
With the help of the identities~\eqref{id1},\eqref{id2},~\eqref{id3}
and~\eqref{id4}
and interchanging the roles of $i,j,k$ we can write
\begin{align*}
    \sin\theta^iV^i=\cos\theta^j \overline{T}^j-\cos\theta^k \overline{T}^k\,,\\
    \sin\theta^kV^k=\cos\theta^i \overline{T}^i-\cos\theta^j \overline{T}^j\,,\\
    \sin\theta^jV^j=\cos\theta^k \overline{T}^k-\cos\theta^i \overline{T}^i\,.
\end{align*}
and so for every
$j\in\{1,\ldots,m\}$, $j\neq i$, $j\neq k$ we have
$\sin\theta^iV^i+ \sin\theta^kV^k+\sin\theta^jV^j=0$, as desired.

The solution $\mathcal{N}$ admits a parametrization
$\varphi$ with the required regularity.
As we have seen for the initial datum, the boundary conditions in Definition~\ref{specialflow}
implies the boundary conditions asked in Definition~\ref{Def:elasticflow}. By definition of solution
of the special flow the parametrizations 
$\varphi=(\varphi^1,\ldots,\varphi^N)$
solves $\partial_t\varphi^i=V^i\nu^i+\overline{T}^i\tau^i$.
Then 
$$\left\langle\partial_t\varphi^i, 
\nu^i\right\rangle \nu^i=\left\langle V^i\nu^i+\overline{T}^i\tau^i,\nu^i\right\rangle\nu^i
=V^i\nu^i =-2\partial_s^2 k^i-(k^i)^3+\mu^i k^i\,,$$
and thus all the properties of solution to the elastic flow
are satisfied.
\end{proof}

\begin{lemma}
Suppose that a closed  curve  parametrized by 
$$
\gamma\in C\left([0,T];C^{5}([0,1];\mathbb{R}^2)\right) \cap  C^1\left([0,T];C^{4}([0,1];\mathbb{R}^2)\right)
$$
is a solution to the elastic flow 
with admissible initial datum $\gamma_0\in C^5([0,1])$.
Then a reparametrization of $\gamma$ is a solution
to the special flow.
\end{lemma}

\begin{proof}
The proof easily follows arguing similarly as in the discussion at the beginning of the section and, in particular,  by recalling that reparametrizations only affect the tangential velocity.
\end{proof}

The above result can be generalized to flow of networks as stated below.

\begin{lemma}
Suppose that a network $\mathcal{N}_0$ of $N$ 
regular curves parametrized by 
$\gamma=(\gamma^1, \ldots\gamma^N)$ with $\gamma^i:[0,1]\to\mathbb{R}^2$, $i\in\{1,\ldots,N\}$
 is an admissible initial network. Then there exist
$N$ smooth functions $\theta^i:[0,1]\to [0,1]$ such that the reparametrisztion 
$\left(\gamma^i\circ\theta^i\right)$
 is an admissible initial parametrization for the special flow.
\end{lemma}

For the proof see~\cite[Lemma 3.31]{GaMePl1}.
Moreover by inspecting the proof of Theorem 3.32 in \cite{GaMePl1}
we see that the following holds.

\begin{prop}\label{riparametriz}
Let $T>0$. 
Let $\mathcal{N}_0$
be an admissible initial network of $N$ curves  parametrized by 
$\gamma_0=(\gamma_0^1, \ldots\gamma_0^N)$
with $\gamma^i:[0,1]\to\mathbb{R}^2$, $i\in\{1,\ldots,N\}$.
Suppose that 
$\mathcal{N}(t)_{t\in [0,T]}$ is a solution to the elastic flow in the time interval 
$[0,T]$ with initial datum $\mathcal{N}_0$ and suppose that
it is parametrized by
regular curves
$\gamma=(\gamma^1, \ldots\gamma^N)$
with $\gamma^i:[0,T]\times[0,1]\to\mathbb{R}^2$.
Then there exists $\widetilde{T}\in (0,T]$ and
a time dependent family of reparametrizations 
$\psi:[0,\widetilde{T}]\times [0,1]\to [0,1]$ such that 
$\varphi(t,x):=(\varphi^1(t,x),\ldots,\varphi^N(t,x))$ with
$\varphi(t,x):=\gamma^i(t,\psi(t,x))$ 
is a solution to the special flow in $[0,\widetilde{T}]$.
\end{prop}

\begin{rem}
In the case of a single open curve, reducing to the special flow is
 is particularly advantageous. 
 Indeed one passes from a the degenerate 
problem~\eqref{evolutionlaw} couple either with quasilinaer 
or fully nonlinear boundary conditions
to a non--degenerate system of quasilinaer PDEs with linear and 
affine
boundary conditions.
\end{rem}

\subsection{Energy monotonicity}\label{sec:energy-decreases-in-time}

Let us name $V^i:=-2\partial_ s^2 k^{i}-(k^i)^2k^i+\mu^i k^i$
the normal velocity of a curve $\gamma^i$ evolving by elastic flow and denote 
the tangential motion by $T^i$:
\begin{equation}\label{motionequationtang}
\partial_t\gamma^i=V^i\nu^i+T^i\tau^i\,.
\end{equation}

\begin{dfnz}
We denote by $\pol_\sigma^h(k)$
a  polynomial in $k,\dots,\ders^h k$ with constant 
coefficients in $\mathbb{R}$ such that every monomial it contains is of the form
\begin{equation*}
C \prod_{l=0}^h	(\ders^lk)^{\beta_l}\quad\text{ with} \quad \sum_{l=0}^h(l+1)\beta_l = \sigma\,,
\end{equation*}
where $\beta_l\in\mathbb{N}$ for $l\in\{0,\dots,h\}$ and $\beta_{l_0}\ge 1$ 
for at least one index $l_0$.
\end{dfnz}

We notice that 
\begin{align}
\partial_s\left(\pol_\sigma^h( k)\right)&=\pol_{\sigma+1}^{h+1}( k)\,,\nonumber\\
\mathfrak{p}_{\sigma_1}^{h_1}(k)\mathfrak{p}_{\sigma_2}^{h_2}
(k)&=\mathfrak{p}_{\sigma_1+\sigma_2}^{\max\{h_1,h_2\}}(k)\,,
\label{calcpol} \\
\mathfrak{p}_\sigma^{h_1}(k)+ \mathfrak{p}_\sigma^{h_2}(k) &= \mathfrak{p}_\sigma^{\max\{h_1,h_2\}}(k).
\end{align}

By~\eqref{eq:VariazioniOggettiGeometrici} the following result holds.

\begin{lemma}\label{evoluzionigeom}
If $\gamma$ satisfies~\eqref{motionequationtang}, the commutation rule
\begin{equation*}
\partial_{t}\partial_{s}=\partial_{s}\partial_{t}+\left(kV-\partial_s T\right)\partial_{s}\,
\end{equation*}
holds. The measure $\mathrm{d}s$ evolves as
\begin{equation*}
\partial_t(\mathrm{d}s)=\left(\partial_s T-kV\right)\mathrm{d}s\,.
\end{equation*}
Moreover the unit tangent vector, unit normal vector and the 
$j$--th
derivatives of scalar curvature of a curve  satisfy
\begin{align}
\partial_{t}\tau&=\left(\partial_s V+T k\right)\nu\,,\nonumber\\
\partial_{t}\nu&=-\left(\partial_s V+T k\right)\tau\,,\nonumber\\
\partial_tk&=\left\langle \partial_{t}\boldsymbol{\kappa},\nu\right\rangle =\partial_s^2 V+T\partial_s k+k^{2}V\,\nonumber\\
&=-2\partial_{s}^{4}k-5k^{2}\partial_{s}^{2}k-6k\left(\partial_{s}k\right)^{2}
+T\partial_{s}k-k^{5}+\mu \left( \partial_{s}^{2}k+k^3\right)\,,\label{kt}\\
\partial_t\partial_s^j k
&=-2\partial_{s}^{j+4}k	
-5k^2\partial_s^{j+2}k
+\mu\,\partial_s^{j+2}k
+T\partial_{s}^{j+1}k
+\mathfrak{p}_{j+5}^{j+1}\left(k\right)+
\mu\,\mathfrak{p}_{j+3}^{j}(k) \\
&= -2\partial_{s}^{j+4}k +T\partial_{s}^{j+1}k +\mathfrak{p}_{j+5}^{j+2}\left(k\right)+
	\mu\,\mathfrak{p}_{j+3}^{j+2}(k)	 \,.
\label{derivativekt}
\end{align}

\end{lemma}

With the help of the previous lemma it is now possible to compute
the derivative in time of a general polynimial $\mathfrak{p}_{\sigma}^h(k)$.
By definition 
every monomial composing $\mathfrak{p}_{\sigma}^h(k)$ is of the form
$\mathfrak{m}(k)=C \prod_{l=0}^h	(\ders^lk)^{\beta_l}$
with $\sum_{l=0}^h(l+1)\beta_l = \sigma$.
Then for every fixed $j\in\{1,\ldots,h\}$
the monomial $\mathfrak{n}(k)=C\beta_{j}(\partial_s^{j}k)^{\beta_j-1}
\prod_{l\neq j, l=0}^h	(\ders^lk)^{\beta_l}$
can be written as   $\mathfrak{n}(k)=\tilde{C}\prod_{l=0}^h
(\ders^lk)^{\alpha_l}$  with $\sum_{l=0}^h(l+1)\alpha_l = \sigma-j-1$.
Differentiating in time $\mathfrak{m}(k)$ we have
\begin{align*}
        \partial_t\left(\mathfrak{m}(k)\right)
&=\sum_{j=0}^h \left(
\left( C\beta_{j}\partial_s^{j}k^{\beta_j-1}\partial_t\partial_s^jk\right)\cdot
\prod_{l\neq j, l=0}^h	(\ders^lk)^{\beta_l}\right)\\
&=\sum_{j=0}^h \left(\left(-2\partial_{s}^{j+4}k	
+T\partial_{s}^{j+1}k
+\mathfrak{p}_{j+5}^{j+2}\left(k\right)+
\mu\,\mathfrak{p}_{j+3}^{j+2}(k)\right)\left(
C\beta_{j}\partial_s^{j}k^{\beta_j-1}\right)\cdot
\prod_{l\neq j, l=0}^h	(\ders^lk)^{\beta_l}\right)\\
&=\mathfrak{p}_{\sigma+4}^{h+4}(k)+T\mathfrak{p}_{\sigma+1}^{h+1}(k)+\mathfrak{p}_{\sigma+4}^{h+2}(k)+\mu \mathfrak{p}_{\sigma+2}^{h+2}(k)\,,
\end{align*}
where we used the product rule~\eqref{calcpol}
and the structure of the monomial $\mathfrak{n}(k)$.
Summing up the contribution of each monomial composing $\mathfrak{p}_{\sigma}^h(k)$ we have
\begin{equation}\label{der-t-pol}
    \partial_t\left(\mathfrak{p}_{\sigma}^h(k)\right)
    =\mathfrak{p}_{\sigma+4}^{h+4}(k)+T\mathfrak{p}_{\sigma+1}^{h+1}(k)+\mu \mathfrak{p}_{\sigma+2}^{h+2}(k)\,.
\end{equation}

\begin{prop}\label{energydecreases}
Let $\mathcal{N}_t$
be a time dependent family of smooth networks composed of $N$ curves, possibly with junctions and fixed endpoint in the plane.
Suppose that $\mathcal{N}_t$ is a 
solution of the elastic flow. 
Then 
\begin{align*}
\partial_{t}\mathcal{E}_\mu(\mathcal{N}_t)&=-\int_{\mathcal{N}} V^2\,\mathrm{d}s\,.
\end{align*}
\end{prop}

\begin{proof}
Using the evolution laws collected in Lemma~\ref{evoluzionigeom}, we get
\begin{align*}
\partial_{t}\int_{\mathcal{N}}k^{2}+\mu\,\mathrm{d}s
&=\int_{\mathcal{N}}2k\partial_{t}k+\left(k^{2}+\mu\right)\left(\partial_s T-kV\right)\,\mathrm{d}s\\
&=\int_{\mathcal{N}}2k\left(\partial_s^2V+T k_{s}+k^{2}V\right)+\left(k^{2}+\mu\right)\left(\partial_s T-kV\right)\,\mathrm{d}s\\
&=\int_{\mathcal{N}}^{}2k\partial_s^2V+k^3V-\mu kV+\partial_s\left(T\left(k^2+\mu\right)\right)\,\mathrm{d}s\,.
\end{align*}
Integrating twice by parts the term $\int 2kV_{ss}$ we obtain 
\begin{equation}\label{eq:perparti}
\partial_{t}\int_{\mathcal{N}}\mathcal{E}_\mu\,\mathrm{d}s
=-\int_{\mathcal{N}}^{}V^2\,\mathrm{d}s+\sum_{i=1}^{N}
\left. 
2k^i\partial_s V^i-2\partial_s k^iV^i+T^i\{\left(k^i\right)^2+\mu^i\}
\right|_{\text{bdry}}\,.
\end{equation}
It remains to show that the contribution of the boundary term in~\eqref{eq:perparti} equals zero,
whatever boundary condition we decide to impose at the endpoint among the ones listed in
Definition~\ref{Def:elasticflow}.
The case of the closed curve is trivial.

Let us start with the case of an endpoint $\gamma^j(y)$ (with $y\in\{0,1\}, j\in\{1,\ldots,N\}$)
subjected to Navier boundary condition, namely $k^j(y)=0$.
The point remains fixed, that implies $V^j(y)=T^j(y)=0$. The term $2k^j(y)\partial_s V^j(y)$
vanishes because  $k^j(y)=0$.

Suppose instead that the curve is clamped at $\gamma^j(y)$ with $\tau^j(y)=\tau^*$.
Then using Lemma~\ref{evoluzionigeom}, $0=\partial_t\tau^j(y)=(\partial_s V^j(y)-T^j(y)k^j(y))\nu^j(y)$. 
Hence 
$$
2k^j(y)\left(\partial_s V^j(y)-T^j(y)k^j(y)\right)=0\,,
$$ 
that combined with $V^j(y)=T^j(y)=0$ implies that the boundary terms vanish in~\eqref{eq:perparti}.

Consider now a junction of order $m$ where natural boundary conditions have been imposed.
Up to inverting the orientation of the parametrizations of the curves, we suppose that all
the curves concur at the junctions at $x=0$.
The curvature condition $k^i(0)=0$ with $i\in\{1,\ldots.m\}$ gives
\begin{equation*}
\sum_{i=1}^m 2k^i(0)\partial_s V^i(0)+T^i(0)\left(k^i(0)\right)^2=0\,.
\end{equation*}
Differentiating in time the concurrency condition
$\gamma^1(0)=\ldots\gamma^m(0)$ we obtain
$$
V^1(0)\nu^1(0)+T^1(0)\tau^1(0)=\ldots=V^m(0)\nu^m(0)+T^m(0)\tau^m(0)\,,
$$
	that combined with 
	the third order condition $0=\sum_{i=1}^m    2\partial_sk^i(0)\nu^i(0)-\mu^i\tau^i(0)$ gives
	\begin{align*}
	0&=\left\langle-\partial_t\gamma^1(0), \sum_{i=1}^m 2\partial_sk^i(0)\nu^i(0)-\mu^i\tau^i(0)\right\rangle\\
	&=\sum_{i=1}^m \left\langle -V^i(0)\nu^i(0)-T^i(0)\tau^i(0),
	2\partial_sk^i(0)\nu^i(0)-\mu^i\tau^i(0)\right\rangle
	=
	\sum_{i=1}^m-2\partial_sk^i(0)V^i(0)+\mu ^iT^i(0)\,,
	\end{align*}
hence  the boundary terms vanish and we get the desired result.

To conclude, consider a junction of order $m$, where the curves concur at 
$x=0$ and suppose that we have imposed there clamped boundary conditions.
In this case using the concurrency condition differentiated in time 
and the third order condition
we find
\begin{align}
0&= 
\sum_{i=1}^m    \left\langle-\partial_t\gamma^1(0), 2\partial_sk^i(0)\nu^i(0)+\left((k^i(0))^2-\mu^i\right)\tau^i(0)\right\rangle\nonumber\\
&=\sum_{i=i}^m-2\partial_s k^i(0)V^i(0)-\left((k^i(0))^2-\mu^i\right)T^i(0)\,.\label{eq:firstbdry}
\end{align}

Differentiating in time the angle condition 
$$
\left\langle \tau^i(0),\tau^{i+1}(0)\right\rangle=c^{i,i+1}=\cos (\theta^{i,i+1})
$$
we have
\begin{align*}
0&=\left\langle \partial_t\tau^i(0),\tau^{i+1}(0)\right\rangle
+\left\langle \tau^i(0),\partial_t\tau^{i+1}(0)\right\rangle\\
&=\left\langle (\partial_s V^i(0)+T^i(0)k^i(0))\nu^i(0),\tau^{i+1}(0)\right\rangle
+\left\langle \tau^i(0), (\partial_s V^{i+1}(0)+T^{i+1}(0)k^{i+1}(0))\nu^{i+1}(0)\right\rangle\\
&=(\partial_s V^i(0)+T^i(0)k^i(0))\sin(\theta^{i,i+1})-(\partial_s V^{i+1}(0)+T^{i+1}(0)k^{i+1}(0))\sin(\theta^{i,i+1})\,,
\end{align*}
and hence $\partial_s V^i(0)+T^i(0)k^i(0)=\partial_s V^{i+1}(0)+T^{i+1}(0)k^{i+1}(0)$.
Repeating the previous computation for every $i\in\{2,\ldots,m-1\}$ we get
\begin{equation*}
V^{1}_s(0)+T^{1}(0)k^{1}(0)=\ldots=
V^{m}_s(0)+T^{m}(0)k^{m}(0)\,,
\end{equation*}
that together with the curvature condition $\sum k^i=0$ at the junction gives
\begin{align*}
0=2\left(\partial_s V^{1}(0)+T^{1}(0)k^{1}(0)\right)\sum_{i=1}^m k^i(0)=\sum_{i=1}^m
2\partial_s V^{i}(0)k^i(0)+2T^{i}(0)(k^{i})^2(0)\,.
\end{align*}
Summing this last equality with~\eqref{eq:firstbdry} we have that the boundary terms
vanishes also in this case. 
\end{proof}

\section{Short time existence}\label{sec:ShortTimeExistence}

We prove a short time existence result for the elastic flow of closed curves.
We then explain how it can be generalized to other situations and
which are the main difficulties that arises when we pass from one curve to networks.

\subsection{Short time existence of the special flow}

First of all we aim to prove the existence of a solution to the special flow.
Omitting the dependence on $(t,x)$ 
we can write the motion equation of a curve subjected to~\eqref{specialflow} as  
\begin{equation*}
\partial_t\varphi=-2\frac{\partial_x^4\varphi}{\vert \partial_x\varphi\vert^4}
+\widetilde{f}(\partial_x^3\varphi,\partial_x^2\varphi,\partial_x\varphi)\,.
\end{equation*}

We linearize the highest order terms of the previous equation 
around the initial parametrization $\varphi^0$ obtaining
\begin{align}\label{definizione-f}
\partial_t\varphi
+\frac{2}{\vert\partial_x\varphi^0\vert^4}\partial_x^4\varphi
&=\left(\frac{2}{\vert\partial_x\varphi^0\vert^4} -\frac{2}{\vert\partial_x\varphi\vert^4}\right)
\partial_x^4\varphi
+\widetilde{f}(\partial_x^3\varphi,\partial_x^2\varphi,\partial_x\varphi)\nonumber \\
&=:f(\partial_x^4\varphi,\partial_x^3\varphi,\partial_x^2\varphi,\partial_x\varphi)\,.
\end{align}

\begin{dfnz}
Given $\varphi^0:\mathbb{S}^1\to\mathbb{R}^2$ an admissible initial parametrization
for~\eqref{specialflow}, the linearized system about  $\varphi^0$ associated to
the special flow of a closed curve is given by 
\begin{equation}\label{linearisedproblem}
\begin{cases}
\begin{array}{lll}
\partial_t\varphi(t,x)+2\frac{\partial_x^4\varphi(t,x)}{\vert\partial_x\varphi^0(x)\vert^4}&=f(t,x)&\;\text{on }[0,T]\times\mathbb{S}^1\\
\varphi(0,x)&=\psi(x) &\;\text{on }\mathbb{S}^1\,.
\end{array}
\end{cases}
\end{equation}
Here $(f,\psi)$ is a generic couple to be specified later on.
\end{dfnz}

Let $\alpha\in (0,1)$ be fixed.
Whenever a curve $\gamma$ is regular, there exists a constant $c>0$ such that 
$\inf_{x\in \mathbb{S}^1}\vert\partial_x\gamma\vert\geq c$. 
From now on we fix an admissible initial parametrization $\varphi^0$ 
with 
\begin{equation*}
\Vert \varphi^0 \Vert_{C^{4+\alpha}(\mathbb{S}^1;\mathbb{R}^2)}=R\,,\quad\text{and}\quad
\inf_{x\in \mathbb{S}^1}\vert \partial_x\varphi^0(x)\vert \geq c\,.
\end{equation*}
Then for every $j\in\mathbb{N}$ there holds
\begin{equation*}
\left\lVert \frac{1}{\vert \partial_x\varphi^0\vert^j}
\right\rVert_{C^{\alpha}(\mathbb{S}^1;\mathbb{R}^2)}\leq C(R,c)
\,.
\end{equation*}

\begin{dfnz}\label{linearspaces}
For  $T>0$
we consider the  linear spaces
\begin{align*}
\mathbb{E}_T:=&\,
C  ^{\frac{4+\alpha}{4},{^{4+\alpha}}}\left([0,T]\times \mathbb{S}^1;\mathbb{R}^2\right)
\,,\\
\mathbb{F}_T:=&\,
C^{\frac{\alpha}{4},{^{\alpha}}}\left([0,T]\times \mathbb{S}^1;\mathbb{R}^2\right)
\times  C^{4+\alpha}\left(\mathbb{S}^1;\mathbb{R}^2\right)
\,,
\end{align*}
endowed with the norms 
\begin{align*}
\Vert\gamma\Vert_{\mathbb{E}_T}:=\Vert\gamma\Vert_{C^{\frac{4+\alpha}{4},^{4+\alpha}} }\,,\quad
\Vert (f,\psi)\Vert_{\mathbb{F}_T}:=\Vert f\Vert_{C^{\frac\alpha4,^{\alpha}}}+
\Vert \psi\Vert_{C^{4+\alpha}}\,.
\end{align*}
and we define the operator 
$\mathcal{L}_{T}:\mathbb{E}_T\to \mathbb{F}_T$ by
\begin{equation*}
\mathcal{L}_{T}(\varphi):=\left(\mathcal{L}^1_{T}(\varphi),\mathcal{L}^2_{T}(\varphi)\right)
:=\left(\partial_t\varphi+\frac{2}{\vert\partial_x\varphi^0\vert^4}\partial_x^4\varphi, \varphi_{\vert t=0}\right)\,.
\end{equation*}
\end{dfnz}

\begin{rem}
For every $T>0$ 
the operator $\mathcal{L}_{T}:\mathbb{E}_T\to \mathbb{F}_T$  is well--defined, linear and 
continuous.
\end{rem}

\begin{teo}\label{linearexistence}
Let $\alpha\in(0,1)$, 
$(f,\psi)\in\mathbb{F}_T$.
Then for every $T>0$ the
system~\eqref{linearisedproblem} 
has a unique solution $\varphi\in \mathbb{E}_T$.
Moreover, for all $T>0$ there exists  $C(T)>0$ such that if $\varphi \in \mathbb{E}_T$ is a solution, then
\begin{equation}\label{stima-dipendente-da-T}
\Vert \varphi\Vert_{\mathbb{E}_T }
\leq C(T)\Vert (f,\psi)\Vert_{\mathbb{F}_T}\,.
\end{equation} 
\end{teo}

\begin{proof}
See for instance~\cite[Theorem 4.3.1]{LunardiSemigroups2013} and \cite[Theorem 4.9]{solonnikov2}.
\end{proof}

From the above theorem we get the following consequence.
\begin{cor}
The linear operator $\mathcal{L}_{T}:\mathbb{E}_T\to \mathbb{F}_T$ is a continuous isomorphism.
\end{cor}

By the above corollary, we can denote by $\mathcal{L}^{-1}_T$ the inverse of $\mathcal{L}_T$.

Notice that till now we have considered fixed $T>0$ and derived~\eqref{stima-dipendente-da-T},
where the constant $C$ depends on $T$. 
Now, once a certain interval of time $(0, \widetilde{T}]$ with 
$\widetilde{T}>0$ is chosen,  we show that
for every $T\in (0, \widetilde{T}]$
it possible to  estimate the norm of $\mathcal{L}^{-1}_T$ with a constant independent of $T$.

\begin{lemma}\label{uniformestimateestension}
For all $\widetilde{T}>0$ there exists a constant $c(\widetilde{T})$ such that
$$
\sup_{T\in (0,\frac{1}{2}\widetilde{T}]}\Vert \mathcal{L}_T^{-1}\Vert_{\mathcal{L}(\mathbb{F}_T,\mathbb{E}_T)}\leq c(\widetilde{T})\,.
$$
\end{lemma}
\begin{proof}
Fix $\widetilde{T}>0$, for all $T\in (0,\widetilde{T}]$,
for every $(f,\psi)\in\mathbb{F}_T$ we define the extension operator
$E(f,\psi):=(\widetilde{E}f,\psi)$ by 
\begin{align*}
&\widetilde{E}:C  ^{\frac{\alpha}{4},{^{\alpha}}}\left([0,T]\times\mathbb{S}^1;\mathbb{R}^2\right)\to
C  ^{\frac{\alpha}{4},{^{\alpha}}}\left([0,\widetilde{T}]\times\mathbb{S}^1;\mathbb{R}^2\right)\\
&\widetilde{E}f(t,x):=
\begin{cases}
f(t,x) &\text{for}\, t\in[0,T],\\
f\left(T\frac{\widetilde{T}-t}{\widetilde{T}-T},x\right) &\text{for}\, t\in(T,\widetilde{T}],\\
\end{cases}
\end{align*}
It is clear that $E(f,\psi)\in\mathbb{F}_{\widetilde{T}}$ and that 
$\Vert E \Vert_{\mathcal{L}(\mathbb{F}_T,\mathbb{F}_{\widetilde{T}})}\leq 1$.

Moreover  $\mathcal{L}^{-1}_{\widetilde{T}}(E(f,\psi))_{\vert [0,T]}=\mathcal{L}^{-1}_T(f,\psi)$ by uniqueness
and then
\begin{align*}
\Vert \mathcal{L}^{-1}_T(f,\psi) \Vert_{\mathbb{E}_T}
&\leq \Vert \mathcal{L}^{-1}_{\widetilde{T}}(E(f,\psi)) \Vert_{\mathbb{E}_{\widetilde{T}}}\\
&\leq \Vert \mathcal{L}_{\widetilde{T}}^{-1}\Vert_{\mathcal{L}(\mathbb{F}_{\widetilde{T}},\mathbb{E}_{\widetilde{T}})}
\Vert E(f,\psi)\Vert_{\mathbb{F}_{\widetilde{T}}}\leq c(\widetilde{T})\Vert (f,\psi)\Vert_{\mathbb{F}_{T}}\,.
\end{align*}
\end{proof}

\begin{dfnz}\label{affinespaces}
We define the affine spaces
\begin{align*}
\mathbb{E}^0_T&=\{\gamma\in 
\mathbb{E}_T\,\text{such that }\,\gamma_{\vert t=0}=\varphi^0\}
\,,\\
\mathbb{F}^0_T&=C^{\frac{\alpha}{4},{^{\alpha}}}\left([0,T]\times\mathbb{S}^1;\mathbb{R}^2\right)
\times\{\varphi^0\}
\,.
\end{align*}
\end{dfnz}

In the following 
we denote by $\overline{B_M}$ the closed ball of radius $M$ and center $0$ 
in $\mathbb{E}_T$.

\begin{lemma}\label{regular}
Let $\widetilde{T}>0$, $M>0$, $c>0$ and $\varphi^0$ an admissible initial parametrization
with $\inf_{x\in\mathbb{S}^1}\vert\partial_x\varphi^0\vert\geq c$.
Then there exists $\widehat{T}=\widehat{T}(c,M)\in (0,\widetilde{T}]$ such that 
for all $T\in (0,\widehat{T}]$ every curve 
$\varphi\in \mathbb{E}^0_T\cap B_M$ is regular with 
\begin{equation}\label{inverseest}
\inf_{x\in \mathbb{S}^1}\vert \partial_x\varphi(t,x)\vert\geq\frac{c}{2}\,.
\end{equation}
Moreover for every $j\in\mathbb{N}$
\begin{equation*}
\left\lVert \frac{1}{\vert \partial_x\varphi(t,x)\vert^j}\right\rVert_{C^{\frac{\alpha}{4},\alpha}([0,T]\times[0,1])}
\leq C(c, M, j)\,.
\end{equation*}
\end{lemma}
\begin{proof}
We have
$$
\vert \partial_x\varphi(t,x)\vert\geq \vert\partial_x \varphi^0(x)\vert -\vert \partial_x\varphi(t,x)-\partial_x \varphi^0(x)\vert\,,
$$
with $\vert \partial_x\varphi(t,x)-\partial_x \varphi^0(x)\vert\leq \left[\varphi\right] _{\beta,0}t^\beta
\leq Mt^\beta$ with $\beta=\frac{3}{4}+\frac{\alpha}{4}$.
Taking $\widehat{T}$ sufficiently small, passing to the infimum we get the first claim.
As a consequence 
\begin{equation}\label{lowbound}
\sup_{x\in[0,1]}\frac{1}{\vert \partial_x\varphi(t,x)\vert}\leq \frac{2}{c}\,.
\end{equation}
Then  for $j=1$ the second estimate follows directly combining
 the estimate~\eqref{lowbound} with
 the definition of the norm 
$\Vert\cdot\Vert_{C^{\frac{\alpha}{4},\alpha}([0,T]\times\mathbb{S}^1)}$.
The case $j\geq 2$ follows from multiplicativity of the norm. 
\end{proof}

Form now on we fix $\widetilde{T}=1$
and we denote by $\hat{T}=\hat{T}(c,M)$
the time given by Lemma~\ref{regular} for given $c$ and $M$.

\begin{dfnz}
For every $T\in (0,\hat{T}]$
we define the map 
\begin{align*}
N_{T}:
\begin{cases}
\mathbb{E}^0_T \to C  ^{\frac\alpha4,{^{\alpha}}}([0,T]\times\mathbb{S}^1;\mathbb{R}^2)\\
\varphi \mapsto
f(\varphi),
\end{cases}
\end{align*}
where the functions $f(\varphi):=f(\partial_x^4\varphi,\partial_x^3\varphi,\partial_x^2\varphi,\partial_x\varphi)$ is defined in~\eqref{definizione-f}.
Moreover we introduce 
the map $\mathcal{N}_T$
given by $\mathbb{E}^0_T \ni  \gamma\,\mapsto 
(N_{T}(\gamma),\gamma\vert_{t=0})$.
\end{dfnz}

\begin{rem}
We remind that $f$ is given by 
\begin{align*}
f(\varphi)=&\left(\frac{2}{\vert\partial_x \varphi^0\vert^4} -\frac{2}{\vert\partial_x\varphi\vert^4}\right)\partial_x^4\varphi+12\frac{\partial_x^3\varphi\left\langle \partial_x^2\varphi,\partial_x\varphi\right\rangle }{\left|\partial_x\varphi\right|^{6}}
+5\frac{\partial_x^2\varphi\left|\partial_x^2\varphi\right|^{2}}{\left|\partial_x\varphi\right|^{6}}\\
&+8\frac{\partial_x^2\varphi\left\langle \partial_x^3\varphi,\partial_x\varphi\right\rangle }
{\left|\partial_x\varphi\right|^{6}}
-35\frac{\partial_x^2\varphi\left\langle \partial_x^2\varphi,\partial_x\varphi\right\rangle ^{2}}
{\left|\partial_x\varphi\right|^{8}}+\mu\frac{\partial_x^2\varphi}{\vert\partial_x\varphi\vert^2}\,.
\end{align*}
By Lemma~\ref{regular}, for $\varphi \in \mathbb{E}^0_T$, we have that for all $t\in[0,T]$
the map $\varphi(t)$ is a regular curve. 
Hence $N_{T}$ is well--defined.
Furthermore we notice that 
the map $\mathcal{N}_t$
is a mapping from $\mathbb{E}^0_T$ to $\mathbb{F}^0_T$.
\end{rem}

The following lemma is a classical result
on parabolic H\"{o}lder spaces. 
For a proof see for instance~\cite{LunardiSemigroups2013}.

\begin{lemma}\label{notproved}
Let $k\in\{1,2,3\}$, $T\in [0,1]$ and $\varphi,\widetilde{\varphi}\in\mathbb{E}_T^0$. We denote by 
$\varphi^{(4-k)}, \widetilde{\varphi}^{(4-k)}$ the 
$(4-k)$--th  space derivative of $\varphi$ and $\widetilde{\varphi}$, respectively.
Then there exist $\varepsilon>0$ and a constant $\widetilde{C}$ independent of $T$ such that
$$
\left\lVert
\varphi^{(4-k)}-\widetilde{\varphi}^{(4-k)}
\right\rVert_{C^{\frac{\alpha}{4},\alpha}}
\leq 
\widetilde{C}T^{\varepsilon}\left\lVert
\varphi^{(4-k)}-\widetilde{\varphi}^{(4-k)}
\right\rVert_{C^{\frac{k+\alpha}{4},k+\alpha}}
\leq \widetilde{C}T^{\varepsilon}\left\lVert \varphi-\widetilde{\varphi}
\right\rVert_{\mathbb{E}_T}\,.
$$
\end{lemma}

\begin{dfnz}
Let $\varphi^0$ be an admissible initial parametrization,
$c:=\inf_{x\in \mathbb{S}^1}\vert\partial_x \varphi^0\vert$.
For a positive $M$ and 
a time $T\in (0,\widehat{T}(c,M)]$ we define 
$\mathcal{K}_T:\mathbb{E}^0_T\cap\overline{B_M}\to\mathbb{E}^0_T$ by 
$$\mathcal{K}_T:=\mathcal{L}_T^{-1}\circ \mathcal{N}_T\,.$$
\end{dfnz}

\begin{prop}\label{contraction}
Let $\varphi^0$ be an admissible initial parametrization,
$c:=\inf_{x\in \mathbb{S}^1}\vert\partial_x \varphi^0\vert$.
Then there exists a positive radius
$M(\varphi^0)>\Vert \varphi^0\Vert_{C^{4+\alpha}}$
and a time $\overline{T}(c,M)$ such that
for all $T\in (0,\overline{T}]$ the map
$\mathcal{K}_T:\mathbb{E}^0_T\cap\overline{B_M}\to\mathbb{E}^0_T$ takes values in $\mathbb{E}^0_T\cap\overline{B_M}$ and it is a contraction.
\end{prop}

In the following proof constants may vary from line to line and depend on $c$, $M$ and $\Vert \varphi^0\Vert_{C^{4+\alpha}}$.

\begin{proof}
Let $M>0$ and $\widetilde{T}>0$ be arbitrary positive numbers. Let $\widehat{T}(c,M)$ be given by~\Cref{regular} and assume without loss of generality that $\widehat{T}(c,M)<\tfrac12 \widetilde{T}$. 
Let $T\in (0,\widehat{T}(c,M)]$ be a generic time.

Clearly $\mathcal{L}_T^{-1}(\mathbb{F}_T^0)\subseteq \mathbb{E}^0_T$ and 
the $\mathcal{K}_T$ is well defined on
$\mathbb{E}^0_T\cap\overline{B_M}$.

First we show that there exists a time
 $T'\in (0,\widehat{T}(c,M))$ 
 such that for all $T\in (0,T']$,
 for every $\varphi,\widetilde{\varphi} \in \mathbb{E}^0_T\cap\overline{B_M}$, it holds
 \begin{equation}\label{contraz}
     \Vert \mathcal{K}_T(\varphi)-
     \mathcal{K}_T(\widetilde{\varphi})\Vert_{\mathbb{E}_T}
     \leq \frac12\Vert \varphi-
     \widetilde{\varphi}\Vert_{\mathbb{E}_T}\,.
 \end{equation}
 We  begin by estimating
\begin{align*}
&\Vert N_{T}(\varphi)-N_{T}(\widetilde{\varphi})\Vert_{C^{\frac{\alpha}{4},\alpha}} =
\Vert f(\varphi)-f(\widetilde{\varphi})\Vert_{C^{\frac{\alpha}{4},\alpha}}\,.
\end{align*}
The highest order term in the above norm is
\begin{equation}\label{highestord}
\begin{split}
&\left(\frac{2}{\vert\partial_x \varphi^0\vert^4} -\frac{2}{\vert\partial_x\varphi\vert^4}\right)
\partial_x^4\varphi+
\left(\frac{2}{\vert\partial_x \widetilde{\varphi}\vert^4} -\frac{2}{\vert\partial_x \varphi^0\vert^4}\right)
\partial_x^4 \widetilde{\varphi} \\
&=  \left(\frac{2}{\vert\partial_x \varphi^0\vert^4} -\frac{2}{\vert\partial_x\varphi\vert^4}\right)
\left(\partial_x^4\varphi-\partial_x^4 \widetilde{\varphi}\right)+
\left(\frac{2}{\vert\partial_x \widetilde{\varphi}\vert^4} -\frac{2}{\vert\partial_x\varphi\vert^4}\right)
\partial_x^4 \widetilde{\varphi}
\end{split}
\end{equation}
We can rewrite the above expression using  the identity
\begin{equation}\label{identitadenominatore}
\frac{1}{\vert a \vert^4} -\frac{1}{\vert b \vert^4}=
\left(\vert b \vert-\vert a \vert \right)
\left(\frac{1}{\vert a \vert^2\vert b \vert}
+\frac{1}{\vert a \vert\vert b \vert^2}\right)
\left( \frac{1}{\vert a \vert^2}
+\frac{1}{\vert b \vert^2}\right)\,.
\end{equation}
We get
\begin{equation*}
\left(\frac{2}{\vert\partial_x \varphi^0\vert^4} -\frac{2}{\vert\partial_x\varphi\vert^4}\right)  =  
\left(\vert \partial_x\varphi \vert-\vert \partial_x \varphi^0 \vert \right)
\left(\frac{1}{\vert \partial_x \varphi^0 \vert^2\vert \partial_x\varphi \vert}
+\frac{1}{\vert \partial_x \varphi^0 \vert\vert \partial_x\varphi \vert^2}\right)
\left( \frac{1}{\vert \partial_x \varphi^0 \vert^2}
+\frac{1}{\vert \partial_x\varphi \vert^2}\right)\,.
\end{equation*}
In order to control $\left(\frac{1}{\vert \partial_x \varphi^0 \vert^2\vert \partial_x\varphi \vert}
+\frac{1}{\vert \partial_x \varphi^0 \vert\vert \partial_x\varphi \vert^2}\right)
\left( \frac{1}{\vert \partial_x \varphi^0 \vert^2}
+\frac{1}{\vert \partial_x\varphi \vert^2}\right)$ 
we  use Lemma~\ref{regular}. 
Now we identify $\varphi^0$ with its constant in time extension
$\psi^0(t,x):=\varphi^0(x)$, which belongs to $\mathbb{E}_T^0$
for arbitrary $T$.  Observe that $\|\psi^0\|_{\mathbb{E}_T} = \|\psi^0\|_{C^{\frac{4+\alpha}{4},4+\alpha}} = \|\varphi^0\|_{C^{4+\alpha}}$ is independent of $T$.
Then making use of 
Lemma~\ref{notproved} we obtain 
\begin{align*}
\left\lVert
\vert\partial_x\varphi\vert -\vert\partial_x\psi^0\vert
\right\rVert_{C^{\alpha,\frac{\alpha}{4}}}
\leq
\left\lVert
\partial_x\varphi -\partial_x\psi^0
\right\rVert_{C^{\alpha,\frac{\alpha}{4}}}
\leq 
C T^{\varepsilon}
\Vert \varphi-\psi^0\Vert_{\mathbb{E}_T}
\leq CM T^{\varepsilon}\,.
\end{align*}
Then 
\begin{align*}
\left\lVert
\left(\frac{2}{\vert\partial_x \varphi^0\vert^4} -\frac{2}{\vert\partial_x\varphi\vert^4}\right)
\left(\partial_x^4\varphi
-\partial_x^4 \widetilde{\varphi}\right)
\right\rVert_{C^{\alpha,\frac{\alpha}{4}}}
\leq
C MT^{\varepsilon}
\Vert \varphi-\widetilde{\varphi}\Vert_{\mathbb{E}_T}\,.
\end{align*}
Similarly we obtain
allows us to write
\begin{align}\label{fractionestimate}
\left\lVert
\left(\frac{2}{\vert\partial_x \widetilde{\varphi}\vert^4} -\frac{2}{\vert\partial_x\varphi\vert^4}\right)
\partial_x^4 \widetilde{\varphi}
\right\rVert_{C^{\alpha,\frac{\alpha}{4}}}
\leq
C MT^{\varepsilon}
\Vert \varphi-\widetilde{\varphi}\Vert_{\mathbb{E}_T}\,.
\end{align}

The lower order terms of 
$f(\varphi)-f(\widetilde{\varphi})$ 
are of the  form
\begin{equation}\label{shape}
\frac{a\left\langle b,c\right\rangle }{\vert d\vert^j}
-\frac{\tilde{a}\langle \tilde{b},\tilde{c}\rangle }{\vert\tilde{d}\vert^{j}}\,,
\end{equation}
with $j\in\{2,6,8\}$ 
and with $a,b,c,d,\tilde{a},\tilde{b},\tilde{c},\tilde{d}$ space derivatives up to order three
of $\varphi$ and $\widetilde{\varphi}$, respectively.
Adding and subtracting the expression
$$
\frac{\tilde{a}\left\langle b,c\right\rangle }{\vert d\vert^j}+\frac{\tilde{a}\langle \tilde{b},c\rangle }{\vert d\vert^j}+
\frac{\tilde{a}\langle \tilde{b},\tilde{c}\rangle }{\vert d\vert^j}
$$
to~\eqref{shape},
we get 
\begin{equation}\label{lowerorder}
\frac{(a-\tilde{a})\left\langle b,c\right\rangle }{\vert d\vert^j}+
\frac{\tilde{a}\left\langle (b-\tilde{b}),c\right\rangle }{\vert d\vert^j}
+\frac{\tilde{a}\left\langle \tilde{b},(c-\tilde{c})\right\rangle }{\vert d\vert^j}
+\left(\frac{1}{\vert d\vert^j}-\frac{1}{\vert \tilde{d}\vert^j}\right)
\tilde{a}\left\langle \tilde{b},\tilde{c}\right\rangle\,.
\end{equation}
With the help of Lemma~\ref{notproved} we can estimate the first term of~\eqref{lowerorder}
in the following way:
\begin{align*}
\left\lVert \frac{(a-\tilde{a})\left\langle b,c\right\rangle }{\vert d\vert^j}
\right\rVert_{C^{\frac{\alpha}{4},\alpha}}
\leq C\Vert a-\tilde{a} \Vert_{C^{\frac{\alpha}{4},\alpha}}
\leq CT^{\varepsilon}\Vert \varphi-\widetilde{\varphi}\Vert_{\mathbb{E}_T}\,.
\end{align*}
The second and the third term of~\eqref{lowerorder}  
can be estimated similarly by Cauchy-Schwarz inequality.
To obtain the desired estimate for the last term 
of~\eqref{lowerorder} we proceed in a similar way
as for the second term of~\eqref{highestord}.
We use the identities
\begin{align*}
\frac{1}{\vert d \vert^2} -\frac{1}{\vert \tilde{d} \vert^2}=&
\left(\vert \tilde{d} \vert-\vert d \vert \right)
\left(\frac{1}{\vert d \vert^2\vert \tilde{d} \vert}
+\frac{1}{\vert d \vert\vert \tilde{d} \vert^2}\right)\,,\\
\frac{1}{\vert d \vert^j} -\frac{1}{\vert \tilde{d} \vert^j}=&
\left(\vert \tilde{d} \vert-\vert d \vert \right)
\left(\frac{1}{\vert d \vert^2\vert \tilde{d} \vert}
+\frac{1}{\vert d \vert\vert \tilde{d} \vert^2}\right)
\left( \frac{1}{\vert d \vert^2}
+\frac{1}{\vert \tilde{d} \vert^2}\right)
\left( \frac{1}{\vert d \vert^{j-4}}
+\frac{1}{\vert \tilde{d} \vert^{j-4}}\right)
\,,
\end{align*}
for $j\in\{6,8\}$ and Lemma~\ref{regular} and~\ref{notproved} and we finally get
\begin{equation*}
 \left\lVert
\left(\frac{1}{\vert d\vert^j}-\frac{1}{\vert \tilde{d}\vert^j}\right)
\tilde{a}\left\langle \tilde{b},\tilde{c}\right\rangle\right\rVert_{C^{\frac{\alpha}{4},\alpha}}
\leq C T^{\varepsilon}\Vert d-\tilde{d}\Vert_{C^{\frac{\alpha}{4},\alpha}}
\leq C T^{\varepsilon}\Vert \varphi-\widetilde{\varphi}\Vert_{\mathbb{E}_T}\,.   
\end{equation*}
Putting the above inequalities together we have
\begin{equation*}
    \Vert f(\varphi)-f(\widetilde{\varphi})\Vert_{C^{\frac{\alpha}{4}},\alpha}\leq 
C T^{\varepsilon}\Vert \varphi-\widetilde{\varphi}\Vert_{\mathbb{E}_T}\,.
\end{equation*}
By~\Cref{uniformestimateestension}, this implies that for all $T\in (0,\widehat{T}(M,c)]$
\begin{equation}\label{eq:StimaKIntermedia}
\begin{split}
\Vert \mathcal{K}_T(\varphi)
-\mathcal{K}_T(\widetilde{\varphi})\Vert_{\mathbb{E}_T}
&=\Vert \mathcal{L}^{-1}_T(\mathcal{N}_T(\varphi))
-\mathcal{L}^{-1}(\mathcal{N}_T(\widetilde{\varphi}))\Vert_{\mathbb{E}_T}\\
&\leq \sup_{T\in[0,\widehat{T}]}\Vert \mathcal{L}^{-1}_T\Vert_{\mathcal{L}(\mathbb{F}_T,\mathbb{E}_T)}
\Vert \mathcal{N}_T(\varphi)-\mathcal{N}_T(\widetilde{
\varphi})\Vert_{\mathbb{F}_T}\\
&\leq C(M,c,\widetilde{T})T^\varepsilon\Vert \varphi-\widetilde{\varphi}\Vert_{\mathbb{E}_T}\,,
\end{split}
\end{equation}
with $0<\varepsilon<1$. Choosing $T'$ small enough we can conclude that
for every $T\in (0,T']$ 
the inequality~\eqref{contraz} holds.

In order to conclude the proof it remains to show that
we can choose $M$ sufficiently big so that 
$\mathcal{K}_T$ maps 
$\mathbb{E}^0_T\cap\overline{B_M}$
into itself.

As before we identify $\varphi^0(x)$
with its constant in time extension $\psi^0(t,x)$.
Notice that the expressions $\mathcal{K}_T(\psi^0)$
and $\mathcal{N}_T(\psi^0)$ are then well defined.

As $M$ is an arbitrary positive constant, let us choose $M$ at the beginning, depending on $\varphi^0$ and $\widetilde{T}$ only, so that
\[
\|\psi^0\|_{\mathbb{E}_T} = \|\varphi^0\|_{C^{4+\alpha}} < \frac{M}{2} 
\qquad \forall\,T>0\,,
\]
and
\[
\begin{split}
\Vert \mathcal{K}_T(\psi^0)\Vert_{\mathbb{E}_T}
&\le \sup_{T\in[0,\widetilde{T}/2-\delta]} \Vert \mathcal{L}_T^{-1}\Vert_{\mathcal{L}(\mathbb{F}_T,\mathbb{E}_T)}\Vert \mathcal{N}_T(\psi^0)\Vert_{\mathbb{F}_T} \\
&= \sup_{T\in[0,\widetilde{T}/2-\delta]} \Vert \mathcal{L}_T^{-1}\Vert_{\mathcal{L}(\mathbb{F}_T,\mathbb{E}_T)}\Vert (f(\varphi^0),\varphi^0))\Vert_{\mathbb{F}_T} \\
&\le c(\widetilde{T}) C(\varphi^0) \\
& < \frac{M}{2} 
\qquad \forall\,\delta>0\,,
\end{split}
\]
where we used that $\Vert (f(\varphi^0),\varphi^0))\Vert_{\mathbb{F}_T}$ is time independent and then estimated by a constant $C(\varphi^0)$ depending only on $\varphi^0$ and we also used~\Cref{uniformestimateestension}. For $T\in(0,T']$, as $T'\le \widehat{T}(c,M) \le \tfrac12\widetilde{T}-\delta$ for some positive $\delta$, we also have
\[
\begin{split}
\Vert \mathcal{K}_T(\varphi)\Vert_{\mathbb{E}_T} 
&\le \Vert \mathcal{K}_T(\psi^0)\Vert_{\mathbb{E}_T} + \Vert \mathcal{K}_T(\varphi) - \mathcal{K}_T(\psi^0)\Vert_{\mathbb{E}_T} \\
&< \frac{M}{2} + C(M,c,\widetilde{T})T^\varepsilon 2M\,,
\end{split}
\]
for any $\varphi\in \mathbb{E}^0_T\cap \overline{B_M}$, where we used~\eqref{eq:StimaKIntermedia}. It follows that by taking $T\le T'$ sufficiently small, we have that $\mathcal{K}_T:\mathbb{E}^0_T\cap \overline{B_M} \to \mathbb{E}^0_T\cap \overline{B_M}$ and it is a contraction.
\end{proof}

\begin{teo}\label{existenceanalyticprob}
Let $\varphi^0$ be an admissible initial parametrization. There exists a positive radius $M$ and a positive time $T$ such that the special flow~\eqref{specialflow} of closed
curves  has a unique solution in $C^{\frac{4+\alpha}{4},4+\alpha}\left([0,T]\times\mathbb{S}^1\right)\cap \overline{B_M}$.	
\end{teo}
\begin{proof}
Choosing $M$ and $\overline{T}$ as in
Proposition~\ref{contraction}, for every $T\in (0,\overline{T}]$ the map 
$\mathcal{K}_T:\mathbb{E}_T^0\cap \overline{B_M}\to\mathbb{E}_T^0\cap \overline{B_M}$
is a contraction of the complete metric space $\mathbb{E}_T^0\cap\overline{B_M}$.
Thanks to Banach--Cacciopoli contraction theorem
$\mathcal{K}_T$ has a unique fixed point in $\mathbb{E}_T^0\cap \overline{B_M}$. 
By definition of $\mathcal{K}_T$, an element of $\mathbb{E}_T^0\cap \overline{B_M}$ is a fixed point for $\mathcal{K}_T$ if and only if it is a solution to the special flow~\eqref{specialflow} of closed
curves in
$C^{\frac{4+\alpha}{4},4+\alpha}\left([0,T]\times\mathbb{S}^1\right)\cap \overline{B_M}$.
\end{proof}

\begin{rem}\label{generalizzazioneopen}
In order to prove an existence and uniqueness theorem for the special flow of curves with fixed endpoints subjected 
to natural or clamped boundary conditions, it is enough to repeat the previous arguments 
with some small adjustments.

In the case of Navier boundary condition we 
replace
$\mathbb{E}_T$, $\mathbb{E}_T^0$, $\mathbb{F}_T$ 
and $\mathbb{F}_T^0$ by 
\begin{align*}
  \mathbb{E}_T^1 &:=\left\lbrace\varphi\in C  ^{\frac{4+\alpha}{4},{^{4+\alpha}}}\left([0,T]\times[0,1];\mathbb{R}^2\right)\,:\,
  \partial_x^2\varphi(0)=\partial_x^2\varphi(1)=0,\,\varphi_{\vert t=0}=\varphi^0\right\rbrace \,,\\
  \mathbb{E}_T^{0,1} &:=\left\lbrace
  \varphi\in \mathbb{E}_T^1\,:\,\varphi(t,0)=P, \varphi(t,1)=Q,
  \right\rbrace\,,\\
   \mathbb{F}_T^{1} &:=
 C  ^{\frac{\alpha}{4},{^{\alpha}}}\left([0,T]\times[0,1];\mathbb{R}^2\right)\times 
 (C  ^{\frac{4+\alpha}{4}}\left([0,T];\mathbb{R}^2\right))^2
 \times C^{4+\alpha}([0,1];\mathbb{R}^2)
 \,, \\
 \mathbb{F}_T^{0,1} &:=
  C  ^{\frac{\alpha}{4},{^{\alpha}}}\left([0,T]\times[0,1];\mathbb{R}^2\right)
 \times \{P\}\times \{Q\}\times \{\varphi^0\}\,, 
\end{align*}
where by $P,Q \in \R^2$.
In this case we introduce the linear operator
$$
\mathcal{L}_T(\varphi):=\left(\partial_t\varphi+\frac{2}{\vert\partial_x\varphi^0\vert^4}\partial_x^4\varphi,
\varphi_{\vert x=0},\varphi_{\vert x=1},\varphi_{\vert t=0}\right)\,.
$$
This modification allows us to treat the linear boundary
conditions $\partial_x^2\varphi(0)=\partial_x\varphi(1)=0$
and the affine ones $\varphi(t,0)=P$, 
$\varphi(t,1)=Q$.

In the case of clamped boundary conditions instead
we have to take into account four vectorial affine boundary conditions.
We modify 
the affine space $\mathbb{E}_T^0$ into 
\begin{equation*}
  \mathbb{E}_T^{0,2} :=\left\lbrace
    \varphi\in \mathbb{E}_T\,:\,\varphi(t,0)=P, \varphi(t,1)=Q,
    \partial_x\varphi(t,0)=\tau^0, \partial_x\varphi(t,1)=\tau^1,
  \varphi_{\vert t=0}=\varphi^0\right\rbrace\,,
\end{equation*}
and 
\begin{align*}
       \mathbb{F}_T^{2} &:=
 C  ^{\frac{\alpha}{4},{^{\alpha}}}\left([0,T]\times[0,1];\mathbb{R}^2\right)\times 
 \left(C  ^{\frac{4+\alpha}{4}}\left([0,T];\mathbb{R}^2\right)\right)^2 \times \left(C  ^{\frac{3+\alpha}{4}}\left([0,T];\mathbb{R}^2\right)\right)^2
 \times C^{4+\alpha}([0,1];\mathbb{R}^2)\,, \\
 \mathbb{F}_T^{0,2} &:= C  ^{\frac{\alpha}{4},{^{\alpha}}}\left([0,T]\times[0,1];\mathbb{R}^2\right)
 \times \{P\}\times \{Q\}\times\{\tau^0\}
 \times \{\tau^1\} \times\{\varphi^0\}\,.
\end{align*}
Finally the operator $\mathcal{L}_T$
in this case is 
$$
\mathcal{L}_T(\varphi):=
\left(\partial_t\varphi+\frac{2}{\vert\partial_x\varphi^0\vert^4}\partial_x^4\varphi, 
\varphi_{\vert x=0},\varphi_{\vert x=1},
\partial_x\varphi_{\vert x=0},
\partial_x\varphi_{\vert x=1},
\varphi_{\vert t=0}\right)\,.
$$
\end{rem}

\begin{rem}\label{generalizzazionenetwork}
Differently from the case of endpoints of order one,
at the  multipoints of higher order we impose also
non linear boundary conditions (quasilinear or even fully non linear). 
Treating these terms is then harder: it is necesssary
to linearize both the main equation and the boundary operator.

Consider for instance the case of the elastic flow of a network
composed of $N$ curves that meet at two junction, both of order $N$ and
subjected to natural boundary conditions.
The concurrency condition and the second order condition are already linear.
Instead the third order condition is of the form
\begin{equation*}\label{thirdorder}
\sum_{i=1}^N\frac{1}{\vert \partial_x \varphi^i\vert^3}
\left\langle \partial_x^3 \varphi,\nu^i\right\rangle \nu^i 
+h^i(\partial_x \varphi^i)=0\,,
\end{equation*}
where we 
omit the dependence on $(t,y)$ with $y\in\{0,1\}$.
The linearized version of the highest order term in the third order condition is:
\begin{align}\label{definizione-b}
-&\sum_{i=1}^N\frac{1}{\vert \partial_x \varphi^{0,i}\vert^3}
\left\langle \partial_x^3 \varphi,\nu_0^i\right\rangle \nu_0^i \nonumber\\
=&
-\sum_{i=1}^N\frac{1}{\vert \partial_x \varphi^{0,i}\vert^3}
\left\langle \partial_x^3 \varphi,\nu_0^i\right\rangle \nu_0^i
+\sum_{i=1}^N\frac{1}{\vert \partial_x \varphi^i\vert^3}
\left\langle \partial_x^3 \varphi,\nu^i\right\rangle \nu^i 
+h^i(\partial_x \varphi^i)=:b(\varphi)
\,,
\end{align}
where we denoted by $\nu_0$ the unit normal vector of the initial datum $\varphi^0$.
Then, instead of~\eqref{linearisedproblem}, the linearized system associated to the special flow  is 
\begin{equation}\label{lyntheta}
\begin{cases}
\begin{array}{ll}
\partial_t\varphi^i(t,x)+\frac{2}{\vert\partial_x \varphi^{0,i}(x)\vert^4}\partial_x^4\varphi^i(t,x)&=f^i(t,x)\\
\varphi^{i}(t,y)-\varphi^{j}(t,y)&=0 \\
\partial_x^2\varphi^i(t,y)&=0 \\
-\sum_{i=1}^N\frac{1}{\vert \partial_x \varphi^{0,i}(y)\vert^3}
\left\langle \partial_x^3 \varphi(t,y),\nu_0^i\right\rangle \nu_0^i(y)&=b(t,y)\\
\varphi^{i}(0,x)&=\psi^{i}(x)\\
\end{array}
\end{cases}\,,
\end{equation}
for $i, j\in\{1,\ldots,N\}$, $j\neq i$, $t\in [0,T]$, $x\in [0,1]$, $y\{0,1\}$. 

The spaces introduced in Definition~\ref{linearspaces} and~\ref{affinespaces}
are replaced by
\begin{align*}
\mathbb{E}_T 
&=\big\{\varphi\in C^{\frac{4+\alpha}{4}, 4+\alpha}\left([0,T]\times [0,1]; (\mathbb{R}^2)^N\right)\;
\text{such that for}\, i,j\in\{1, \ldots, N\}, t\in [0,T], \\
&\qquad  y\in\{0,1\}\,\text{it holds}\; \varphi^i(t,y)=\varphi^j(t,y), \partial_x^2\varphi^i(t,y)=0\big\}
\,,\\
\mathbb{F}_T&
= C^{\frac{\alpha}{4}, \alpha}\left([0,T]\times [0,1]; (\mathbb{R}^2)^N\right)
\times \left(C^{1+\alpha}\left([0,T]; \mathbb{R}^2\right)\right)^2\times C^{4+\alpha}\left([0,1];(\mathbb{R}^2)^N\right)\,,\\
\mathbb{E}^0_T&=\{\varphi\in 
\mathbb{E}_T\,\text{such that}\,\varphi\vert_{ t=0}=\varphi^0
\}\,,\\
\mathbb{F}^0_T&
	=C^{\frac{\alpha}{4},{^{\alpha}}}\left([0,T]\times[0,1];(\mathbb{R}^2)^N\right)
\times \left(C^{1+\alpha}\left([0,T]; \mathbb{R}^2\right)\right)^2
\times\{\varphi^0\}
\,.
\end{align*}

The operator $\mathcal{L}_{T}:\mathbb{E}_T\to \mathbb{F}_T$ becomes
\begin{equation*}
\mathcal{L}_{T}(\varphi)
:=\left(\partial_t\varphi+\frac{2}{\vert\partial_x\varphi^0\vert^4}\partial_x^4\varphi, 
-\sum_{i=1}^N\frac{1}{\vert \partial_x \varphi^{0,i}(y)\vert^3}
\left\langle \partial_x^3 \varphi(t,y),\nu_0^i\right\rangle \nu_0^i(y),\varphi_{\vert t=0}\right)\,,
\end{equation*}
and the operator that encodes the non--linearities of the problem is
$\mathcal{N}_{T}:\mathbb{E}^0_T\to \mathbb{F}^0_T$ that maps $\varphi$
into the triple $(N^1_{T}(\gamma),N^2_{T}(\gamma),\gamma\vert_{t=0})$
with
\begin{align*}
N^1_{T}:
\begin{cases}
\mathbb{E}^0_T \to C  ^{\frac\alpha4,{^{\alpha}}}([0,T]\times[0,1];\mathbb{R}^2)\\
\varphi \mapsto
f(\varphi)\,,
\end{cases}\\
N^2_{T}:
\begin{cases}
\mathbb{E}^0_T \to C  ^{1+\alpha}([0,T]\times[0,1];\mathbb{R}^2)\\
\varphi \mapsto
b(\varphi)\,,
\end{cases}
\end{align*}
where the functions $f(\varphi):=f(\partial_x^4\varphi,\partial_x^3\varphi,\partial_x^2\varphi,\partial_x\varphi)$ 
and $b(\varphi):=b(\partial_x^3\varphi,\partial_x^2\varphi,\partial_x\varphi)$
are defined in~\eqref{definizione-f} and in~\eqref{definizione-b}.
The map $\mathcal{K}$ will be defined accordingly.
We do not here describe the details concerning the solvability of the linear system,
as well as  the proof of the 
contraction property of $\mathcal{K}$ and 
we refer to~\cite[Section 3.4.1]{GaMePl1}.

\end{rem}

\subsection{Parabolic smoothing}

When dealing with parabolic problems, it is natural 
to investigate the regularization
of the solutions of the flow.
More precisely, we claim that the following holds.

\begin{prop}\label{parabolicsmoth}
Let $T>0$ and 
$\varphi_0=(\varphi^1_0, \ldots,\varphi^N_0)$ be an admissible initial parametrization (possibly) with endpoints of order one 
and (possibly) with junctions of different orders 
$m\in \mathbb{N}_{\geq 2}$. 
Suppose that $\varphi_{t\in [0,T]}$,
$\varphi=(\varphi^1, \ldots,\varphi^N)$ is a solution in 
$\mathbb{E}_T$ to the special flow 
in the time interval $[0,T]$ with initial datum $\varphi_0$. 
Then the solution $\varphi$ is smooth for positive times in the sense that
$$
\varphi\in C^\infty\left([\varepsilon,T]\times[0,1];
(\mathbb{R}^2)^N\right)
$$
 for every $\varepsilon\in(0,T)$.
\end{prop}

We give now a sketch of proof of this fact 
in the case of closed curves.
Basically, it is possible to prove the result in two different ways: with the so--called Angenent's parameter trick~\cite{angenent,angen3,daprato-grisvard} 
or making use of the 
classical theory of linear parabolic equations~\cite{solonnikov2}.

\medskip

\noindent\textit{Sketch of the proof.} 
For the sake of notation let
\begin{align*}
\mathcal{A}(\gamma)
=-2\frac{\partial_x^4 \gamma}{\left|\partial_x\gamma\right|^{4}}
+12\frac{\partial_x^3 \gamma\left\langle \partial_x^2 \gamma,\partial_x\gamma\right\rangle }{\left|\partial_x\gamma\right|^{6}}
+5\frac{\partial_x^2 \gamma\left|\partial_x^2 \gamma\right|^{2}}{\left|\partial_x\gamma\right|^{6}}
+8\frac{\partial_x^2 \gamma\left\langle \partial_x^3 \gamma,\partial_x\gamma\right\rangle }
{\left|\partial_x\gamma\right|^{6}}
-35\frac{\partial_x^2 \gamma\left\langle \partial_x^2 \gamma,\partial_x\gamma\right\rangle ^{2}}
{\left|\partial_x\gamma\right|^{8}}
+\mu\frac{\partial_x^2 \gamma}{\left|\partial_x\gamma\right|^{2}}\,.
\end{align*}
Then the motion equation reads $\partial_t\gamma=\mathcal{A}(\gamma)$.
We consider the map
\begin{equation*}
    G:
\begin{cases}
(0,\infty)\times \mathbb{E}_T\to C^{4+\alpha}(\mathbb{S}^1;\mathbb{R}^2)\times
C^{\frac{\alpha}{4},\alpha}([0,T]\times\mathbb{S}^1;\R^2)\\
(\lambda,\gamma)\to 
\left(\gamma_{\vert t=0}-\gamma_0,\partial_t\gamma-\lambda A(\gamma)\right)
\end{cases}
\end{equation*}
We notice that if we take $\lambda=1$ and 
$\gamma=\varphi$ the solution of the special flow we get 
$G(1,\varphi)=0$.
The Fr\'{e}chet derivative 
$\delta G(1,\varphi)(0,\cdot):\mathbb{E}_T\to  C^{4+\alpha}(\mathbb{S}^1;\mathbb{R}^2)\times C^{\frac{\alpha}{4},\alpha}([0,T]\times\mathbb{S}^1;\R^2)$ is given by
\begin{equation*}
    \delta G(1,\varphi)(0,\gamma)=\left(\gamma_{\vert t=0},
\partial_t\gamma+\frac{2}{\vert \partial_x\partial\varphi\vert^4}\partial_x^4 \gamma+F_\varphi(\gamma)\right)
\end{equation*}
where $F_\varphi$ is linear in $\gamma$,
where $\partial_x^3 \gamma,\partial_x^2 \gamma$ and $\partial_x\gamma$
appears and the coefficients 
are depending of $\partial_x\varphi,\partial_x^2\varphi,\partial_x^3\varphi$
and $\partial_x^4\varphi$.
The computation to write in details the 
Fr\'{e}chet derivative  is rather long and we 
do not write it here.
Since the time derivative appears only as $\partial_t\gamma$
and it is not present in $A(\gamma)$, 
formally one can follow the computations of 
Section~\ref{sec:SecondVariation}.

It is possible to prove that $\delta G(1,\varphi)(0,\cdot)$
is an isomorphism.
This is equivalent to show that given any 
$\psi \in C^{4+\alpha}\mathbb{S}^1;\mathbb{R}^2)$ and 
$f\in C^{\frac{\alpha}{4},\alpha}([0,T]\times \mathbb{S}^1;\mathbb{R}^2)$ the system
\begin{equation*}
  \begin{cases}
\partial_t\gamma(t,x)+\frac{2}{\vert \partial_x\varphi(t,x)\vert^4}\partial_x^4 \gamma(t,x)+F(\gamma)
=f(t,x)\\
\gamma(0,x)=\psi(x)
\end{cases}  
\end{equation*}
has a unique solution.

Then the implicit function theorem implies the existence 
of a neighbourhood $(1+\varepsilon,1-\varepsilon)\subseteq (0,\infty)$, a neighbourhood $U$ of $\varphi$
in $\mathbb{E}_T$ and a function $\Phi:(1+\varepsilon,1-\varepsilon) \to U$ with $\Phi(1)=0$ and
\begin{equation*}
    \{(\lambda,\gamma)\in (1+\varepsilon,1-\varepsilon) \times U: G(\lambda,\gamma)=0\}
    =\{(\lambda, \Phi(\lambda)):\lambda\in (1+\varepsilon,1-\varepsilon)\}\,.
\end{equation*}

Given $\lambda$ close to $1$ consider 
\begin{equation*}
    \varphi_{\lambda}(t,x):=\varphi(\lambda t, x)\,,
\end{equation*}
where $\varphi$, as before, is a solution to the special flow.
This satisfies $G(\lambda,\varphi_{\lambda})=0$. 
Moreover by uniqueness $\varphi_{\lambda}=\Phi(\lambda)$.
Since $\Phi$ is smooth, this shows that $\varphi_{\lambda}$ is a smooth function of $\lambda$ with values in $\mathbb{E}_T$.
This implies 
\begin{equation*}
    t\partial_t\varphi
    =\partial_\lambda (\varphi_{\lambda})_{\vert \lambda=1}
    \in \mathbb{E}_T
\end{equation*}
from which we gain regularity in time of the solution 
$\varphi$. 

Then using the fact that $\varphi$ is a solution to the special flow and the structure of the motion equation of the special flow it is possible to increase regularity also in space. 

We can then start a bootstrap to obtain that the solution
is smooth for every positive time. 

\medskip

Alternatively we can 
show inductively that there exists $\alpha\in(0,1)$ such that for all $k\in\mathbb{N}$ and $\varepsilon\in(0,T)$,
\begin{equation*}
\varphi\in C^{\frac{2k+2+\alpha}{4},2k+2+\alpha}\left([\varepsilon,T]\times\mathbb{S}^1;\mathbb{R}^2\right)\,.
\end{equation*}
The case $k=1$ is true because 
$\varphi\in C^{\frac{4+\alpha}{4},4+\alpha}\left([0,T]\times\mathbb{S}^1;\mathbb{R}^2\right)$ by Theorem~\ref{existenceanalyticprob}.

Now assume that the assertion holds true for some $k\in\mathbb{N}$ and consider any $\varepsilon\in(0,T)$. 
Let $\eta\in C_0^\infty\left((\frac{\varepsilon}{2},\infty);\mathbb{R}\right)$ be a cut--off function with $\eta\equiv 1$ on $[\varepsilon,T]$. 
By assumption, 
$$
\varphi\in C^{\frac{2k+2+\alpha}{4},2k+2+\alpha}\left(\left[\varepsilon,T\right]\times\mathbb{S}^1;\mathbb{R}^2\right)\,,
$$ 
and thus 
it is straightforward to check that the function $g$ defined by
\begin{equation*}
(t,x)\mapsto g(t,x):=\eta(t)\varphi(t,x)
\end{equation*}
lies in  $C^{\frac{2k+2+\alpha}{4},2k+2+\alpha}\left(\left[0,T\right]\times\mathbb{S}^1;\mathbb{R}^2\right)$.
Moreover $g$ satisfies a parabolic problem of the following form: for all $t\in (0,T)$, $x\in \mathbb{S}^1$:
\begin{equation}\label{systemcutoff}
	\begin{cases}
	\begin{array}{ll}
	\partial_t g(t,x)+\frac{2}{\vert\partial_x\varphi(t,x)\vert^4}\partial_x^4 g(t,x)
	+f\left(\partial_x\varphi,\partial_x^2\varphi,\partial_x g,\partial_x^2 g,\partial_x^3 g\right)(t,x)&=\eta'(t)\varphi(t,x)\,,	\\
	g(0,x)&=0\,.
	\end{array}
	\end{cases}
\end{equation}
The lower order terms in the motion equation are given by
	\begin{align*}
	f\left(\partial_x\varphi,\partial_x^2\varphi,\partial_x g,\partial_x^2 g,\partial_x^3 g\right)(t,x)=&-12 \frac{\left\langle\partial_x^2\varphi,\partial_x\varphi\right\rangle}{\left\lvert\partial_x\varphi\right\rvert^6}\partial_x^3 g-8\frac{\partial_x^2\varphi}{\vert\partial_x\varphi\vert ^6}\left\langle \partial_x^3 g,\partial_x\varphi\right\rangle\\	&-\left(5\frac{\lvert\partial_x^2\varphi\rvert ^2}{\lvert\partial_x\varphi\rvert^6}-35\frac{\left\langle\partial_x^2\varphi,\partial_x\varphi\right\rangle^2}{\vert\partial_x\varphi\vert^8}+\mu\frac{1}{\vert\partial_x\varphi\vert ^2}\right)\partial_x^2 g\,.
	\end{align*} 
The problem is linear in the components of $g$ and in the highest order term of exactly the same structure as the linear system~\eqref{linearisedproblem} with time dependent coefficients in the motion equation.
The coefficients and the right hand side fulfil the regularity requirements of~\cite[Theorem 4.9]{solonnikov2} in the case $l=2(k+1)+2+\alpha$. As $\eta^{(j)}(0)=0$ for all $j\in\mathbb{N}$, the initial value $0$ satisfies the compatibility conditions of order $2(k+1)+2$ with respect to the given right hand side. Thus~\cite[Theorem 4.9]{solonnikov2} yields that there exists a unique solution to~\eqref{systemcutoff} $g$ with the regularity 
	\begin{equation*}
	g\in C^{\frac{2(k+1)+2+\alpha}{4},2(k+1)+2+\alpha}\left([0,T]\times\mathbb{S}^1;\mathbb{R}^2\right)\,.
	\end{equation*}
This completes the induction as $g=\varphi$ on $[\varepsilon,T]$.

\qed

\subsection{Short time existence and uniqueness}

We conclude this section by proving the local 
(in time) existence and uniqueness result 
for the elastic flow.

As before, we give the proof of this theorem in the case 
of closed curves and then we explain how to adapt it 
in all the other situations.

We remind that a solution of the elastic flow
is unique if it is unique up to reparametrizations.

\begin{teo}[Existence and uniqueness]\label{geomexistence}
Let $\mathcal{N}_0$ be an admissible initial network. Then there exists
a positive time $T$ such that within the time interval $[0,T]$ the
elastic flow of networks admits a unique solution $\mathcal{N}(t)$.
\end{teo}

\begin{proof}
We write a proof for the case of the elastic flow of closed curves.

\textit{Existence.}
Let $\gamma_0$ be an admissible initial closed curve
of class $C^{4+\alpha}([0,1];\mathbb{R}^2)$.
Then $\gamma_0$ is also an admissible initial 
parametrization for the special flow.
By Theorem~\ref{existenceanalyticprob}
there exists a solution of the special flow, that is also a solution of the elastic flow.

\textit{Uniqueness.}
Consider a solution $\gamma_t$ of the elastic flow.
Then we can reparametrize the $\gamma_t$
into a solution to the special flow using
Proposition~\ref{riparametriz}.
Hence uniqueness follows from Theorem~\ref{existenceanalyticprob}.
\end{proof}

We now explain how to prove \emph{existence} of solution to the 
elastic flow of networks.
Differently from the situation of closed curves,
an admissible initial network $\mathcal{N}_0$
admits a parametrization 
$\gamma=(\gamma^1, \ldots,\gamma^N)$ of class
$C^{4+\alpha}([0,1];\mathbb{R}^2)$ that, in general,
is not an admissible initial parametrization
in the sense of Definition~\ref{DEf:admissible-initial-para}.
However it is always possible to reparametrize
each curve $\gamma^i$ by $\psi^i:[0,1]\to[0,1]$
in such a way that 
$\varphi=(\varphi^1, \ldots,\varphi^N)$
with $\varphi^i:=\gamma^i\circ\psi^i$
is an admissible initial parametrization for the special flow.
Then by the suitable modification of
Theorem~\ref{existenceanalyticprob}
there exists a solution to the special flow, that is also a solution of the elastic flow.

Thus all the difficulties lie is proving the existence of
the reparametrizations $\psi^i$.

In all cases we look for $\psi^i:[0,1]\to[0,1]$
with 
$\psi^i(0)=0$, $\psi^i(1)=1$ and $\partial_x\psi^i(x)\neq 0$
for every $x\in [0,1]$. 
We now list all possible further conditions a certain
$\psi^i$ has to fulfill at $y=0$ or $y=1$ in the different possible situations. It will then be clear that such reparametrizations $\psi^i$ exist.

\begin{itemize}
\item  If $\gamma(y)$ is an endpoint of order one with Navier boundary conditions
(namely $\gamma(y)=P$, $\boldsymbol{\kappa}(y)=0$), then $\psi(y)$ needs to satisfy
the following conditions:
\begin{equation*}
\begin{cases}
\partial_x\psi(y)=1\\
\partial_x^2\psi(y)=-\left\langle
\frac{\partial_x\gamma(y)}{\vert\partial_x\gamma(y)\vert},
\frac{\partial_x^2 \gamma(y)}{\vert\partial_x\gamma(y)\vert}\right\rangle=:a(y)\\
\partial_x^3\psi (y)=0\\
\partial_x^4\psi(y)=-\frac{1}{\vert\partial_x\gamma(y)\vert^5}\left\langle
\frac{\partial_x\gamma(y)}{\vert\partial_x\gamma(y)\vert},
\frac{\partial_x^4 \gamma(y)}{\vert\partial_x\gamma(y)\vert}
+6a(y)\frac{\partial_x^3 \gamma(y)}{\vert\partial_x\gamma(y)\vert}
+3a^2(y)\frac{\partial_x^2 \gamma(y)}{\vert\partial_x\gamma(y)\vert}
\right\rangle=:-\frac{1}{\vert\partial_x\gamma(y)\vert^5}b(y)\,.
\end{cases}
\end{equation*}    
Indeed, with such a request, we have
$\varphi(y)=\gamma(\psi(y))=\gamma(y)=P$
and 
\begin{align*}
\partial_x^2\varphi(y)
&=\partial_x^2 \gamma(\psi(y))(\partial_x\psi(y))^2
+\partial_x\gamma(\psi(y))\partial_x^2\psi(y)\\
&=\partial_x^2 \gamma(y)+\partial_x\gamma(y)\left(-\left\langle
\frac{\partial_x\gamma(y)}{\vert\partial_x\gamma(y)\vert},
\frac{\partial_x^2 \gamma(y)}{\vert\partial_x\gamma(y)\vert}\right\rangle\right)=\vert\partial_x\gamma\vert^2\boldsymbol{\kappa}(y)=0\,.
\end{align*}
Moreover $\overline{T}(y)=0$. Indeed
\begin{align*}
\overline{T}
&=-2\left\langle
\frac{\partial_x^4\varphi(y)}{\vert\partial_x\varphi(y)\vert^4},
\frac{\partial_x\varphi(y)}{\vert\partial_x\varphi(y)\vert}
\right\rangle\\
&=-2\left\langle
\frac{\partial_x^4 \gamma(y)+6\partial_x^3 \gamma(y)a(y)
+3\partial_x^2 \gamma(y)a^2(y)+\partial_x\gamma(y)\partial_x^4\psi(y)}{\vert
\partial_x\gamma(y)\vert^4},
\frac{\partial_x\gamma(y)}{\vert\partial_x\gamma(y)\vert}
\right\rangle\\
&=-2\frac{1}{\vert\partial_x\gamma(y)\vert^4}b(y)
+2\frac{1}{\vert\partial_x\gamma(y)\vert^5}
\left\langle b(y)\partial_x\gamma(y),\frac{\partial_x\gamma(y)}{\vert\partial_x\gamma(y)\vert}
\right\rangle\\
&=-2\frac{1}{\vert\partial_x\gamma(y)\vert^4}b(y)
+2\frac{1}{\vert\partial_x\gamma(y)\vert^5}b(y)
\left\langle
\frac{\partial_x\gamma(y)}{\vert\partial_x\gamma(y)},\frac{\partial_x\gamma(y)}{\vert\partial_x\gamma(y)\vert}
\right\rangle\vert\partial_x\gamma(y)\vert=0\,.
\end{align*}
\item If $\gamma(y)$ is an endpoint order one where clamped boundary conditions are imposed ($\gamma(y)=P$, 
$\frac{\partial_x\gamma(y)}{\vert \partial_x\gamma(y)\vert}=\tau^*$ with 
$\tau^*$ a unit vector) 
we require $\psi(y)$ to fulfill
\begin{equation*}
\begin{cases}
\partial_x\psi(y)=\frac{1}{\vert\partial_x\gamma(y)\vert}\\
\partial_x^2\psi(y)=0\\
\partial_x^3\psi(y)=0\\
\partial_x^4\psi(y)=b(y)\,.
\end{cases}
\end{equation*}
with 
$b(y)=\left\langle 
\frac{\partial_x^4 \gamma(y)}{\vert\partial_x\gamma(y)\vert^4}
-6\frac{\partial_x^3 \gamma(y)}{\vert\partial_x\gamma(y)\vert^3}\left\langle \frac{\partial_x^2 \gamma(y)}{\vert\partial_x\gamma(y)\vert^2},\tau^*\right\rangle -\frac{5}{2}\frac{\partial_x^2 \gamma(y)\left|\partial_x^2 \gamma(y)\right|^{2}}{\left|
\partial_x\gamma(y)\right|^{6}}
-4\frac{\partial_x^2 \gamma(y)}{\vert\partial_x\gamma(y)\vert^2}\left\langle \frac{\partial_x^3 \gamma(y)}{\vert\partial_x\gamma(y)\vert^3},\tau^*\right\rangle \right.$ 
$
\left. 
+\frac{35}{2}\frac{\partial_x^2 \gamma(y)}{\vert\partial_x\gamma(y)\vert^2}\left\langle \frac{\partial_x^2 \gamma(y)}{\vert\partial_x\gamma(y)\vert^2},\tau^*\right\rangle ^{2}
-\frac{\mu}{2}\frac{\partial_x^2 \gamma(y)}{\vert\partial_x\gamma(y)\vert^2},
\frac{\partial_x\gamma(y)}{\vert\partial_x\gamma(y)\vert}\right\rangle$.
So that $\varphi(y)=\gamma(\psi(y))=\gamma(y)=P$,
$\partial_x\varphi(y)=\partial_x\gamma(\psi(y))\partial_x\psi(y)
=(\partial_x\gamma(y))(\frac{1}{\vert\partial_x\gamma\vert})=\tau^*$,
and $\overline{T}(y)=0$.

\item Suppose instead that $\gamma^{p_1}(y_1)=\ldots=\gamma^{p_m}(y_m)$
is a multipoint of order $m$ with natural boundary conditions.
Then each curve is paramatrized by $\gamma^{p_i}\in C^{4+\alpha}([0,1];\mathbb{R}^2)$ and the network
$\mathcal{N}_0$ satisfies the  conditions
ii),  iv) and v) of Definition~\ref{Def:admissible-initial-net}.

The non--degeneracy condition is satisfied because of iv).

By requiring 
\begin{equation*}
\begin{cases}
\partial_x\psi^{p_i}(y_i)=1\\
\partial_x^2\psi^{p_i}(y_i)=a^{p_i}(y_i)\\
\partial_x^3\psi^{p_i}(y_i)=0\,,
\end{cases}
\end{equation*}  
where $a^{p_i}(y_i):=-\left\langle
\frac{\partial_x\gamma^{p_i}(y_i)}{\vert\partial_x\gamma^{p_i}(y_i)\vert},
\frac{\partial_x^2\gamma^{p_i}(y_i)}{\vert\partial_x\gamma^{p_i}(y_i)\vert}\right\rangle$
all the conditions imposed by the system are satisfied.
We have to choose $\partial_x^4\psi^{p_i}(y_i)$
in a manner that implies the fourth order compatibility condition 
\begin{equation}\label{condizionequartordine}
V^{p_1}_\varphi(y_1)\nu^{p_1}_\varphi(y_1)
+\overline{T}^{p_1}_\varphi(y_1)\tau^{p_1}_\varphi(y_1)
=\ldots
=V^{p_m}_\varphi(y_m)\nu^{p_m}_\varphi(y_m)
+\overline{T}^{p_m}_\varphi(y_m)\tau^{p_m}_\varphi(y_m)\,,
\end{equation}
where the subscript $\varphi$ we mean that 
all the quantities in~\eqref{condizionequartordine}
are computed with respect on the 
parametrization $\varphi^{p_i}:=\gamma^{p_i}\circ \psi^{p_i}$.
Notice that the geometric quantities 
$V$, $\nu$ and $\tau$ are invariant under reparametrization, they coincide for  $\varphi^{p_i}$ 
and $\gamma^{p_i}$ and so from now on we omit the subscript.
Condition iv) allows us to consider 
two consecutive unit normal vectors 
$\nu^{p_i}(y_i)$ and $\nu^{p_k}(y_k)$
such that
$\mathrm{span}\{\nu^{p_i}(y_i),\nu^{p_k}(y_k)\}=\R^2$.
Then, by condition v), 
for every $j\in\{1,\ldots,m\}$, $j\neq i$, $j\neq k$
we have
\begin{equation}\label{sommaV}
   \sin\theta^i V^{p_i}(y_i)+\sin\theta^k V^{p_k}(y_k)
+\sin\theta^j V^{p_j}(y_j)=0\,, 
\end{equation}
where $\theta^i$ is the angle between $\nu^{p_k}(y_k)$
and $\nu^{p_j}(y_j)$, 
$\theta^k$ between $\nu^{p_j}(y_j)$
and $\nu^{p_i}(y_i)$
and $\theta^{j}$  between $\nu^{p_i}(y_i)$,
and $\nu^{p_k}(y_k)$
and at most one between $\sin\theta^i$ and $\sin\theta^k$
is equal to zero.
Consider first every curve $\gamma^{p_j}$ with 
$j\in\{1,\ldots,m\}$, $j\neq i$, $j\neq k$
for which  both $\sin\theta^i$ and $\sin\theta^k$
are different from zero, then the conditions
\begin{align}\label{condzioneangolinonzero}
\sin\theta^i\overline{T}_{\varphi}^{p_i}(y_i)
&=\cos\theta^k V^{p_k}(y_k)-\cos\theta^j V^{p_j}(y_j)\nonumber\\
\sin\theta^k\overline{T}_{\varphi}^{p_k}(y_k)
&=\cos\theta^j V^{p_j}(y_j)-\cos\theta^i V^{p_i}(y_i)\nonumber\\
\sin\theta^j\overline{T}_{\varphi}^{p_j}(y_j)
&=\cos\theta^i V^{p_i}(y_i)-\cos\theta^k V^{p_k}(y_k)
\end{align}
combined together with~\eqref{sommaV} imply~\eqref{condizionequartordine} (see~\cite{menzelthesis}
for details).
Instead for all the curves $\gamma^{p_j}$ with 
$j\in\{1,\ldots,m\}$, $j\neq i$, $j\neq k$
for which, for example,  $\sin\theta^i=0$  it is possible to 
prove (see again~\cite{menzelthesis}) that 
\begin{align}\label{condizioneangolozero}
\sin\theta^k V^{p_k}(y_k)
&+\sin\theta^j V^{p_j}(y_j)=0\nonumber\\
\sin\theta^k\overline{T}_{\varphi}^{p_i}(y_i)
&=V^{p_j}(y_j)-\cos\theta^k V^{p_i}(y_i)\nonumber\\
\sin\theta^j\overline{T}_{\varphi}^{p_k}(y_k)
&=V^{p_i}(y_i)-\cos\theta^j V^{p_k}(y_k)\nonumber\\
\sin\theta^k\overline{T}_{\varphi}^{p_j}(y_j)
&=\cos\theta^k V^{p_j}(y_j)- V^{p_i}(y_i)
\end{align}
yielding~\eqref{condizionequartordine}.
One can show that for every 
$i\in\{1, \ldots,m\}$, imposing such requirements (i.e., either~\eqref{sommaV},~\eqref{condzioneangolinonzero}   or~\eqref{condizioneangolozero}) implies that
$\partial_x^4\psi^{p_i}(y_i)$ is uniquely determined.

\item Also the case of a multipoint with clamped boundary conditions can be treated following the arguments 
of the just considered cases of natural 
boundary conditions.
\end{itemize}

To summarise, for every $i\in\{1,\ldots,N\}$
we must prove the existence of 
$\psi^i:[0,1]\to[0,1]$
with $\partial_x\psi^i(x)\neq 0$
for every $x\in [0,1]$ satisfying 
\begin{equation}\label{valorineipunti}
    \begin{cases}
 \psi^i(0)=0\\
 \partial_x\psi^i(0)=c_1\\
 \partial_x^2\psi^i(0)=c_2\\
 \partial_x^3\psi^i(0)=0\\
\partial_x^4  \psi^i(0)=c_3\\
    \end{cases}
\quad \text{and}\quad    
    \begin{cases}
 \psi^i(1)=1\\
 \partial_x\psi^i(1)=c_4\\
\partial_x^2 \psi^i(1)=c_5\\
\partial_x^3 \psi^i(1)=0\\
 \partial_x^4 \psi^i(1)=c_6\\
    \end{cases}
\end{equation}
with $c_1,c_2,c_3$ and $c_4,c_5,c_6$
depending on the type of the endpoint $\gamma^i(0)$
and $\gamma^i(1)$.
The $\psi^i$ can be (roughly) constructed by
choosing $\psi^i$ to be,  
near the points $0$ and $1$, 
the respective fourth Taylor polynomial that is determined by the values of the derivatives 
appearing in~\eqref{valorineipunti}. 
Then one connects the two polynomial graphs by a suitable increasing smooth function.

\medskip

To get \emph{uniqueness} when we let evolve an open curve
or a network, one has to use
Proposition~\ref{riparametriz}. 
We refer to~\cite{DaChPo20, GaMePl1, menzelthesis, Sp17}
for a complete proof.

\begin{rem}
The previous theorem gives a solution of class 
$C^{\frac{4+\alpha}{4},4+\alpha}([0,T]\times [0,1];\mathbb{R}^2)$ whenever the initial datum
is of class $C^{4+\alpha}([0,1];\mathbb{R}^2)$ 
and satisfies all the conditions
listed in Definition~\ref{Def:admissible-initial-net}. 
We can remove the fourth order conditions iii)--iv) setting the problem in Sobolev spaces, with the initial datum in 
$W^{4-4/p,p}([0,1];\mathbb{R}^2)$ with $p\in (5, \infty)$. 
Even in this lower regularity class it is possible to
prove uniqueness of solutions (see~\cite{GaMePl2}), but 
we pay in regularity of the solution, that is merely in 
$W^{1,p}\left((0,T);L^p\left((0,1);\mathbb{R}^2\right)\right)\cap L^p\left((0,T);W^{4,p}\left((0,1);\mathbb{R}^2 \right)\right)$.

With the strategy we presented in this paper it is possible
to get a smooth solution in $[0,T]$ if in addition
the initial datum admits a smooth parametrization
and satisfies the compatibility conditions of any order
(for a complete proof of this result we refer to~\cite{DaChPo20}).
Since the solution of class $C^{\frac{4+\alpha}{4},4+\alpha}$ is unique, a fortiori the smooth solution is unique.
Although a smooth solution is desiderable, asking
for compatibility conditions of any order is a very strong
request. 
\end{rem}

\section{Long time existence}\label{sec:LongTimeExistence}

\begin{dfnz}
A time--dependent family of networks $\mathcal{N}_t$ parametrized by
$\gamma_t=(\gamma^1,\ldots,\gamma^N)$
is a maximal solution to
the  elastic flow with initial datum $\mathcal{N}_0$ in $[0,T)$
if it is a solution in the sense of Definition~\ref{Def:elasticflow} in $(0,\hat{T}]$ for all $\hat{T}<T$,  $\gamma\in C^\infty\left([\varepsilon,T)\times[0,1];
(\mathbb{R}^2)^N\right)$ for all $\varepsilon>0$
	and if there does not exist a smooth solution 
	$\widetilde{\mathcal{N}}_t$ in $(0,\widetilde{T}]$ with $\widetilde{T}\geq T$ and such that $\mathcal{N}=\widetilde{\mathcal{N}}$
		in $(0,T)$.
\end{dfnz}

If $T=\infty$ in the above definition, $\widetilde{T}\geq T$ is supposed to mean $\widetilde{T}=\infty$.
The maximal time interval of existence of a solution to the elastic flow will be denoted by $[0,T_{\max})$, for $T_{\max}\in (0,+\infty]$.\\

Notice that	the existence of a maximal solution 
is granted by Theorem~\ref{existenceanalyticprob}, Theorem~\ref{geomexistence} and  Proposition~\ref{parabolicsmoth}.

\subsection{Evolution of geometric quantities}

In this section we 
use the following version of the Gagliardo--Nirenberg Inequality which follows from~\cite[Theorem 1]{nirenberg} and a scaling argument.

Let $\eta$ be a smooth regular curve in $\mathbb{R}^2$ 
with finite length $\ell$ and let $u$ be a smooth function defined on $\eta$. 
Then for every $j\geq 1$, $p\in [2,\infty]$ and $n\in\{0,\dots,j-1\}$ we have the estimates
\begin{equation*}
\lVert \partial_s^nu\rVert_{L^p}\leq \widetilde{C}_{n,j,p}\lVert \partial_s^ju\rVert_{L^2}^\sigma\lVert u\rVert_{L^2}^{1-\sigma}+\frac{B_{n,j,p}}{\ell^{j\sigma}}\lVert u\rVert_{L^2}
\end{equation*}
where 
\begin{equation*}
\sigma=\frac{n+1/2-1/p}{j}
\end{equation*}
and the constants $\widetilde{C}_{n,j,p}$ and $B_{n,j,p}$ are independent of $\eta$.
In particular, if  $p=+\infty$, 
\begin{equation}\label{int2}     
{\Vert\partial_s^n u\Vert}_{L^\infty}     \leq \widetilde{C}_{n,j}  
{\Vert\partial_s^j u\Vert}_{L^2}^{\sigma}       
{\Vert u\Vert}_{L^2}^{1-\sigma}+      
 \frac{B_{n,j}}{\ell^{j\sigma}}{\Vert         
u\Vert}_{L^2}\qquad\text{ { with} }\quad \text{ $\sigma=\frac{n+1/2}{j}$.}  
\end{equation}

We notice that in the case of a time--dependent family of curves 
with length equibounded from below 
by some positive value, the Gagliardo--Nirenberg inequality holds with uniform constants.

By the monotonicity of the elastic energy along the flow (Section~\ref{sec:energy-decreases-in-time}), the following result holds.
\begin{cor}\label{curvaturebounded}
Let $\mathcal{N}_t=\bigcup_{i=1}^N \gamma^i_t$ 
be a maximal solution to the elastic flow with initial datum
$\mathcal{N}_0$ in the maximal time interval $[0,T_{\max})$
and let $\mathcal{E}_\mu(\mathcal{N}_0)$
be the elastic energy of the initial datum. 
Then for all $t\in (0,T_{\max})$ it holds
\begin{equation}\label{basic-bound}
\int_{\gamma^i_t} \vert k^i\vert^2\,\mathrm{d}s
\leq \int_{\mathcal{N}_t} \vert k\vert^2\,\mathrm{d}s
\leq \mathcal{E}_\mu(\mathcal{N}_0)\,.
\end{equation}
\end{cor}

Now we consider the evolution in time
of the length of the  curves of the network.

\begin{lemma}\label{lengthbounded}
Let $\mathcal{N}_t=\bigcup_{i=1}^N \gamma^i_t$ 
be a maximal solution to the elastic flow
in the maximal time interval $[0,T_{\max})$
with initial datum $\mathcal{N}_0$ 
and let $\mathcal{E}_\mu(\mathcal{N}_0)$
be the elastic energy of the initial datum. 
Let $\mu^1,\ldots,\mu^N>0$ and $\mu^*:=\min_{i=1,\ldots, N}\mu^i$.
Then for all $t\in (0,T_{\max})$ it holds
\begin{equation}\label{basic-bound2}
\ell(\gamma^i_t)
\leq \mathrm{L}(\mathcal{N}_t)
\leq \frac{1}{\mu^*}\mathcal{E}_\mu(\mathcal{N}_0)\,.
\end{equation}
Furthermore if $\mathcal{N}_t$ is composed of 
a time dependent family of closed curves $\gamma_t$, then for all $t\in (0,T_{\max})$
\begin{equation}\label{GBbound}
    \ell(\gamma_t)\geq \frac{4\pi^2}{\mathcal{E}_\mu(\gamma_0)}\,.
\end{equation}
Suppose instead that $\gamma_t$ 
is a time dependent family of curves subjected either to Navier boundary conditions
or to clamped boundary conditions with 
$\gamma(t,0)=P$ and $\gamma(t,1)=Q$ for every $t\in [0,T_{\max})$.
Then for all $t\in (0,T_{\max})$
\begin{equation}\label{boundlunghezza}
    \ell(\gamma_t)\geq \vert P-Q\vert >0 \;\text{if}\; P\neq Q
    \quad \text{and} \quad \ell(\gamma_t)\geq \frac{\pi^2}{\mathcal{E}_\mu(\gamma_0)}>0 \;\text{if}\; P= Q\,.
\end{equation}
\end{lemma}
\begin{proof}
Formula~\eqref{basic-bound2} is a direct consequence of
Proposition~\ref{energydecreases}.
Suppose $\gamma_t$ is a one--parameter family of single closed curves. Then by Gauss--Bonnet theorem we have
\begin{equation}\label{catena}
    2\pi\leq \int_{\gamma_t} \vert k\vert\,\mathrm{d}s
    \leq \left(\int_{\gamma_t} \vert k\vert^2\,\mathrm{d}s\right)^{1/2}
    \left(\int_{\gamma_t} 1\,\mathrm{d}s\right)^{1/2}
    = \ell(\gamma_t)^{1/2}\left(\int_{\gamma_t} \vert k\vert^2\,\mathrm{d}s\right)^{1/2}\,,
\end{equation}
that combined with~\eqref{basic-bound} gives~\eqref{GBbound}. 
Clearly if $\gamma_t$ 
is composed of a curve with fixed endpoints $\gamma(t,0)=P$
and $\gamma(t,1)=Q$ with $P\neq Q$, then 
$\ell(\gamma_t)\geq \vert P-Q\vert>0$.
Suppose now that $P=Q$.
Then by a generalization of the Gauss--Bonnet
Theorem (see~\cite[Corollary A.2]{danovplu}) 
to not necessarily embedded curves 
with coinciding endpoints it holds
$\int_{\gamma_t}\vert k\vert\,\mathrm{d}s\geq \pi$
and so repeating the chain of inequalities~\eqref{catena}
one gets~\eqref{boundlunghezza}.
\end{proof}

\begin{rem}
 In many situations it seems not possible to generalize the above computations  
in the case of networks  to control the lengths of the curves neither individually nor globally.
At the moment there are no explicit examples of networks whose curves disappear during the evolution,
but we believe in this possible scenario.

Consider for example a sequence of networks composed of three curves 
that meet only at their endpoints in two triple junctions.  In particular, 
suppose that the networks is composed of 
two arcs of circle of radius $1$  and length $\varepsilon$
that meet with a segment (of length $2\sin \tfrac\varepsilon2\sim\varepsilon$)
with angles of amplitude $\frac\varepsilon2$. 
The energy (with $\mu^i=1$ for any $i$) of this network is
$\mathcal{E}_\mu(\mathcal{N}_\varepsilon)=4\varepsilon+2\sin\tfrac\varepsilon2$ and it converges to zero
when $\varepsilon\to 0$.

A similar behavior has been shown by N\"{u}rnberg in the following numerical examples
based on the methods developed in~\cite{BaGaNu12} (see~\cite[Section 5.5]{menzelthesis} for more details). The initial datum is the standard double bubble.

\begin{figure}[H]
\begin{center}
\includegraphics[scale=1]{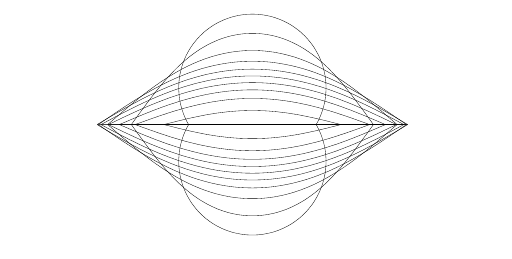}
\caption{A numerical example of a shrinking network. The weighs $\mu^i$ are all equal to $0.2$.}
\end{center}
\end{figure}

First the symmetric double bubble expand and then it starts flattening. 
The length of all the curves becomes smaller and smaller and the same happen to the amplitude of
the angles.  The simulation suggest that the networks shrink to a point in finite time.

In this other example instead only the length of one curve goes to zero and the network
composed of three curve becomes a ``figure eight''.
\begin{figure}[H]
\begin{center}
\includegraphics[scale=1]{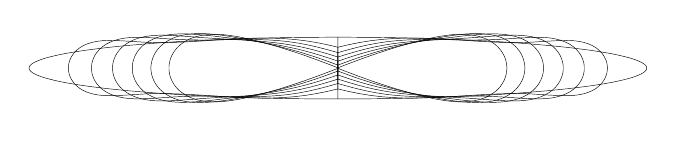}
\caption{A numerical example of a disappearance of one curve. The weighs $\mu^i$ are all equal to $2$.}
\end{center}
\end{figure}
 
\end{rem}

\begin{rem}\label{crescita-lineare}
If some of the weights 
$\mu^i$ of the  definition of the elastic flow are equal to zero, then the $L^2$--norm of the curvature remains bounded, but
the lengths of the network can go to infinity.
However, during the flow of either a single closed
curve or a curve with Navier boundary conditions,
the length of the curve can go to infinity, but not in finite time.
Suppose $\mu=0$, in this case we call the functional $\mathcal{E}_0$.
It holds
\begin{align*}
    \frac{d}{dt}\ell(\gamma_t)
  &=\frac{d}{dt}\int_{\gamma_t}1\,\mathrm{d}s  
  =\int_{\gamma_t} \partial_s T-k V\,\mathrm{d}s
  =T(1)-T(0)+\int_{\gamma_t} 2k\partial_s^2k+k^4\,\mathrm{d}s\\
  &=\int_{\gamma_t} -2\vert\partial_s k\vert^2+k^4\,\mathrm{d}s
+k(1)\partial_sk(1)-k(0)\partial_sk(0)\\
&=\int_{\gamma_t} -2\vert\partial_s k\vert^2+k^4\,\mathrm{d}s\,,
\end{align*}
indeed, in the case of a closed curve $T(1)=T(0)$ and $k(1)\partial_sk(1)=k(0)\partial_sk(0)$, while natural boundary conditions implies $T(1)=T(0)=k(1)=k(0)=0$. The Gagliardo-Nirenberg inequality gives
\begin{align*}
\Vert k	\Vert_4\leq \widetilde{C}\Vert \partial_sk\Vert_2^{\frac{1}{4}} \Vert k\Vert_2^{\frac{3}{4}}+\frac{B}{\ell^{\frac{1}{4}}}\Vert k\Vert_2\leq 
 c\Vert k\Vert_2^{\frac{3}{4}}\left(\Vert \partial_sk\Vert_2^{\frac{1}{4}} +\Vert k\Vert_2^{\frac{1}{4}}\right)
 \leq 2^{{\frac34}} c\Vert k\Vert_2^{\frac{3}{4}}\left(\Vert \partial_sk\Vert_2+\Vert k\Vert_2\right)^{\frac{1}{4}} \,,
\end{align*}
where $c=\max\left\lbrace  \widetilde{C},B/\ell^{\frac{1}{4}}\right\rbrace$. 
Thanks to~\eqref{GBbound} and~\eqref{boundlunghezza}, we know that
$\ell$ is uniformly bounded from below away from zero and thus that constants are independent of the length. Also, as $\|k\|_2 \le C(\mathcal{E}_0(\gamma_0))$, using Young inequality we obtain
\begin{align*}
\Vert k	\Vert_4^4&\leq 
 C\Vert k\Vert_2^3\left(\Vert \partial_sk\Vert_2 +\Vert k\Vert_2\right)
 \leq \varepsilon C \Vert \partial_sk\Vert_2^2 +C(\mathcal{E}_0(\gamma_0),\varepsilon).
\end{align*}
By taking $\varepsilon$ small enough we then conclude
\begin{align*}
    \frac{d}{dt}\ell(\gamma_t)&\leq \int_{\gamma_t} -2\vert\partial_s k\vert^2+k^4\,\mathrm{d}s
    \leq \int_{\gamma_t} -\vert \partial_sk\vert^2\,\mathrm{d}s 
    + C(\mathcal{E}_0(\gamma_0)),
    \end{align*}
thus in both cases
$  \frac{d}{dt}\ell(\gamma_t)\leq C(\mathcal{E}_0(\gamma_0))$
and hence the length grows at most linearly.\\
Unfortunately in the case of clamped curves we are not able to reproduce the same computation because
we cannot get rid of the boundary terms $k(1)\partial_s k(1)$ and $k(0)\partial_s k(0)$. However we are not
aware of examples in which the length of a clamped curve subjected to the $L^2$--gradient flow 
of $\mathcal{E}_0$ blows up in finite time.
\end{rem}

\begin{lemma}\label{stimeintegrali}
Let $\gamma:[0,1]\to\R^2$ be a smooth regular curve.
Then the following estimates hold:
\begin{align}\label{stima}
\int_{\gamma}^{} |\mathfrak{p}_{2j+6}^{j+1}\left(k\right)|\,\mathrm{d}s
\leq \varepsilon \lVert\partial_s^{j+2}k\rVert_{L^2}^2+C(\varepsilon,\ell(\gamma))\left(\lVert k\rVert_{L^2}^2+\lVert k\rVert^{2(2j+5)}_{L^2}\right)\,,\nonumber\\
\int_{\gamma}^{}|\mathfrak{p}_{2j+4}^j\left(k\right)|\,\mathrm{d}s
\leq \varepsilon \lVert\partial_s^{j+1}k\rVert_{L^2}^2
+C(\varepsilon,\ell(\gamma))\left(\lVert k\rVert_{L^2}^2+C\lVert k\rVert^{2(2j+3)}_{L^2}\right)\,,
\end{align}
for any $\varepsilon>0$.
\end{lemma}

\begin{proof}
Every monomial of $\mathfrak{p}_{2j+6}^{j+1}\left(k\right)$ is of the form
$C\prod_{l=0}^{j+1}\left(\partial_s^lk\right)^{\alpha_l}$ 
with $\alpha_l\in\mathbb{N}$ and $\sum_{l=0}^{j+1}\alpha_l(l+1)=2j+6$. 
We define $J:=\{l\in\{0,\dots,j+1\}:\alpha_l\neq 0\}$
and for every $l\in J$  we set
$$
\beta_l:=\frac{2j+6}{(l+1)\alpha_l}\,.
$$
We observe that $\sum_{l\in J}^{}\frac{1}{\beta_l}=1$ and $\alpha_l\beta_l>2$
for every $l\in J$. Thus the H\"older inequality implies
\begin{equation*}
C\int_{\gamma}\prod_{l\in J}^{} (\partial_s^lk)^{\alpha_l}\,\mathrm{d}s
\leq C\prod_{l\in J}^{}\left(\int_{\gamma}\lvert\partial_s^lk\rvert^{\alpha_l\beta_l}\,\mathrm{d}s\right)^{\frac{1}{\beta_l}}
=C\prod_{l\in J}^{}\lVert\partial_s^lk\rVert^{\alpha_l}_{L^{\alpha_l\beta_l}}\,.
\end{equation*}
Applying the Gagliardo--Nirenberg inequality for every $l\in J$ yields for every $i\in\{1,\ldots,j+1\}$
\begin{equation*}
\lVert\partial_s^lk^i\rVert_{L^{\alpha_l\beta_l}}\leq C_{l,j,\alpha_l,\beta_l}\lVert\partial_s^{j+2}k^i\rVert_{L^2}^{\sigma_l}\lVert k^i\rVert_{L^2}^{1-\sigma_l}+\frac{B_{l,j,\alpha_l,\beta_l}}{\ell(\gamma)^{(j+2)\sigma_l}}\lVert k^i\rVert_{L^2}\,
\end{equation*}
where for all $l\in J$ the coefficient $\sigma_l$ is given by
$$
\sigma_l=\frac{l+1/2-1/(\alpha_l\beta_l)}{j+2}\,.
$$
We may choose 
$$
C=\max\left\{C_{l,j,\alpha_l,\beta_l},
\frac{B_{l,j,\alpha_l,\beta_l}}{\ell(\gamma)^{(j+2)\sigma_l}}:l\in J \right\}\,.
$$ 
Since 
the polynomial $\mathfrak{p}_{2j+6}^{j+1}\left(k\right)$
consists of finitely many monomials
of the above type, we can write
\begin{align*}
C\int_{\gamma}^{}\prod_{l\in J}^{}\lvert\partial_s^lk\rvert^{\alpha_l}\,\mathrm{d}s
&\leq C\prod_{l\in J }^{}\lVert\partial_s^lk\rVert^{\alpha_l}_{L^{\alpha_l\beta_l}}\\
&\leq C\prod_{l\in J}^{}\lVert k\rVert^{(1-\sigma_l)\alpha_l}_{L^2}\left(\lVert\partial_s^{j+2}k\rVert_{L^2}+\lVert k\rVert_{L^2}\right)^{\sigma_l\alpha_l}_{L^2}\\
&= C\lVert k\rVert^{\sum_{l\in J}(1-\sigma_l)\alpha_l}_{L^2}\left(\lVert\partial_s^{j+2}k\rVert_{L^2}+\lVert k\rVert_{L^2}\right)^{\sum_{l\in J}\sigma_l\alpha_l}_{L^2}.
\end{align*}
Moreover we have
\begin{align*}
\sum_{l\in J}\sigma_l\alpha_l&
\leq 2-\frac{1}{(j+2)^2}<2\,.
\end{align*}
Applying Young's inequality with $p:=\frac{2}{\sum_{l\in J}\sigma_l\alpha_l}$ and $q:=\frac{2}{2-\sum_{l\in J}^{}\sigma_l\alpha_l}$ we obtain
\begin{align*}
C\int_{\gamma}^{}\prod_{l\in J}^{}\lvert\partial_s^lk\rvert^{\alpha_l}\,\mathrm{d}s
&\leq \frac{C}{\varepsilon}\lVert k\rVert^{2\frac{\sum_{l\in J}(1-\sigma_l)\alpha_l}{2-\sum_{l\in J}^{}\sigma_l\alpha_l}}_{L^2}
+\varepsilon C \left(\lVert\partial_s^{j+2}k\rVert_{L^2}+\lVert k\rVert_{L^2}\right)^{2}_{L^2}
\end{align*}
where
\begin{align*}
2\frac{\sum_{l\in J}(1-\sigma_l)\alpha_l}{2-\sum_{l\in J}^{}\sigma_l\alpha_l}&
=2(2j+5)\,.
\end{align*}
As $C$ depends only on $j$ and the length of the curve, we get choosing $\varepsilon$ small enough 
\begin{align*}
\int_{\gamma}|\mathfrak{p}_{2j+6}^{j+1}|\left(k\right)\,\mathrm{d}s
&\leq\varepsilon \left(\lVert\partial_s^{j+2}k\rVert_{L^2}+\lVert k\rVert_{L^2}\right)^{2}_{L^2}+\frac{C}{\varepsilon}\lVert k\rVert^{2(2j+5)}_{L^2}
\,.
\end{align*}
To conclude it is enough to take  
choose a suitable $\varepsilon>0$.
The second inequality in~\eqref{stima} can be proved
in the very same way.
\end{proof}

\begin{lemma}\label{stimepuntuali}
Let $\gamma:[0,1]\to\R^2$ be a smooth regular curve.
Suppose that $\gamma$ has a fixed endpoint
of order one $\gamma(y)$ with $y\in\{0,1\}$.
Then the following estimates hold:
\begin{align}\label{stimanormainfinito}
|\mathfrak{p}^{j+1}_{2j+5}(k)(y)|\leq \varepsilon\lVert\partial_s^{j+2}k\rVert_{L^2}^2+C(\varepsilon,\ell(\gamma))\left(\lVert k\rVert_{L^2}^2+\lVert k\rVert^{2(2j+5)}_{L^2} \right)\,,\nonumber \\
|\mathfrak{p}^{j+1}_{2j+3}(k)(y)|\leq \varepsilon\lVert\partial_s^{j+2}k\rVert_{L^2}^2+C(\varepsilon,\ell(\gamma))\left(\lVert k\rVert_{L^2}^2+C\lVert k\rVert^{(2j+3)^2}_{L^2}\right)\,,
\end{align}
for any $\varepsilon>0$.
\end{lemma}
\begin{proof}
The term $|\mathfrak{p}^{j+1}_{2j+5}(k)(y)|$
can be controlled by a sum of terms like 
$C\prod_{l=0}^{j+1}\Vert\ders^{l}k\Vert_{L^\infty}^{\alpha_l}$
with $\sum_{l=0}^{j+1}(l+1)\alpha_l=2j+5$.
Then, for every $l\in\{0,\ldots, j+1\}$
we use interpolation inequalities with $p=+\infty$
to obtain
$$
\Vert\ders^lk\Vert_{L^\infty}\leq C_l\left(
    {\Vert\partial_s^{j+2}  k\Vert}_{L^2}^{\sigma_l}
    {\Vert k\Vert}_{L^2}^{1-\sigma_l}+
    {\Vert k\Vert}_{L^2}\right)\,,
$$
with $\sigma_l=\frac{l+1/2}{j+2}$.
Thus 
\begin{align*}
\prod_{l=0}^{j+1}\Vert\ders^{l}k\Vert_{L^\infty}^{\alpha_l}
\leq&\, 
C \prod_{l=0}^{j+1}\left({\Vert\partial_s^{j+2}  k\Vert}_{L^2}
+{\Vert k\Vert}_{L^2}\right)^{\sigma_l\alpha_l}
{\Vert k\Vert}_{L^2}^{(1-\sigma_l)\alpha_l}\\
\leq&\, 
C \left({\Vert\partial_s^{j+2}k\Vert}_{L^2}
    + {\Vert k\Vert}_{L^2}\right)^{\sum_{l=0}^{j+1}\sigma_l\alpha_l}
{\Vert k\Vert}_{L^2}^{\sum_{l=0}^j(1-\sigma_l)\alpha_l}
\end{align*}
with 
\begin{align*}
\sum_{l=0}^{j+1}\sigma_l\alpha_l&
=\,\sum_{l=0}^{j+1}\alpha_l
\frac{l+1/2}{j+2}
=\frac{2j+5-1/2\sum_{l=0}^j\alpha_l}{j+2}\\
&\leq\, \frac{2j+5-1/2\sum_{l=0}^{j+1}\alpha_l(l+1)/(j+2)}{j+2}
\\
&=\, \frac{2j+5 -1 -1/2(j+2)}{j+2}=
2 - \frac{1}{2(j+2)^2}<2\,.
\end{align*}
Then  by Young inequality,
$$
\left({\Vert\partial_s^{j+2}  k\Vert}_{L^2}
+{\Vert k\Vert}_{L^2}\right)^{\sum_{l=0}^{j+1}\sigma_l\alpha_l}
{\Vert k\Vert}_{L^2}^{\sum_{l=0}^{j+1}(1-\sigma_l)\alpha_l}
\leq \varepsilon \left( {\Vert\partial_s^{j+2}  k\Vert}_{L^2}
    + {\Vert k\Vert}_{L^2}\right)^2
+ C {\Vert  k\Vert}_{L^2}^{a^*}
$$
and the last exponent
$a^*=2\frac{\sum_{l=0}^j(1-\sigma_l)\alpha_l}
{2-\sum_{l=0}^j\sigma_l\alpha_l}$
is equal to $2(2j+5)$. Choosing 
a value $\varepsilon>0$ small enough, we get the desired estimate.

Similarly $\mathfrak{p}^{j+1}_{2j+3}(k)(y)$
can be controlled by 
a sum of terms like 
$C\prod_{l=0}^{j+1}\Vert\ders^{l}k\Vert_{L^\infty}^{\alpha_l}$
with $\sum_{l=0}^{j+1}(l+1)\alpha_l=2j+3$.
We can repeat the same proof.
Also in this case $\sum_{l=0}^{j+1}\sigma_l\alpha_l<2$:
indeed 
\begin{align*}
\sum_{l=0}^{j+1}\sigma_l\alpha_l&
=\frac{2j+3-1/2\sum_{l=0}^{j+1}\alpha_l}{j+2}\\
&\leq\, \frac{2j+3-1/2\sum_{l=0}^{j+1}\alpha_l(l+1)/(j+2)}{j+2}
\\
&=\, \frac{2j+3 + 1-1 - \frac{2j+3}{2j+4}}{j+2}=
	2 - \frac{1}{j+2}-\frac{2j+3}{2(j+2)^2}<2\,.
\end{align*}
This time we get that the exponent $a^*\in (\frac{2}{j+5},\frac{(2j+3)^2}{2})$.
We have 
$a^*=\frac{(j+\frac{5}{2})(\sum_{l=0}^{j+1}\alpha_l)-2j-3}{
1+\frac{1}{2}\sum_{l=0}^{j+1}\alpha_l}$.
Because of the properties of the polynomial $\mathfrak{p}_{2j+3}^{j+1}(k)$
we have that $2\leq \sum_{l=0}^{j+1}\alpha_l< 2j+3$.
Then 
$a^*<\frac{(j+\frac{5}{2})(2j+3)-(2j+3)}{1+\frac{1}{2}\sum_{l=0}^{j+1}\alpha_l}
<\frac{(2j+3)^2}{2}$.
Moreover 
$a^*\geq \frac{2(j+\frac{5}{2})-2j-3}{1+\frac{1}{2}(2j+3)}
=\frac{2}{j+5}$.
Now that we have ensured that $a^*$ is bounded from below 
away from zero we can conclude that the desired estimate hold true.
\end{proof}

\begin{lemma}\label{derivate-pari-nulle}
Let $\gamma_t$ be a maximal solution in $[0,T_{\max})$
to the elastic flow of a curve with Navier boundary conditions.
The for all $t\in [0,T_{\max})$ and for all $n\in\mathbb{N}$ 
\begin{equation*}
    \partial_s^{2n}k(t,0)=\partial_s^{2n}k(t,1)=0\,.
\end{equation*}
\end{lemma}
\begin{proof}
We prove the result by induction.
Since we required Navier boundary conditions, apart from the fixed
endpoints 
$\gamma(t,0)=P$ and $\gamma(t,1)=Q$, we also have 
$k(0)=k(1)=0$.
Differentiating in time the first condition we get for $y\in\{0,1\}$
\begin{align*}
   0&=\partial_t \gamma(t,y)= V(t,y)\nu(t,y)+T(t,y)\tau(t,y)\\
   &=(-2\partial_s^2 k(t,y)-k^3(t,y)+\mu k (t,y))\nu(t,y)+T(t,y)\tau(t,y)\,,
\end{align*}
that implies $T(t,y)=0$ and $\partial_s^2 k (t,y)=0$,
and this gives the first step of the induction.
Let $n\in\mathbb{N}$ and suppose that 
$\partial_s^{2n}k(0)=\partial_s^{2n}k(1)=0$ holds for any natural number  $m\leq n$. Then making use of~\eqref{derivativekt} 
 we have 
\begin{align*}
 0=\partial_t\partial_s^{2(n-1)}k(t,y)
&=-2\partial_s^{2(n+1)}k(t,y)
-5k^2\partial_s^{2n}k(t,y)+\mu \partial_s^{2n}k(t,y)-T\partial_s^{2n-1}k(t,y)\\
&+\mathfrak{p}^{2n-1}_{2n+3}(k)(t,y)+\mu \mathfrak{p}^{2n-2}_{2n+1}(k)(t,y)
=-2\partial_s^{2(n+1)}k(t,y)\,,   
\end{align*} 
where we use the induction hypothesis,  $T(t,y)=0$
and the fact that
each monomial of 
$\mathfrak{p}^{2n-1}_{2n+3}(k)$, $\mu \mathfrak{p}^{2n-2}_{2n+1}(k)$
contains at least one term of the form
$\partial_s^{2m}k$.
\end{proof}

\begin{lemma}
Let $\gamma_t$ be a maximal solution in $[0,T_{\max})$
to the elastic flow of a curve with clamped boundary conditions.
The for all $t\in [0,T_{\max})$, $y\in\{0,1\}$
and $n\in\mathbb{N}$
\begin{align}
\partial_s^{4n+2}k(t,y)
=\mathfrak{p}^{4n}_{4n+3}(k)(t,y)
+\mu \mathfrak{p}^{4n}_{4n+1}(k)(t,y)\,,\label{caso4j+2}\\
\partial_s^{4n+3}k(t,y)
=\mathfrak{p}^{4n+1}_{4n+4}(k)(t,y)
+\mu \mathfrak{p}^{4n+1}_{4n+2}(k)(t,y)\,.\label{caso4j+3}
\end{align}
\end{lemma}
\begin{proof}
Consider first the fixed endpoints condition $\gamma(t,0)=P$
and $\gamma(t,1)=Q$.
Differentiating in time we have, for $y\in\{0,1\}$ 
\begin{align*}
   0&=\partial_t \gamma(t,y)= V(t,y)\nu(t,y)+T(t,y)\tau(t,y)\\
   &=(-2\partial_s^2 k(t,y)-k^3(t,y)+\mu k (t,y))\nu(t,y)+T(t,y)\tau(t,y)\,.
\end{align*}
Since both normal and tangential velocity have to be zero,
we get $T(t,y)=0$ and 
$\partial_s^2 k(t,y)=\frac{\mu}{2} k(t,y)-\frac{1}{2}k^3(t,y)$: 
the case $n=0$ of~\eqref{caso4j+2} holds true. 
Fix a certain $n\in\mathbb{N}$, 
suppose that~\eqref{caso4j+2} is true for any natural number
$m\leq n$. 
Then 
\begin{align*}
    0&=\partial_t\left(\partial_s^{4n+2}k(t,y)
    +\mathfrak{p}^{4n}_{4n+3}(k(t,y))+\mu \mathfrak{p}^{4n}_{4n+1}(k(t,y))\right)\\
    &=-2\partial_{s}^{4n+6}k(t,y)	
-T(t,y)\partial_{s}^{4n+3}k(t,y)
+\mathfrak{p}_{4n+7}^{4n+4}\left(k(t,y)\right)+
\mu\,\mathfrak{p}_{4n+5}^{4n+4}(k(t,y))\\
&+\mathfrak{p}_{4n+7}^{4n+4}(k(t,y))+T(t,y)\mathfrak{p}_{4n+4}^{4n+1}(k(t,y))
+\mu\mathfrak{p}_{4n+5}^{4n+4}(k(t,y))+\mu T(t,y)\mathfrak{p}_{4n+2}^{4n+1}(k(t,y))\\
&=-2\partial_{s}^{4n+6}k(t,y)	
+\mathfrak{p}_{4n+7}^{4n+4}\left(k(t,y)\right)+
\mu\,\mathfrak{p}_{4n+5}^{4n+4}(k(t,y))\,.
\end{align*}

We prove also~\eqref{caso4j+3} by induction.
We consider the clamped boundary condition $\tau(t,0)=\tau_0$
and $\tau(t,1)=\tau_1$. In this case differentiating in time we obtain
\begin{align*}
    0=\partial_t \tau(t,y)=(\partial_s V(t,y)+T(t,y)k(t,y))\nu(t,y)\,,
\end{align*}
that implies $0=\partial_s V(t,y)
=2\partial_s^3 k(t,y)+3k^2(t,y)\partial_sk(t,y)+\mu \partial_sk(t,y)$.
The induction step follows as in the previous situation.
\end{proof}

\begin{rem}
It is not possible to 
generalize~\eqref{caso4j+2} and~\eqref{caso4j+3}
to $\partial_s^nk=\mathfrak{p}^{n-2}_{n+1}(k)+\mu \mathfrak{p}^{n-2}_{n-1}(k)$, where $n\in\mathbb{N}$
is arbitrary. 
Indeed we do not have any particular 
request on $k$ and $\partial_s k$ at the boundary,  
we cannot produce the step $n=0$ of the induction.
\end{rem}

\begin{lemma}\label{lemmastimaderivatecurvatura}
Let $\gamma_t$ be a maximal solution to the elastic flow 
either of closed curves or of an open curve
subjected to Navier boundary conditions with initial datum
$\gamma_0$ in the maximal time interval $[0,T_{\max})$.
Let $\mathcal{E}_\mu(\gamma_0)$
be the elastic energy of the initial curve $\gamma_0$. 
Then for all $t\in (0,T_{\max})$ and $j\in\mathbb{N}$,
$j\geq 1$ it holds
\begin{equation}\label{stimaderivatecurvatura}
\frac{d}{dt}\int \frac{1}{2}\vert\partial_s^jk\vert^2\,\mathrm{d}s
\leq C(j,\mathcal{E}_\mu(\gamma_0))\,.
\end{equation}
\end{lemma}
\begin{proof}
Using~\eqref{derivativekt}
we obtain
\begin{align}\label{inizio-stime}
\frac{d}{dt}&\int_{\gamma_t} \frac{1}{2}\vert\partial_s^jk\vert^2\,\mathrm{d}s
=\int_{\gamma_t} \partial_s^jk\partial_t\partial_s^j k
+\frac{1}{2}\vert\partial_s^jk\vert^2(\partial_sT-kV)\,\mathrm{d}s
\nonumber\\
&=\int_{\gamma_t} \partial_s^jk
\left\lbrace 
-2\partial_{s}^{j+4}k	
-5k^2\partial_s^{j+2}k
+\mu\,\partial_s^{j+2}k
+\mathfrak{p}_{j+5}^{j+1}\left(k\right)+
\mu\,\mathfrak{p}_{j+3}^{j}(k)+T\partial_s^{j+1}k
\right\rbrace\,\mathrm{d}s\nonumber\\
&+\int_{\gamma_t} \frac{1}{2}\vert\partial_s^jk\vert^2(\partial_sT-kV)\,\mathrm{d}s\,.
\end{align}
We begin by considering the terms involving the tangential velocity: 
we have 
\begin{equation}\label{parte-tang}
\int_{\gamma_t}T\partial_s^jk\partial_s^{j+1}k
+\frac{1}{2}\partial_sT(\partial_s^jk)^2\,\mathrm{d}s=\frac12 \left(
T(t,1)(\partial_s^jk(t,1))^2-T(t,0)(\partial_s^jk(t,0))^2\right)=0\,,
\end{equation}
since for a closed curve $T(t,1)(\partial_s^jk(t,1))^2=T(t,0)(\partial_s^jk(t,0))^2$
and in the case of Navier boundary conditions $T(t,1)=T(t,0)=0$.

Integrating twice by parts the term 
$\int -2\partial_s^jk\partial_s^{j+4}k\,\mathrm{d}s$ and once
$\int \mu\partial_s^jk\partial_s^{j+2}k-5k^2\partial_s^j k\partial_s^{j+2}k\,\mathrm{d}s$
we have
\begin{align}\label{integrazione}
\frac{d}{dt}\int \frac{1}{2}\vert\partial_s^jk\vert^2\,\mathrm{d}s
&=\int -2\vert\partial_s^{j+2}k\vert^2-\mu\vert\partial_s^{j+1}k\vert^2
+\mathfrak{p}_{2j+6}^{j+1}\left(k\right)
+\mu \mathfrak{p}_{2j+4}^j\left(k\right) \,\mathrm{d}s\,.
\end{align}
Also in the case of open curves with Navier boundary conditions  there is no boundary contribution thanks to Lemma~\ref{derivate-pari-nulle}.
Combing~\eqref{integrazione} together with~\eqref{stima} one has
\begin{equation}\label{ultimo-passaggio}
\frac{d}{dt}\int \frac{1}{2}\vert\partial_s^jk\vert^2\,\mathrm{d}s
\leq \int -\vert\partial_s^{j+2}k\vert^2-\frac{\mu}{2}\vert\partial_s^{j+1}k\vert^2
 \,\mathrm{d}s+ C\Vert k\Vert^{2(2j+5)}_2+C\Vert k\Vert_{L^2}^2\leq C(j,\mathcal{E}_\mu(\gamma_0))\,,
\end{equation}
where in the last inequality we used~\eqref{basic-bound}.
\end{proof}

The case of clamped boundary conditions is more tricky.

\begin{lemma}\label{lemmastimaderivatecurvatura-clamped}
Let $\gamma_t$ be a maximal solution to the elastic flow 
subjected to  clamped boundary conditions
with initial datum
$\gamma_0$ in the maximal time interval $[0,T_{\max})$.
Let $\mathcal{E}_\mu(\gamma_0)$
be the elastic energy of the initial curve $\gamma_0$. 
Then for all $t\in (0,T_{\max})$ and $n\in\mathbb{N}$,
$n\geq 0$ it holds
\begin{equation*}
\frac{d}{dt}\int \frac{1}{2}\vert\partial_s^{4n}k\vert^2\,\mathrm{d}s
\leq C(n,\mathcal{E}_\mu(\gamma_0))\,.
\end{equation*}
\end{lemma}
\begin{proof}
Consider the equality~\eqref{inizio-stime}. 
As in the case of Navier boundary conditions,
also in the case of clamped boundary conditions  $T(t,1)=T(t,0)=0$ and thus we have~\eqref{parte-tang}.
Then integrating by parts the terms
$\int -2\partial_s^jk\partial_s^{j+4}k\,\mathrm{d}s$ and 
$\int \mu\partial_s^jk\partial_s^{j+2}k-5k^2\partial_s^j k\partial_s^{j+2}k\,\mathrm{d}s$
appearing in~\eqref{inizio-stime} we obtain
\begin{align*}
\frac{d}{dt}\int \frac{1}{2}\vert\partial_s^jk\vert^2\,\mathrm{d}s
&=\int -2\vert\partial_s^{j+2}k\vert^2-\mu\vert\partial_s^{j+1}k\vert^2
+\mathfrak{p}_{2j+6}^{j+1}\left(k\right)
+\mu \mathfrak{p}_{2j+4}^j\left(k\right) \,\mathrm{d}s\nonumber\\
&-2\partial_s^jk\partial_s^{j+3}k+2\partial_s^{j+1}k\partial_s^{j+2}k+\mu \partial_s^jk\partial_s^{j+1}k-5k^2\partial_s^jk\partial_s^{j+1}k\,\Big\vert^1_0 \nonumber \\
&=\int -2\vert\partial_s^{j+2}k\vert^2-\mu\vert\partial_s^{j+1}k\vert^2
+\mathfrak{p}_{2j+6}^{j+1}\left(k\right)
+\mu \mathfrak{p}_{2j+4}^j\left(k\right) \,\mathrm{d}s\nonumber\\
&-2\partial_s^jk\partial_s^{j+3}k
+2\partial_s^{j+1}k\partial_s^{j+2}k+\mathfrak{p}^{j+1}_{2j+5}(k)+\mu \mathfrak{p}^{j+1}_{2j+3}(k)\,\Big\vert^1_0 \,.\nonumber
\end{align*}
Suppose $j=4n$ with $n\in\mathbb{N}$. Then, using~\eqref{caso4j+2} and~\eqref{caso4j+3} 
\begin{align*}
 \partial_s^jk\partial_s^{j+3}k
&= \partial_s^{4n}k\partial_s^{4n+3}k
=\partial_s^{4n}k\mathfrak{p}^{4n+1}_{4n+4}(k) 
+\mu \partial_s^{4n}k \mathfrak{p}^{4n+1}_{4n+2}(k)
=\mathfrak{p}^{j+1}_{2j+5}(k) +\mu \mathfrak{p}^{j+1}_{2j+3}(k)\,,\\  
\partial_s^{j+1}k\partial_s^{j+2}k&=
\partial_s^{4n+1}k\partial_s^{4n+2}k
=\partial_s^{4n+1}k\mathfrak{p}^{4n}_{4n+3}(k)+\mu   \partial_s^{4n+1}k\mathfrak{p}^{4n}_{4n+1}(k)
=\mathfrak{p}^{j+1}_{2j+5}(k) +\mu \mathfrak{p}^{j+1}_{2j+3}(k)\,.
\end{align*}
So, for $j=4n$ with $n\in\mathbb{N}$,
combing~\eqref{integrazione} 
together with~\eqref{stima} 
and~\eqref{stimanormainfinito} 
one has
\begin{align*}
\frac{d}{dt}\int \frac{1}{2}\vert\partial_s^jk\vert^2\,\mathrm{d}s
&\leq \int -\frac{1}{2}\vert\partial_s^{j+2}k\vert^2
-\frac{\mu}{4}\vert\partial_s^{j+1}k\vert^2
 \,\mathrm{d}s+ C\Vert k\Vert^{2(2j+5)}_2+
C\Vert k\Vert^{(2j+3)^2}_{L_2}+C\Vert k\Vert_{L^2}^2\\
&\leq C(j,\mathcal{E}_\mu(\gamma_0))\,.
\end{align*}
\end{proof}

We pass now to networks. 
In the case of clamped boundary conditions, apart from the monotonicity of the energy, 
geometric estimates on the derivative of the curvature are not know (see also Section~\ref{Openprob}). 

To describe the results contained in~\cite{DaChPo19,GaMePl2} for networks with  junctions
subjected to natural boundary conditions we need a preliminary definition.
\begin{dfnz}
	We say that at a junction 
	of order $m\in\mathbb{N}_{\geq 2}$ the 
	uniform non--degeneracy condition is satisfied if 
	there exists $\rho>0$ such that
	\begin{equation}\label{nondegeneracy}
	\inf_{t\in[0,T_{\max})}\max_{i=1,\ldots,m}\left\{\left\vert\sin\alpha^i(t)\right\vert\right\}\geq \rho\,,
	\end{equation}
	where with $\alpha^i$ we denote the angles 
	between two consecutive tangent vectors of the 
	curves concurring at the junction. 
\end{dfnz}

Then~\cite[Proposition 6.15]{GaMePl2} reads as follow:
\begin{lemma}\label{estimatetogether}
Let $\mathcal{N}(t)$ be a maximal solution to the elastic flow with initial datum $\mathcal{N}_0$ in the maximal time interval $[0,T_{\max})$ and let $\mathcal{E}_\mu(\mathcal{N}_0)$ be the elastic energy
of the initial network. 
Suppose that at all the junctions (of any order
$m\in\mathbb{N}_{\geq 2}$) we impose natural boundary
conditions, 
for $t\in(0,T_{\max})$ the lengths of all the
curves of the network $\mathcal{N}(t)$
are uniformly bounded away from zero and 
that the uniform non--degeneracy condition is satisfied.
Then for all $t\in (0,T_{\max})$ it holds
\begin{equation}\label{estimate}
\frac{d}{dt}\int_{\mathcal{N}_t}
\left\lvert\partial_s^2k\right\rvert^2 \,\mathrm{d}s
\leq C(\mathcal{E}_\mu(\mathcal{N}_0))\,.
\end{equation}
\end{lemma}

This lemma is proved in the case of network with triple junctions, but with an accurate inspection
of the proof one notices that it can be adapted to junctions of any order $m\in\mathbb{N}_{\geq 2}$.
The structure of the proof of~\cite[Proposition 6.15]{GaMePl2} is the same of 
Lemma~\ref{lemmastimaderivatecurvatura}  and Lemma~\ref{lemmastimaderivatecurvatura-clamped}. 
The main difference is the treatment of the boundary terms, which is more intricate.
The uniform bound from below on the length of each curve is needed in Lemma~\ref{stimeintegrali}
and Lemma~\ref{stimepuntuali}, that are both used in the proof. The uniform non--degeneracy conditions 
allows us to express the tangential velocity at the boundary 
in function of the normal velocity (see Remark~\ref{tang-in-funz-norm}).
As in Lemma~\ref{lemmastimaderivatecurvatura}  and Lemma~\ref{lemmastimaderivatecurvatura-clamped},
in~\cite[Proposition 6.15]{GaMePl2} the tangential velocity is arbitrary. 

To generalize~\cite[Proposition 6.15]{GaMePl2} from $\partial_s^2k$ to  $\partial_s^{2+4j}k$
with $j\in\mathbb{N}$ we must also require that
the tangential velocity
in the interior of the curves
is a linear interpolation
between the tangential velocity at the junction  (given in terms of the normal velocity)
and zero (for further details we refer the reader to~\cite{DaChPo19}).

\subsection{Long time existence}

\begin{teo}[Global Existence]\label{teo:GlobalExistence}
Let $\mu>0$ and let 
$$
\gamma\in C^{\frac{4+\alpha}{4},4+\alpha}([0,T_{\max})\times [0,1]))
\cap C^{\infty}([\varepsilon, T_{\max})\times [0,1])
$$ 
be a maximal solution to the elastic flow
of a single curve (either closed, or with fixed endpoints in
$\mathbb{R}^2$)
in the maximal time interval $[0,T_{\max})$ 
with admissible initial datum $\gamma_0\in C^{4+\alpha}([0,1])$.
Then $T_{\max}=+\infty$. In other words, the solution
exists globally in time.
\end{teo}

\begin{proof}
Suppose by contradiction that $T_{\max}$ is finite.
In the whole time interval $[0,T_{\max})$
the length of the curves $\gamma_t$ is uniformly bounded 
from above and from below away from zero and 
the $L^2$--norm of the curvature is uniformly bounded. 

If the curve is closed or subjected to Navier 
boundary conditions, then \eqref{stimaderivatecurvatura} tells us that
for every $t_1,t_2\in (\varepsilon, T_{\max})$, $t_1<t_2$
\begin{equation*}
    \int_{\gamma_{{t}_2}} \vert \partial_s^j k\vert^2\,\mathrm{d}s
    -\int_{\gamma_{t_1}} \vert \partial_s^j k\vert^2\,\mathrm{d}s
\leq     C\mathcal{E}_{\mu}(\gamma_0)\left(t_2-t_1\right)
\leq     C\mathcal{E}_{\mu}(\gamma_0)\left(T_{\max}-\varepsilon\right)\,.
\end{equation*}
The estimate implies 
$\partial_s^j k\in L^{\infty}\left((\varepsilon, T_{\max});L^2\right)$\,.
Instead in the case of clamped boundary condition
we get 
\begin{equation*}
    \int_{\gamma_{t_2}} \vert \partial_s^4 k\vert^2\,\mathrm{d}s
    -\int_{\gamma_{t_1}} \vert \partial_s^4 k\vert^2\,\mathrm{d}s
\leq     C\mathcal{E}_{\mu}(\gamma_0)\left(t_2-t_1\right)
\leq     C\mathcal{E}_{\mu}(\gamma_0)\left(T_{\max}-\varepsilon\right)\,,
\end{equation*}
because Lemma~\ref{lemmastimaderivatecurvatura-clamped}
holds true only when $j$ is a multiple of $4$.
Again this estimate gives $\partial_s^4 k\in L^{\infty}\left((\varepsilon, T_{\max});L^2\right)$\,.
Using Gagliardo--Nirenberg 
inequality for all $t\in[0,T_{\max})$
we get
\begin{align*}
    \Vert \partial_s k(t)\Vert_{L^2}
    \leq C_1\Vert \partial_s^{4}k(t)\Vert_{L^2}^\sigma
    \Vert k(t)\Vert_{L^2}^{1-\sigma}+
    C_2 \Vert k(t)\Vert_{L^2}\leq C(\mathcal{E}_{\mu}(\gamma_0))\,,\\
       \Vert \partial_s^{2}k(t)\Vert_{L^2}
\leq C_1\Vert \partial_s^{4}k(t)\Vert_{L^2}^\sigma
    \Vert k(t)\Vert_{L^2}^{1-\sigma}+
    C_2 \Vert k(t)\Vert_{L^2}\leq C(\mathcal{E}_{\mu}(\gamma_0))\,,\\
       \Vert \partial_s^{3}k(t)\Vert_{L^2}
\leq C_1\Vert \partial_s^{4}k(t)\Vert_{L^2}^\sigma
    \Vert k(t)\Vert_{L^2}^{1-\sigma}+
    C_2 \Vert k(t)\Vert_{L^2}\leq C(\mathcal{E}_{\mu}(\gamma_0))\,,
\end{align*}
with constants independent on $t$.

Hence in all cases (closed curves, either Navier of clamped boundary conditions), 
since $\tau \in L^{\infty}\left((\varepsilon, T_{\max});
L^\infty\right)$, by interpolation we obtain
$$k,\partial_s k,\partial_s^2 k  \in L^{\infty}\left((\varepsilon, T_{\max});
L^\infty\right)\,.$$
Putting this observation together with the formulas
\begin{align*}
\boldsymbol{\kappa}&=    k\nu\,,\\
\partial_s\boldsymbol{\kappa}&=    \partial_s k\nu-k^2\tau\,,\\
\partial_s^2\boldsymbol{\kappa}&= (\partial_s^2 k-k^3)\nu
-3k\partial_s k\tau\,,\\
\partial_s^3\boldsymbol{\kappa}&= 
(\partial_s^3 k  -6k^2\partial_s k)\nu
-(4k\partial_s^2 k+ 3\partial_s k^2)\tau\,,
\end{align*}
we get $\boldsymbol{\kappa},\partial_s \boldsymbol{\kappa},
\partial_s^2 \boldsymbol{\kappa} 
\in L^{\infty}\left((\varepsilon, T_{\max});L^\infty\right)$ and 
$\partial_s^3\boldsymbol{\kappa}
\in L^{\infty}\left((\varepsilon, T_{\max});L^2\right)$.
We also get 
$(\partial_t\gamma)^\perp \in L^{\infty}\left((\varepsilon, T_{\max});L^\infty\right)$. 

We reparametrize $\gamma(t)$ into $\tilde{\gamma}(t)$
with the property $\vert \partial_x\tilde{\gamma}(t,x)\vert =\ell(\tilde{\gamma}(t))$ for every $x\in [0,1]$
and for all $t\in [0,T_{\max})$.
This choice in particular implies
\begin{equation*}
    0<c\leq \sup_{t\in [0,T_{\max}), x\in [0,1]} \vert
    \partial_x\tilde{\gamma}(t,x)\vert \leq C<\infty\,.
\end{equation*}
We make use of the relations
\begin{equation*}
    \boldsymbol{\kappa}(t,x)
    =\frac{\partial_x^2\tilde{\gamma}(t,x)}{\ell (\tilde{\gamma}(t))^2}\,,\quad
      \partial_s\boldsymbol{\kappa}(t,x)
    =\frac{\partial_x^3\tilde{\gamma}(t,x)}{\ell (\tilde{\gamma}(t))^3}\,,\quad
      \partial_s^2\boldsymbol{\kappa}(t,x)
    =\frac{\partial_x^4\tilde{\gamma}(t,x)}{\ell (\tilde{\gamma}(t))^4}\,,\quad
      \partial_s^3\boldsymbol{\kappa}(t,x)
    =\frac{\partial_x^5\tilde{\gamma}(t,x)}{\ell (\tilde{\gamma}(t))^5}\,
\end{equation*}
to get 
\begin{equation*}
    \int_0^1 \frac{\vert \partial_x^2\tilde{\gamma}(t,x) \vert^2 }
    {\ell (\tilde{\gamma}(t))^3}\,\mathrm{d}x
    =\int _{\tilde{\gamma}_t}
    \vert \boldsymbol{k}\vert ^2\,\mathrm{d}s
    =\int _{\tilde{\gamma}_t}
    \vert k\vert ^2\,\mathrm{d}s \leq \mathcal{E}_\mu(\gamma_0)\,,
\end{equation*}
and
\begin{equation*}
    \int_0^1 \frac{\vert \partial_x^5\tilde{\gamma}(t,x) \vert^2 }
    {\ell (\tilde{\gamma}(t))^9}\,\mathrm{d}x
    =\int _{\tilde{\gamma}_t}
    \vert \partial_s^3\boldsymbol{k}\vert ^2\,\mathrm{d}s\leq C(\mathcal{E}_\mu(\gamma_0))\,.
\end{equation*}
These estimates allows us to conclude that  
$\partial_x^2\tilde{\gamma},
\partial_x^3\tilde{\gamma},
\partial_x^4\tilde{\gamma}\in L^{\infty}\left((\varepsilon, T_{\max});L^\infty((0,1))\right)$ and 
$\partial_x^5\tilde{\gamma}\in L^{\infty}\left((\varepsilon, T_{\max});L^2((0,1))\right)$. 
Moreover 
$(\partial_t \tilde{\gamma})^\perp=(\partial_t\gamma)^\perp 
\in L^{\infty}\left((\varepsilon, T_{\max});L^\infty((0,1))\right)$
implies $\tilde{\gamma}\in L^{\infty}\left((\varepsilon, T_{\max});L^\infty((0,1))\right)$.
Then there exists $\gamma_\infty(\cdot)$ limit as $t\to T_{\max}$
of $\tilde{\gamma}(t,\cdot)$ together with 
the limit of its derivatives till $5$-th order.

The curve $\gamma_\infty$ is an admissible initial curve,
indeed it belongs to $C^{4+\alpha}([0,1])$ and in the case
fixed endpoint are present, by continuity of 
$k$ and $\partial_s^2 k$ it holds
$$
2\partial_s^2k_{\infty}(0)+k_{\infty}^3(0)-\mu k_{\infty}(0)
=2\partial_s^2k_{\infty}(1)+k_{\infty}^3(1)-\mu k_{\infty}(1)
=0\,.
$$
Then there exists an elastic flow 
$\overline{\gamma}\in C^{\frac{4+\alpha}{4},4+\alpha}([T_{\max}, T_{\max}+\delta]\times [0,1]))$
with initial datum $\gamma_\infty$ in the time interval $[T_{\max}, T_{\max}+\delta]$ with $\delta>0$.
We reparametrize $\overline{\gamma}$ in $\hat{\gamma}$
with the property 
$\vert\partial_x\hat{\gamma}(t,x)\vert=\ell(\hat{\gamma}(t))$
for every $x\in [0,1]$ and $t\in [T_{\max}, T_{\max}+\delta]$.
Then for every $x\in [0,1]$ 
$$
\lim_{t\nearrow T_{\max}} \tilde{\gamma}(t,x)=\gamma_{\infty}(x)
=\lim_{t\searrow T_{\max}}\hat{\gamma}(t,x)
$$
and also for $j\in\{1,2,3,4,5\}$
$$
\lim_{t\nearrow T_{\max}} \partial_x^j\tilde{\gamma}(t,x)=\partial_x^j\gamma_{\infty}(x)
=\lim_{t\searrow T_{\max}}\partial_x^j\hat{\gamma}(t,x)\,.
$$
Then 
$$
\lim_{t\nearrow T_{\max}} \partial_t\tilde{\gamma}(t,x)
=\lim_{t\searrow T_{\max}}\partial_t\hat{\gamma}(t,x)\,.
$$
Thus we can define 
$$
g:[0,T_{\max}+\delta]\to\mathbb{R}^2, g(t, [0,1]):=
\begin{cases}
\tilde{\gamma}\quad\text{for}\; t\in[0,T_{\max})\\
\hat{\gamma}\quad\text{for}\;t\in[T_{\max}, T_{\max}+\delta]\,.
\end{cases}
$$
solution of the elastic flow  in $[0, T_{\max}+\delta]$.
This contradicts the maximality of $\gamma$.
\end{proof}

\begin{rem}
In the case of the elastic flow either of 
closed curve or of curves with Navier boundary conditions
with the help of Remark~\ref{crescita-lineare}
it is possible to generalize the above result to the case
$\mu\geq 0$ (see~\cite{poldenthesis,DzKuSc02}).
\end{rem}

\begin{rem}
In Lemma~\ref{lemmastimaderivatecurvatura}
and Lemma~\ref{lemmastimaderivatecurvatura-clamped}
we derive estimates for every derivative 
of $k$. 
The above proof shows that it is enough to get 
the estimate of Lemma~\ref{lemmastimaderivatecurvatura}
for $j=1,2,3$ 
and of  Lemma~\ref{lemmastimaderivatecurvatura-clamped}
for $n=1$.
\end{rem}

At the moment for general networks 
we are able to get the following partial result:

\begin{teo}[Long time behavior, \cite{DaChPo19, GaMePl2}]\label{thm:LongTimeNetworks}
Let $\mathcal{N}_0$ be a geometrically admissible initial network. 
Suppose that $\left(\mathcal{N}(t)\right)$ is a 
maximal solution to
the elastic flow with initial datum $\mathcal{N}_0$ 
in the maximal time interval $[0,T_{\max})$ with $T_{\max}\in (0,\infty)\cup\{\infty\}$. 
Suppose that at each junction we impose
Navier boundary conditions.
Then $$T_{\max}=+\infty\,,$$ or at least one of the following happens:
\begin{itemize} 
\item[(i)] the inferior limit as $t\nearrow T_{\max}$ of the length of at least one curve of $\mathcal{N}(t)$ is zero;
\item[(ii)]
at one of the junctions it occurs that
$\liminf_{t\nearrow T_{\max}}\max_i\left\{\left\vert\sin\alpha^i(t)\right\vert\right\}=0$, 
where $\alpha^i(t)$  are the angles 
formed by the unit tangent vectors of the curves
concurring at a junction.
\end{itemize}
\end{teo}

\section{Asymptotic behavior}\label{sec:SmoothConvergence}

In this section we collect results on the asymptotic convergence of the elastic flow, that is, we analyze the possibility that the solutions have a limit as times goes to $+\infty$ and such limit is an {\em elastica}, i.e., a critical point of the energy. The first step in this direction is the proof of the subconvergence of the flow, that consists in the fact that the solution converges to an elastica along an increasing sequence of times, up to reparametrization and translation in the ambient space. We present such a result for the flow of closed curves and for a single curve with Navier or clamped boundary conditions.

\begin{prop}[Subconvergence]\label{Subconvergence}
Let $\mu>0$ and let $\gamma_t$ be a solution
of the elastic flow of closed curves in $[0,+\infty)$ with initial datum 
$\gamma_0$.
Then, up to subsequences, reparametrization,
and translations, the curve $\gamma_t$ converges to
an elastica as $t\to \infty$.
\end{prop}

\begin{proof}
We remind that thanks to~\eqref{basic-bound} and~\eqref{GBbound} for every $t\in [0,+\infty)$
the length $\ell (\gamma_t)$ is uniformly bounded 
from above and from below away from zero by constants
that depends only on the energy of the initial datum
$\mathcal{E}_{\mu}(\gamma_0)$ and on $\mu$. 
We can rewrite inequality~\eqref{ultimo-passaggio}
as
\begin{equation*}
\frac{d}{dt}\int_{\gamma(t)} \frac{1}{2}\vert\partial_s^jk\vert^2\,
\mathrm{d}s
+ \int_{\gamma(t)} \vert\partial_s^{j+2}k\vert^2\,
\mathrm{d}s
\leq C(\mathcal{E}_\mu(\gamma_0))\,.
\end{equation*}
Using interpolation inequalities (with constants $c_1,c_2$
independent of time) for every $j\in\mathbb{N}$
we get 
\begin{align*}
    \frac{d}{dt}\int_{\gamma(t)} \vert\partial_s^jk\vert^2\,
\mathrm{d}s
+ \int_{\gamma(t)} \vert\partial_s^{j}k\vert^2\,
\mathrm{d}s
&\leq 
\frac{d}{dt}\int_{\gamma(t)} \vert\partial_s^jk\vert^2\,
\mathrm{d}s
+ c_1 \int_{\gamma(t)} \vert\partial_s^{j+2}k\vert^2\,
\mathrm{d}s
+c_2 \int_{\gamma(t)} \vert k \vert^2\,
\mathrm{d}s\\
&\leq C(\mathcal{E}_\mu(\gamma_0))\,.
\end{align*}
By comparison we obtain 
\begin{equation}\label{eq:FinalBound}
    \int_{\gamma(t)} \vert\partial_s^jk\vert^2\,\mathrm{d}s\leq 
\int_{\gamma_0} \vert\partial_s^jk\vert^2\,\mathrm{d}s
+C(\mathcal{E}_\mu(\gamma_0)).
\end{equation}
Hence by Sobolev embedding we get that $\|\pa_s^jk\|_{L^\infty}$ is uniformly bounded in time, for any $j\in \N$.
By Ascoli--Arzel\'a Theorem, up to subsequences and reparametrizations, there exists the limit
$\lim_{i\to \infty} \partial_s^j k_{t_i}=:\partial_s^j k_{\infty}$ uniformly on $[0,1]$, for some sequence of times $t_i\to+\infty$.
Thus, for a suitable choice of a sequence of points
$p_i\in\R^2$ such that
$\gamma(t_i, 0)-p_{t_i}$ is the origin $O$ of $\mathbb{R}^2$, 
the sequence of curves $\gamma(t_i,\cdot)-p_i$ 
converges (up to reparametrizations) smoothly to a limit curve 
$\gamma_{\infty}$ with $\gamma(0)=O$
and $0<c\leq \ell (\gamma_{\infty})\leq C <\infty$.

It remains to show that the limit curve is a stationary point
for the energy $\mathcal{E}_\mu$.
Let $V\coloneqq \pa_t\gamma^\perp = -2 (\pa_s^\perp)^2 \boldsymbol{\kappa} - |\boldsymbol{\kappa}|^2\boldsymbol{\kappa}+\mu \boldsymbol{\kappa}$ and $v(t):=\int_{\gamma(t)} |V|^2\,\mathrm{d}s$.
By Proposition~\ref{energydecreases} we know that
$\partial_t E(\gamma (t,\cdot))=-v(t)$,
thus 
\begin{equation}\label{derivatekinLinfinito}
   \int _{0}^\infty v(t)\,\mathrm{d}t= \mathcal{E}_\mu(\gamma(0, \cdot))-\lim_{t\to \infty}\mathcal{E}_\mu(\gamma(t, \cdot))\leq  \mathcal{E}_\mu(\gamma_0) ,
\end{equation}
and then $v\in L^1(0,\infty)$. By~\eqref{eq:FinalBound} we can estimate
$$
\vert \partial_t{v}(t)\vert \leq C(\mu,\gamma_0)\,.
$$
Therefore $v(t)\to 0$ as $t\to+\infty$ and then the limit curve is a critical point.
\end{proof}

Notice that in the previous proof we cannot hope for a (uniform in time) bound on the supremum norm of $\gamma$ itself. In fact all the parabolic estimates are obtained on the curvature vector of the evolving curve. This means that we are not yet able to exclude
that the flow leaves any compact set as $t\to \infty$.

\begin{prop}[Subconvergence]\label{Subconvergence--fixed--endpoints}
Let $\mu>0$ and let
$(\gamma_t$ be a solution in $[0,+\infty)$
of the elastic flow of open curves subjected either to
Navier or clamped boundary conditions.
Then, up to subsequences and reparametrization, 
the curve $\gamma_t$ converges to
an elastica as $t\to \infty$.
\end{prop}

\begin{proof}
Whatever boundary condition we consider 
the length of $\gamma_t$ is bounded from above
by~\eqref{basic-bound2} and from below away from zero
by~\eqref{boundlunghezza}.
Furthermore, since the endpoints are fixed,
the evolving curve will remain in a ball of center 
$\gamma(t,0)=P$ and radius $2\mathcal{E}_\mu(\gamma_0)$
and so for every 
$t\in [0,T_{\max})$ it holds
$\sup_{x\in [0,1]}\vert \gamma(t,x)\vert <C$
with $C$ independent of time.

Consider the case of Navier boundary conditions:
as in the proof of Proposition~\ref{Subconvergence}
we obtain that $\| \partial_s^j k\|_{L^\infty}$ is uniformly bounded in time, for every $j\in\mathbb{N}$.
In the clamped case instead we have that
$\partial_s^j k\in L^{2}(0,\infty);L^{\infty})$
only when $j$ is a multiple of $4$.
However, by interpolation, we get that $\| \partial_s^j k\|_{L^\infty}$ is uniformly bounded in time,
for every $j\in\mathbb{N}$.
Therefore, up to subsequences and reparametrization, the curves $\gamma(t_i,\cdot)$
converges smoothly to a  limit curve $\gamma_\infty$ for some sequence of times $t_i\to+\infty$.
One can show that $\gamma_\infty$ is a critical point exactly as in Proposition~\ref{Subconvergence}.
\end{proof}

By the same methods, it is also possible to prove the following subconvergence result.

\begin{prop}[\cite{DaChPo19}]
Let $\mathcal{N}_0$ be a geometrically admissible initial network composed of three curves that meet
at a triple junction. 
Suppose that $\mathcal{N}_t$ is a 
maximal solution to
the elastic flow with initial datum $\mathcal{N}_0$ 
in the maximal time interval $[0,+\infty)$.
Suppose that along the flow the length 
of the three curves is uniformly bounded from below,
at the junction Navier boundary condition is imposed and
the uniform non--degeneracy conditions
is fulfilled. 
Then it is possible to find a sequence 
of time $t_i\to \infty$, 
such that the networks $\mathcal{N}_{t_i}$ converge, after an appropriate reparametrization, to a critical point for the energy $\mathcal{E}_\mu$ with Navier boundary conditions.
\end{prop}

\medskip

From now on we restrict ourselves to the case of closed curves and we want to improve the subconvergence result of~\Cref{Subconvergence} into full convergence of the flow. More precisely, we want to prove that the solution of the elastic flow of closed curves admits a full limit as time increases and such a limit is a critical point.

Recall that when we say that $\gamma:[0,1]\to\R^2$ is a smooth closed curve, periodic conditions at the endpoints are understood. More precisely, it holds that $\partial_x^k\gamma(0)=\partial_x^k\gamma(1)$ for any $k\in \mathbb{N}$. Therefore we can write that a closed smooth curve is just a smooth immersion $\gamma:\mathbb{S}^1\to \mathbb{R}^2$ of the unit circle $\mathbb{S}^1\simeq[0,2\pi]/_\sim$. In this section we shall adopt this notation as a shortcut for recalling that periodic boundary conditions are assumed. Moreover, for sake of simplicity, we assume that the constant weight on the length in the functional $\mathcal{E}_\mu$ is $\mu=1$ and we write
$\mathcal{E}$.

Now we can state the result about the full convergence of the flow.

\begin{teo}[{Full convergence~\cite{MaPo20, Po20}}]\label{FullConvergence}
Let $\gamma:[0,+\infty)\times \SS^1\to \R^2$ be the smooth solution of the elastic flow with initial datum $\gamma(0,\cdot)=\gamma_0(\cdot)$, that is a smooth closed curve.

Then there exists a smooth critical point $\gamma_\infty$ of $\mathcal{E}$ such that $\gamma(t,\cdot)\to \gamma_\infty(\cdot)$ in $C^m(\SS^1)$ for any $m\in\N$, up to reparametrization. In particular, the support $\gamma(t,\SS^1)$ stays in a compact set of $\R^2$ for any time.
\end{teo}

Let us remark again that sub-convergence of a flow is a consequence of the parabolic estimates that one can prove. However this fact is not sufficient for proving the existence of a full limit as $t\to+\infty$ of $\gamma(t,\cdot)$ in any reasonable notion of convergence. We observe that sub-convergence does not even prevent from the possibility that for different sequences of times $t_j,\tau_j \nearrow+\infty$ and points $p_j,q_j\in \R^2$, the curves $\gamma(t_j,\cdot)-p_j$ and $\gamma(\tau_j,\cdot)-q_j$ converge to different critical points. The sub-convergence clearly does not imply that the flow remains in a compact set for all times either. This last fact, that is, uniform boundedness of the flow in $\R^2$, is not a trivial property for fourth order equation like the elastic flow. Indeed, such evolution equation lacks of maximum principle and therefore uniform boundedness of the flow in $\R^2$ cannot be deduced by comparison with other solutions.

However, all these properties will follow at once as the proof of~\Cref{FullConvergence} will be completed, that is, as we prove the full convergence of the flow.

\medskip

The proof of~\Cref{FullConvergence} is based on several intermediate results. The strategy is rather general and the main steps can be sum up as follows. First we need to set up a suitable functional analytic framework in which the energy functional and its first and second variations are considered in a neighborhood of an arbitrary critical point. In this setting we prove some variational properties that are needed in order to produce a \L ojasiewicz--Simon gradient inequality. Such an inequality estimates the difference in energy between the critical point and points in a neighborhood of it in terms of the operator norm of the first variation. As the first variation functional coincide with the driving velocity of the flow, this furnishes an additional estimate. Such an estimate will be finally applied to the flow as follows.
Since we know that the flow subconverges, it passes arbitrarily close to critical points of the energy at large times. The application of the \L ojasiewicz--Simon inequality will then imply that, once the flow passes sufficiently close to a critical point, it will be no longer able to ``escape'' from a neighborhood of such critical point. This will eventually imply the convergence of the flow.

The use of this kind of inequality for proving convergence of solutions to parabolic equations goes back to the semimal work of Simon~\cite{Si83}. The \L ojasiewicz--Simon inequality we shall employ follows from the abstract results of~\cite{Ch03} and the method is inspired by the strategy used in~\cite{ChFaSc09}, where the authors study the Willmore flow of closed surfaces. We mention that another successful application of this strategy is contained in~\cite{DaPoSp16}, where the authors study open curves with clamped boundary conditions. Finally, a recent application of this strategy to the flow of generalized elastic energies on Riemannian manifolds is contained in~\cite{Po20}. We remark that this strategy is rather general and it can be applied to many different geometric flows. We refer to~\cite{MaPo20} for a more detailed exposition of the method, in which the authors stress on the crucial general ingredients needed to run the argument.

\medskip

Now we introduce the above mentioned framework. Observe that for any fixed smooth curve $\gamma:\SS^1\to \R^2$ there exists $\rho(\gamma)>0$ such that if $\psi:\SS^1\to\R^2$ is a field of class $H^4$ with $\|\psi\|_{H^4}\le \rho$, then $\gamma + \psi$ is a regular curve of class $H^4$. Whenever $\gamma$ is fixed, we will always assume that $\rho=\rho(\gamma)>0$ is sufficiently small so that $\gamma+\psi$ is a regular curve for any field $\psi$ with $\|\psi\|_{H^4}\le \rho$. Hence the following definition is well posed.

\begin{dfnz}
Let $\gamma:\SS^1\to\R^2$ be a regular smooth closed curve and $\rho=\rho(\gamma)>0$ sufficiently small. We define
\[
H^{m,\perp}_\gamma \coloneqq \left\{ \psi\in H^m(\SS^1,\R^2) \,\,|\,\, \scal{\psi(x),\tau(x)}=0 \,\,\,\mbox{ a.e. on } \,\SS^1 \right\},
\]
for any $m\in\N$ and we denote $L^{2,\perp}_\gamma\equiv H^{0,\perp}_\gamma$. Moreover we define
\[
E:B_\rho(0)\subseteq H^{4,\perp}_\gamma \to \R 
\qquad \qquad E(\psi)\coloneqq \mathcal{E}(\gamma+\psi).
\]
\end{dfnz}

For a fixed  smooth curve $\gamma:\SS^1\to\R^2$, the functional $E$ is defined on an open subset of an Hilbert spaces. Therefore, we can treat the first and second variations of $\mathcal{E}$ as functionals over such Hilbert spaces. More precisely, we know that in classical functional analysis if $F:B_\rho\subseteq V\to\R$ is a twice Gateaux differentiable functional defined on a ball $B_\rho$ in a Banach space $V$, then $\delta F: B_\rho \to V^\star $ and $\delta^2 F: B_\rho \to L(V,V^\star)$, where
\[
\delta F (v)[w] \coloneqq \frac{d}{d\varepsilon} F(v + \varepsilon w) \Big\vert_{\varepsilon=0},
\]
\[
\delta^2 F (v) [w][z]\coloneqq \delta^2 F (v) [w,z] \coloneqq \frac{d}{d\eta}\frac{d}{d\varepsilon} F(v + \varepsilon w + \eta z) \Big\vert_{\varepsilon=0}\Big\vert_{\eta=0},
\]
where $(\cdot)^\star$ denotes the dual space of $(\cdot)$ and $L(V,V^\star)$ is the space of bounded linear functionals from $V$ to $V^\star$. We adopted the notation $\delta^2 F (v) [w,z]$ by the fact that the second variation can be also seen as a bilinear symmetric form on $V$, indeed $\delta^2 F (v) [w][z]=\delta^2 F (v) [z][w]$. In such a setting we have
\[
\delta E : B_\rho(0) \subseteq H^{4,\perp}_\gamma \to (H^{4,\perp}_\gamma)^\star 
\qquad\qquad \delta E(\psi)[\varphi] \coloneqq \frac{d}{d\varepsilon} E(\psi + \varepsilon\varphi) \Big\vert_{\varepsilon=0},
\]
and the second variation functional
\[
\begin{split}
&\delta^2 E :  B_\rho(0) \subseteq H^{4,\perp}_\gamma \to L(H^{4,\perp}_\gamma,(H^{4,\perp}_\gamma)^\star) \\
 &\delta^2 E(\psi)[\varphi,\zeta] \coloneqq \frac{d}{d\eta}\frac{d}{d\varepsilon} E(\psi + \varepsilon\varphi + \eta \zeta) \Big\vert_{\varepsilon=0}\Big\vert_{\eta=0}.
\end{split}
\]
We will actually only need to consider $\delta^2 E$ evaluated at $\psi=0$, that is, the second variation of $\mathcal{E}$ at the given curve $\gamma$.

The study carried out in~\Cref{sec:FirstVariation}  shows that the functional $E$ is Fr\'{e}chet differentiable with
\[
\delta E (\psi) [\varphi] = \int_{\SS^1} \scal{ 2(\ders^\perp)^2 \boldsymbol{\kappa}_{\gamma+\psi} +|\boldsymbol{\kappa}_{\gamma+\psi}|^2\boldsymbol{\kappa}_{\gamma+\psi} - \boldsymbol{\kappa}_{\gamma+\psi}, \varphi}\,\mathrm{d}s,
\]
where $\boldsymbol{\kappa}_{\gamma+\psi}$ is the curvature vector of the curve $\gamma+\psi$ and both arclength derivative $\ders$ and measure $\mathrm{d}s$ are understood with respect to the curve $\gamma+\psi$.
The explicit formula for such first variation functionals shows that actually $\delta E(\psi) $  belongs to the smaller dual space $(L^{2,\perp}_\gamma)^\star$. Indeed $\delta E(\psi) $ is represented in $L^2$-duality as
\begin{equation}\label{eq:GradienteNormalFields}
\begin{split}
    \delta E(\psi)[\varphi] &= \left\langle |\gamma'+\psi'| \left(  2(\ders^\perp)^2 \boldsymbol{\kappa}_{\gamma+\psi} +|\boldsymbol{\kappa}_{\gamma+\psi}|^2\boldsymbol{\kappa}_{\gamma+\psi} - \boldsymbol{\kappa}_{\gamma+\psi} \right) , \varphi \right\rangle_{L^2(\mathrm{d}x)},
\end{split}
\end{equation}
for any $\varphi \in L^{2,\perp}_\gamma$, where the derivative $\ders^\perp$ is understood with respect to the curve $\gamma+\psi$.

Moreover, the results in~\Cref{sec:SecondVariation} similarly imply that the second variation $\delta^2 E (0) [\varphi,\cdot]$ evaluated at some $\varphi\in H^{4,\perp}_\gamma$ belongs to the smaller dual space $(L^{2,\perp}_\gamma)^\star$ via the pairing
\begin{equation}\label{eq:SecondVariation2}
\delta^2 E (0) [\varphi,\zeta] = \left\langle 
|\gamma'| \left((\ders^\perp)^4 \varphi + \Omega(\varphi) \right)
,\zeta\right\rangle_{L^2(\mathrm{d}x)},
\end{equation}
for any $\zeta \in L^{2,\perp}_\gamma$, where $\Omega:H^{4,\perp}_\gamma \to L^{2,\perp}_\gamma $ is a compact operator and the derivative $\ders^\perp$ is understood with respect to the curve $\gamma$.

In this setting, we can prove the following properties on the first and second variations.

\begin{prop}\label{prop:AnalyticityFredholm}
Let $\gamma:\SS^1\to\R^2$ be a regular smooth closed curve and $\rho=\rho(\gamma)>0$ sufficiently small. Then the following holds true.
\begin{enumerate}
    \item The functions $E:B_\rho\subseteq H^{4,\perp}_\gamma\to \R$ and $\delta E: B_\rho\subseteq H^{4,\perp}_\gamma\to  (L^{2,\perp}_\gamma)^\star$ are analytic.
    \item The second variation operator $\delta^2 E(0): H^{4,\perp}_\gamma \to (L^{2,\perp}_\gamma)^\star$, which is defined by
    \[
    \delta^2 E(0)[\varphi][\zeta] \coloneqq \delta^2 E(0)[\varphi,\zeta] 
    \qquad \forall \varphi \in H^{4,\perp}_\gamma,\,\,\forall\,\zeta \in L^{2,\perp}(\gamma),
    \]
    is a Fredholm operator of index zero, i.e., ${\rm dim } \, \ker \delta^2 E(0) = {\rm codim }\, ({\rm Imm } \, \delta^2 E(0))$ is finite.
\end{enumerate}
\end{prop}

\begin{proof}
The fact that both $E$ and $\delta E$ are analytic maps follows from the fact that such operators are compositions and sums of analytic functions. For a detailed proof of this fact we refer to~\cite[Lemma 3.4]{DaPoSp16}.

Now we prove the second statement. Consider the operator $\mathcal{L}:H^{4,\perp}_\gamma \to L^{2,\perp}_\gamma$ defined by
\[
\mathcal{L}(\varphi) = |\gamma'|(\ders^\perp)^4 \varphi + |\gamma'|\Omega(\varphi),
\]
where $\Omega$ is as in~\eqref{eq:SecondVariation2}. We clearly have that $\delta^2 E(0): H^{4,\perp}_\gamma \to (L^{2,\perp}_\gamma)^\star$ is Fredholm of index zero if and only if $\mathcal{L}$ is. Since $|\gamma'|\Omega(\cdot)$ is compact, the thesis is then equivalent to say that $ H^{4,\perp}_\gamma \ni \,\varphi \mapsto |\gamma'| (\ders^\perp)^4 \varphi \, \in L^{2,\perp}_\gamma$ is Fredholm of index zero (see~\cite[Section 19.1, Corollary 19.1.8]{HormanderIII}). Since $|\gamma'|$ is uniformly bounded away from zero, the thesis is equivalent to prove that the operator
\[
(\ders^\perp)^4: H^{4,\perp}_\gamma \to L^{2,\perp}_\gamma
\]
is Fredholm of index zero. By~\cite[Corollary 19.1.8]{HormanderIII}, as ${\rm id} : H^{4,\perp}_\gamma \to L^{2,\perp}_\gamma$ is compact, this is equivalent to show that the operator
\[
{\rm id} + (\ders^\perp)^4: H^{4,\perp}_\gamma \to L^{2,\perp}_\gamma
\]
is Fredholm of index zero. We can prove, in fact, that ${\rm id} + (\ders^\perp)^4$ is even invertible.

Injectivity follows as if $\varphi + (\ders^\perp)^4 \varphi =0$, then multiplication by $\varphi$ and integration by parts give
\[
\int |(\ders^\perp)^2\varphi|^2 + |\varphi|^2 \, \mathrm{d}s = 0,
\]
and then $\varphi=0$.

Let now $X\in L^{2,\perp}_\gamma$ be any field and consider the continuous functional $F:H^{2,\perp}_\gamma \to \R$ defined by
\[
F(\varphi) \coloneqq \int_{\SS^1} \frac12 |(\ders^\perp)^2\varphi|^2 + \frac12 |\varphi|^2 - \scal{\varphi, X}\,\mathrm{d}s.
\]
The explicit computation shows that
\begin{equation}\label{eq:DoubleNormalDerivative}
    (\ders^\perp)^2\varphi = \partial^2_s\varphi + \left( 2\langle \partial_s\varphi,\boldsymbol{\kappa}\rangle + \langle \varphi, \partial_s \boldsymbol{\kappa}\rangle \right) \tau  - \langle \partial_s\varphi,\tau\rangle \boldsymbol{\kappa}.
\end{equation}
Since $\int |\partial_s\varphi|^2\,\mathrm{d}s = - \int \scal{\varphi,\ders \varphi}\,\mathrm{d}s \le \varepsilon\int |\ders\varphi|^2 + C(\varepsilon) \int |\varphi|^2$, computing $|\ders\varphi|^2$ using~\eqref{eq:DoubleNormalDerivative}, Young's inequality yields that
\[
\int_{\SS^1} |\varphi|^2 +  |\partial_s\varphi|^2 +  |\partial^2_s\varphi|^2 \, \mathrm{d}s \le C(\gamma) \int_{\SS^1}  \frac12 |(\ders^\perp)^2\varphi|^2 + \frac12 |\varphi|^2 \, \mathrm{d}s.
\]
Therefore, by direct methods in Calculus of Variations, it follows that there exists a minimizer $\phi$ of $F$ in $H^{2,\perp}_\gamma$. In particular $\phi$ solves
\begin{equation}\label{eq:EquationForSurjectivity}
\int_{\SS^1} \scal{(\ders^\perp)^2 \phi, (\ders^\perp)^2 \varphi } + \scal{\phi,\varphi} \,\mathrm{d}s = \int_{\SS^1} \scal{X, \varphi}\,\mathrm{d}s,
\end{equation}
for any $\varphi\in H^{2,\perp}_\gamma$. If we show that $\phi \in H^{4,\perp}_\gamma$, then $\phi + (\ders^\perp)^4\phi = X$ and surjectivity will be proved.
However, this follows by very standard arguments, simply noticing that once writing the integrand of the functional $F$ in terms of $\partial_s^2\varphi$, $\partial_s\varphi$ and $\varphi$, by means of equation~\eqref{eq:DoubleNormalDerivative}, its dependence on the highest order term $\partial_s^2\varphi$ is quadratic and the ``coefficients'' are given by the geometric quantities of $\gamma$, which is smooth.
\end{proof}

We remark that it is essential to employ normal fields in the proof of the Fredholmness properties of $\delta^2 E(0)$ in~\Cref{prop:AnalyticityFredholm} in order to rule out the tangential degeneracy related to the geometric nature of the energy functional.

\medskip

The above analysis of the second variation is exactly what is needed in order to derive a \L ojasiewicz--Simon gradient inequality. More precisely, we can rely on the following functional analytic result, which is a corollary of the results in~\cite{Ch03}. We recall the result here without proof.

\begin{prop}[{\cite[Corollary 2.6]{Po20}}]\label{prop:AbstractLoja}
Let $E:B_{\rho_0}(0) \subseteq V \to \R$ be an analytic map, where $V$ is a Banach space and $0$ is a critical point of $E$. Suppose that we have a Banach space $W=Z^\star\hookrightarrow V^\star$, where $V\hookrightarrow Z$, for some Banach space $Z$, that $\mathrm{Imm}\,\delta E\subseteq W$ and the map $\delta E:U\to W$ is analytic. Assume also that $\delta^2 E(0)\in L(V,W)$ and it is Fredholm of index zero.\\
Then there exist constants $C$, $\rho>0$ and $\theta\in(0,1/2]$ such that
\begin{equation*}
|E(\psi)-E(0)|^{1-\theta}\le C \| \delta E(\psi) \|_{W},
\end{equation*}
for any $\psi \in B_\rho(0)\subseteq U$.
\end{prop}

We can use~\Cref{prop:AbstractLoja} in order to derive a \L ojasiewic--Simon inequality on our elastic functional $\mathcal{E}$.

\begin{cor}[\L ojasiewicz--Simon gradient inequality]\label{cor:Loja}
Let $\gamma:\SS^1\to \R^2$ be a smooth critical point of $\mathcal{E}$. Then there exists $C,\sigma>0$ and $\theta \in (0,\tfrac12]$ such that
\begin{equation*}
	|\mathcal{E}(\gamma+\psi)-\mathcal{E}(\gamma)|^{1-\theta}\le C \| \delta E (\psi) \|_{(L^{2,\perp}_\gamma)^\star},
\end{equation*}
for any $\psi \in B_\sigma(0)\subseteq H^{4,\perp}_\gamma(\SS^1,\R^2)$.
\end{cor}

\begin{proof}
By~\Cref{prop:AnalyticityFredholm} we can apply~\Cref{prop:AbstractLoja} on the functional $E:B_{\rho_0}(0)\subseteq H^{4,\perp}_\gamma\to \R$ with the spaces $V= H^{4,\perp}_\gamma$ and $W=(L^{2,\perp}_\gamma)^\star$. This immediately implies the thesis.
\end{proof}

Let $\gamma:\SS^1\to \R^2$ be an embedded smooth curve. Choosing such $\rho$ small enough, the open set $U=\{p\in\R^2\,:\,d_\ga(p):=d(p,\gamma)<\rho\}$ is a tubular neighborhood of $\gamma$ with the property of {\em unique orthogonal projection}. The ``projection'' map $\pi:U\to\gamma(\SS^1)$ turns out to be $C^2$ in $U$ and given by $p\mapsto p-\nabla d^2_\ga(p)/2$, moreover the vector $\nabla d^2_\ga(p)$ is orthogonal to $\gamma$ at the point $\pi(p)$, see~\cite[Section 4]{MantegazzaMennucci} for instance.
Then, given $\varphi \in B_\rho(0)\subseteq H^4(\SS^1,\R^2)$, we can define a map $\chi:\SS^1\to\SS^1$ by
$$
\chi(x)=\gamma^{-1}\bigl[\pi\bigl(\ga(x)+\varphi(x)\bigr)\bigr],
$$
that is $C^2$ and invertible if $\ga'(x)+\varphi'(x)$ is never parallel to the unit vector $\nabla d_\ga(\ga(x)+\varphi(x))$, which is true if we have (possibly) chosen a smaller $\rho$ (so that $|\varphi|$ and $|\partial_x \varphi|$ are small and the claim follows as $\langle \gamma'(x),\nabla d_\ga(p)\rangle\to0$ as $p\to\ga(x)$).

We consider the vector field along $\gamma$ defined by
$$
\psi(\chi(x)):=\frac12 \nabla d^2_\gamma(\gamma(x)+\varphi (x))
$$
which is orthogonal to $\gamma$ at the point $\pi(\gamma(x)+\varphi(x))=\gamma(\chi(x))$,  for every $x\in\SS^1$, by construction. Hence $\psi$ is a normal vector field along the reparametrized curve $x\mapsto\gamma(\chi(x))$. Thus, we have
\begin{align*}
\gamma(\chi(x))+\psi (\chi(x))=&\,\pi\bigl(\ga(x)+\varphi(x)\bigr)+\nabla d^2_\ga(\ga(x)+\varphi(x))/2\\
=&\,\ga(x)+\varphi(x)-\nabla d^2_\ga(\ga(x)+\varphi(x))/2+\nabla d^2_\ga(\ga(x)+\varphi(x))/2\\
=&\,\ga(x)+\varphi(x).
\end{align*}
and we conclude that the curve $\ga+\varphi$ can be described by the (reparametrized) regular curve $(\ga+\psi)\circ\chi$, with $\psi\circ \chi$ normal vector field along $\gamma\circ \chi$. Moreover, by construction it follows that $\psi\circ \chi \in H^{4,\perp}_{\gamma\circ \chi}$.
Moreover, it is clear that if $\varphi\to0$ in $H^4$ then also $\psi\to0$ in $H^4$.

All this can be done also for a regular curve $\gamma:\SS^1\to\R^2$ which is only {\em immersed} (that is, it can have self--intersections), recalling that locally every immersion is an embedding and repeating the above argument a piece at a time along $\gamma$, getting also in this case a normal field $\psi$ describing a curve $\gamma+\varphi$ for $\varphi \in B_\rho(0)\subseteq H^4(\SS^1,\R^2)$.

\medskip

Now, if $\gamma=\gamma(t,x)$ is the smooth solution of the elastic flow wih datum $\gamma_0$, by~\Cref{Subconvergence} there exist a smooth critical point $\gamma_\infty$, a sequence $t_j\to+\infty$, a sequence of points $p_j\in\R^2$ and $\overline{\gamma}_{t_j}=\overline{\gamma}(t_j,\cdot)$ reparametrization of $\gamma(t_j,\cdot)$ such that
\begin{equation}\label{eq:subconvergence}
\overline{\gamma}_{t_j}-p_j \xrightarrow[j\to+\infty]{} \ga_\infty
\end{equation}
in $C^m(\SS^1,\R^n)$ for any $m\in \N$. Moreover, we know there are positive constants $C_L=C_L(\gamma_0)$ and $C(m,\gamma_0)$, for any $m\in \N$, such that
\[
\frac{1}{C_L} \le \ell (\gamma_t) \le C_L
\]
and
\begin{equation}\label{eq:StimaParabolica}
\| (\partial_s^\perp)^m \boldsymbol{\kappa}(t,\cdot) \|_{L^2(\mathrm{d}s)}\le C(m,\gamma_0)
\end{equation} 
for every $t\ge0$.

If for suitable times $t\in J$ we can write $\gamma=\gamma_\infty+\varphi_t$ with $\|\varphi_t\|_{H^4}<\rho=\rho_{\gamma_\infty}$ small enough, then it is an immediate computation to see that, if we describe $\gamma$ as a ``normal graph'' reparametrization along $\gamma_\infty$ by $\gamma_\infty+\psi_t$ as in the above discussion, then 
\begin{equation}\label{eq:StimaParabolica2}
\|\psi_t\|_{H^m}\le C(m,\gamma_0,\gamma_\infty)\,,
\end{equation} 
for every $m\in\N$ for any $t\in J$.

We are finally ready for proving the desired smooth convergence of the flow.

\begin{proof}[Proof of~\Cref{FullConvergence}]

Let us set $\gamma_t:=\gamma(t,\cdot)$ and we let $\ga_\infty$, $t_j$, $p_j$ and $\overline{\gamma}_{t_j}=\overline{\gamma}(t_j,\cdot)$ be as in~\eqref{eq:subconvergence}. Since the energy is non-increasing along the flow, we can assume that $\mathcal{E}(\ga_t)\searrow\mathcal{E}(\ga_\infty)$, as $t\to+\infty$ and $\mathcal{E}(\ga_t) > \mathcal{E}(\ga_\infty)$ for any $t$.
Thus, it is well defined the positive function
\[
H(t)= \left[\mathcal{E}(\ga_t) - \mathcal{E}(\ga_\infty) \right]^\theta,
\]
where $\theta\in(0,1/2]$ is given by~\Cref{cor:Loja} applied on the curve $\ga_\infty$. The function $H$ is monotone decreasing and converging to zero as $t\to+\infty$ (hence it is bounded above by $H(0)= \left[\mathcal{E}(\ga_0) - \mathcal{E}(\ga_\infty) \right]^\theta$).

Now let $m\ge6$ be a fixed integer. By Proposition~\ref{Subconvergence}, for any $\ep>0$ there exists $j_\ep\in\N$ such that
 \[
 \| \overline{\ga}_{t_{j_\ep}} - p_{j_\ep} -\ga_\infty \|_{C^m(\SS^1,\R^n)}\le \ep\qquad\text{ and }\qquad H(t_{j_\ep})\le\ep.
 \]
 Choosing $\varepsilon>0$ small enough, in order that 
 $$
 (\overline{\ga}_{t_{j_\ep}} - p_{j_\ep} -\ga_\infty)\in B_{\rho_{\ga_\infty}}(0)\subseteq H^4(\SS^1,\R^n),
 $$
for every $t$ in some interval $[t_{j_\varepsilon},t_{j_\ep} + \delta)$ there exists $\psi_t \in H^{4,\perp}_{\ga_\infty}$ such that the curve $\widetilde{\ga}_t = \ga_\infty + \psi_t$ is the ``normal graph'' reparametrization of $\gamma_t-p_{j_\ep}$. Hence 
 $$
 (\partial_t \widetilde{\ga})^\perp = -(2(\partial_s^\perp)^2 \boldsymbol{\kappa}_{\widetilde\ga_t} - |\boldsymbol{\kappa}_{\widetilde\ga_t}|^2\boldsymbol{\kappa}_{\widetilde\ga_t} + \boldsymbol{\kappa}_{\widetilde\ga_t}),
 $$
as the flow is invariant by translation and changing the parametrization of the evolving curves only affects the tangential part of the velocity. Since $\widetilde{\ga}_{t_\ep}$ is such reparametrization of $\overline{\ga}_{t_{j_\ep}} - p_{j_\ep}$ and this latter is close in $C^m(\SS^1,\R^n)$ to $\ga_\infty$, possibly choosing smaller $\ep,\delta>0$ above, it easily follows that for every $t\in[t_{j_\varepsilon},t_{j_\ep} + \delta)$ there holds
 $$
 \| \psi_t \|_{H^4} <\sigma,
 $$
where $\sigma>0$ is as in Corollary~\ref{cor:Loja} applied on $\ga_\infty$ and we possibly choose it smaller than the constant $\rho_\infty$.
 
We want now to prove that if $\ep>0$ is sufficiently small, then actually we can choose $\delta=+\infty$ and $\| \psi_t \|_{H^4} <\sigma$ for every time.

For $E$ as in Corollary~\ref{cor:Loja}, we have
 \begin{align}
 [\mathcal{E}(\ga_t)-\mathcal{E}(\ga_\infty)]^{1-\theta}=&\,[\mathcal{E}(\widetilde\ga_t)-\mathcal{E}(\ga_\infty)]^{1-\theta}\nonumber\\
 =&\,\left[E(\psi_t) - E(0) \right]^{1-\theta}\nonumber\\
 \le&\, C_1(\gamma_\infty,\sigma) \| \delta E (\psi_t) \|_{(L^{2,\perp}_{\ga_\infty})^\star}\nonumber\\
 =&\,C_1(\gamma_\infty,\sigma)\sup_{\Vert S\Vert_{L^{2,\perp}_{\ga_\infty}=1}}\int_{\SS^1}  \left\langle |\widetilde\ga_t'|\bigl(2(\partial_s^\perp)^2 
 \boldsymbol{\kappa}_{\widetilde\ga_t} + |\boldsymbol{\kappa}_{\widetilde\ga_t}|^2\boldsymbol{\kappa}_{\widetilde\ga_t} - \boldsymbol{\kappa}_{\widetilde\ga_t}\bigr), S \right\rangle\,\mathrm{d}x\nonumber\\
\leq&\,C_1(\gamma_\infty,\sigma)\sup_{\Vert S\Vert_{L^2(\SS^1,\R^n)=1}}\int_{\SS^1}  \left\langle |\widetilde\ga_t'|\bigl(2(\partial_s^\perp)^2 \boldsymbol{\kappa}_{\widetilde\ga_t} + |\boldsymbol{\kappa}_{\widetilde\ga_t}|^2\boldsymbol{\kappa}_{\widetilde\ga_t} - \boldsymbol{\kappa}_{\widetilde\ga_t}\bigr), S \right\rangle\,\mathrm{d}x\nonumber\\
 =&\,C_1(\gamma_\infty,\sigma)\left(\int_{\SS^1}  |\widetilde\ga_t'|^2\bigl\vert 2(\partial_s^\perp)^2 \boldsymbol{\kappa}_{\widetilde\ga_t} + |\boldsymbol{\kappa}_{\widetilde\ga_t}|^2\boldsymbol{\kappa}_{\widetilde\ga_t} - \boldsymbol{\kappa}_{\widetilde\ga_t}\bigr\vert^2\,\mathrm{d}x\right)^{1/2}\label{eqcar1}
 \end{align}
 where we can assume that $C_1(\gamma_\infty,\sigma)\geq 1$.\\
 Now, $\scal{\widetilde{\ga}_t,\tau_{\ga_\infty}}=\scal{\ga_\infty,\tau_{\ga_\infty}}$ is time independent, then $\scal{\pa_t \widetilde{\ga},\tau_{\ga_\infty}}=0$ and possibly taking a smaller $\sigma>0$, we can suppose that $|\tau_{\ga_\infty}-\tau_{\widetilde{\ga}}|\le \tfrac12$ for any $t\ge t_{j_\ep}$ such that $\| \psi_t \|_{H^4} <\sigma$. Hence,
 \[
 |(\pa_t\widetilde{\ga})^\perp|= | \pa_t \widetilde{\ga} - \scal{  \pa_t \widetilde{\ga} , \tau_{\widetilde{\ga}} } \tau_{\widetilde{\ga}} | 
 =  | \pa_t \widetilde{\ga} + \scal{  \pa_t \widetilde{\ga} , \tau_{\ga_\infty} -  \tau_{\widetilde{\ga}} } \tau_{\widetilde{\ga}} |\ge |\pa_t\widetilde{\ga}| - |\pa_t\widetilde{\ga}||\tau_{\ga_\infty}-\tau_{\widetilde{\ga}}|
  \ge \frac12 |\pa_t\widetilde{\ga}|.
 \]
 Differentiating $H$, we then get
 \begin{equation}
 \begin{split}
 \frac{d}{dt} H(t)=&\,\frac{d}{dt}[\mathcal{E}(\widetilde\ga_t)-\mathcal{E}(\ga_\infty)]^{\theta}\\
 =&\,-\theta H^{\frac{\theta-1}{\theta}}\int_{\SS^1}  |\widetilde\ga_t'|\bigl\vert 2(\partial_s^\perp)^2\boldsymbol{\kappa}_{\widetilde\ga_t} + |\boldsymbol{\kappa}_{\widetilde\ga_t}|^2\boldsymbol{\kappa}_{\widetilde\ga_t} - \boldsymbol{\kappa}_{\widetilde\ga_t}\bigr\vert^2\,\mathrm{d}x\nonumber\\
 \leq&\,-\theta H^{\frac{\theta-1}{\theta}}C_2(\ga_\infty,\sigma)\left(\int_{\SS^1}  \bigl\vert(\partial_t\widetilde\ga)^{\perp}\bigr\vert^2\,\mathrm{d}x\right)^{1/2}\left(\int_{\SS^1}  |\widetilde\ga_t'|^2\bigl\vert 2(\partial_s^\perp)^2\boldsymbol{\kappa}_{\widetilde\ga_t} + |\boldsymbol{\kappa}_{\widetilde\ga_t}|^2\boldsymbol{\kappa}_{\widetilde\ga_t} - \boldsymbol{\kappa}_{\widetilde\ga_t}\bigr\vert^2\,\mathrm{d}x\right)^{1/2}\nonumber\\
 \leq&\,-H^{\frac{\theta-1}{\theta}}C(\ga_\infty,\sigma) \|\pa_t\widetilde\ga \|_{L^2(\mathrm{d}x)}[\mathcal{E}(\widetilde\ga_t)-\mathcal{E}(\widetilde\ga_\infty)]^{1-\theta}\\ 
 =&\,-C(\ga_\infty,\sigma)\|\pa_t\widetilde\ga \|_{L^2(\mathrm{d}x)},\label{eqcar777}
 \end{split}
 \end{equation}
 where $C(\ga_\infty,\sigma)=\theta C_2(\ga_\infty,\sigma)/2C_1(\ga_\infty,\sigma)$. This inequality clearly implies the estimate
\begin{equation}\label{eqcar778}
 C(\ga_\infty,\sigma)\int_{\xi_1}^{\xi_2}\|\pa_t\widetilde\ga \|_{L^2(\mathrm{d}x)}\,\mathrm{d}t\leq H(\xi_1)-H(\xi_2)\leq H(\xi_1)
 \end{equation}
 for every $t_{j_\ep}\le \xi_1<\xi_2<t_{j_\ep}+\delta$ such that $\| \psi_t \|_{H^4} <\sigma$. Hence, for such $\xi_1,\xi_2$ we have
 \begin{equation}
 \begin{split}
 \Vert\widetilde\ga_{\xi_2}-\widetilde\ga_{\xi_1}\Vert_{L^2(\mathrm{d}x)}
 =&\,\left(\int_{\SS^1} |\widetilde\ga_{\xi_2}(x)-\widetilde\ga_{\xi_1}(x)|^2\,\mathrm{d}x\right)^{1/2}\\
 \leq&\,\biggl(\int_{\SS^1} \left(\int_{\xi_1}^{\xi_2}\partial_t\widetilde\ga(t,x)\,\mathrm{d}t\,\right)^2\,\mathrm{d}x\biggr)^{1/2}\\
 =&\,\left\Vert\int_{\xi_1}^{\xi_2}\partial_t\widetilde\ga\,\mathrm{d}t\,\right\Vert_{L^2(\mathrm{d}x)}\\
 \leq&\,\int_{\xi_1}^{\xi_2}\Vert\partial_t\widetilde\ga\Vert_{L^2(\mathrm{d}x)}\,\mathrm{d}t\\
 \leq&\, \frac{{H(\xi_1)}}{C(\ga_\infty,\sigma)}\\
 \le&\, \frac{\ep}{C(\ga_\infty,\sigma)},
 \end{split}\label{eq:CauchyL2}
 \end{equation}
 where we used that $H(\xi_1)\le H(t_{j_\ep})\le \ep$ and the fact that $\bigl\Vert\int_{\xi_1}^{\xi_2} v\,\mathrm{d}t\,\bigr\Vert_{L^2(\mathrm{d}x)}\leq\int_{\xi_1}^{\xi_2}\Vert v\Vert_{L^2(\mathrm{d}x)}\,\mathrm{d}t$, which easily follows from Holder inequality.

 Therefore, for $t\ge t_{j_\ep}$ such that $\| \psi_t \|_{H^4} <\sigma$, we have
 \[
 \|\psi_t\|_{L^2(\mathrm{d}x)}=\| \widetilde{\ga}_t - \ga_\infty \|_{L^2(\mathrm{d}x)}\le \| \widetilde{\ga}_t - \widetilde{\ga}_{t_{j_\ep}} \|_{L^2(\mathrm{d}x)}  + 
 \|   \widetilde{\ga}_{t_{j_\ep}}   - \ga_\infty  \|_{L^2(\mathrm{d}x)}\le  \frac{\ep}{C(\ga_\infty,\sigma)} +\ep \sqrt{2\pi}.
 \]
 Then, by means of Gagliardo--Nirenberg interpolation inequalities (see~\cite{adams} or~\cite{Aubin}, for instance) and estimates~\eqref{eq:StimaParabolica2}, for every $l\geq 4$, we have
 \[
 \|\psi_t \|_{H^l} \le C \|\psi_t \|_{H^{l+1}}^a \|\psi_t \|_{L^2(\mathrm{d}x)}^{1-a} \le C(l,\ga_0,\ga_\infty,\sigma)\ep^{1-a},
 \]
 for some $a\in(0,1)$ and any $t\ge t_{j_\ep}$ such that $\| \psi_t \|_{H^4} <\sigma$.\\
 In particular setting $l+1=m\ge6$, if $\ep>0$ was chosen sufficiently small depending only on $\ga_0$, $\ga_\infty$ and $\sigma$, then $\|\psi_t\|_{H^4}<\sigma/2$ for any time $t\ge t_{j_\ep}$, which means that we could have chosen $\delta=+\infty$ in the previous discussion.

 Then, from estimate~\eqref{eq:CauchyL2} it follows that $\widetilde{\ga}_t$ is a Cauchy sequence in $L^2(\mathrm{d}x)$ as $t\to+\infty$, therefore $\widetilde{\ga}_t $ converges in $L^2(\mathrm{d}x)$ as $t\to+\infty$ to some limit curve $\widetilde{\gamma}_\infty$ (not necessarily coincident with $\gamma_\infty$). Moreover, by means of the above interpolation inequalities, repeating the argument for higher $m$ we see that such convergence is actually in $H^m$ for every $m\in\N$, hence in $C^m(\SS^1,\R^n)$ for every $m\in\N$, by Sobolev embedding theorem. This implies that $\widetilde{\gamma}_\infty$ is a smooth critical point of $\mathcal{E}$. As the original flow $\ga_t$ is a fixed translation of $\widetilde{\ga}_t$, up to reparametrization, this completes the proof.
\end{proof}

Collecting the results we proved about the elastic flow of closed curves, we can state the following comprehensive theorem.

\begin{teo}\label{teo:FullExistenceAndConvergence}
Let $\gamma_0:\mathbb{S}^1\to \R^2$ be a smooth closed curve. Then there exists a unique solution $\gamma:[0,+\infty)\times \SS^1\to \R^2$ to the elastic flow
\[
\begin{cases}
\partial_t \gamma = -\left( 2(\ders^\perp)^2\boldsymbol{\kappa} + |\boldsymbol{\kappa}|^2\boldsymbol{\kappa} -\boldsymbol{\kappa} \right) & \mbox{ on } [0,+\infty)\times \SS^1,\\
\gamma(0,x)=\gamma_0(x) & \mbox{ on } \SS^1.
\end{cases}
\]
Moreover there exists a smooth critical point $\gamma_\infty$ of $\mathcal{E}$ such that $\gamma(t,\cdot)\to \gamma_\infty(\cdot)$ in $C^m(\SS^1)$ for any $m\in\N$, up to reparametrization.
\end{teo}

\begin{rem}
We remark that~\Cref{teo:FullExistenceAndConvergence} is true exactly as stated for the analogously defined flow in the Euclidean spaces $\R^n$ for any $n\ge 2$ (see~\cite{MaPo20}). Indeed, it is immediate to see that the proof above generalizes to higher codimension. We observe that the very same statement holds for the suitably defined elastic flow defined in the hyperbolic plane and in the two-sphere by~\cite[Corollary 1.2]{Po20}. It is likely that smooth convergence of the elastic flow still holds true in hyperbolic spaces and spheres of any dimension $\ge2$ and, more generally, in homogeneous Riemannian manifolds, that is, complete Riemannian manifolds such that the group of isometries acts transitively on them. For further results and comments about the convergence of the elastic flow in Riemannian manifolds we refer to~\cite{Po20}.
\end{rem}

Let us conclude by stating the analogous full convergence result proved for the elastic flow of open curves with clamped boundary conditions.

\begin{teo}[\cite{Ch12, DaPoSp16}]
	Let $\gamma_0:[0,1]\to \R^n$ be a smooth curve. Then there exists a unique solution $\gamma:[0,+\infty)\times[0,1]\to\R^n$ to the elastic flow satisfying the clamped boundary conditions
	\[
	\gamma(t,0)=\gamma_0(0), \quad
	\gamma(t,1)=\gamma_0(1), \quad
	\ders\gamma(t,0)=\tau_{\gamma_0}(0), \quad
	\ders\gamma(t,1)=\tau_{\gamma_0}(1), \quad
	\]
	with initial datum $\gamma_0$. Moreover there exists a smooth critical point $\gamma_\infty$ of $\mathcal{E}$ subjected to the above clamped boundary conditions such that $\gamma(t,\cdot)\to \gamma_\infty(\cdot)$ in $C^m([0,1])$ for any $m\in\N$, up to reparametrization.
\end{teo}

\section{Open questions}\label{Openprob}

We conclude the paper by mentioning some related open problems.

\begin{itemize}
\item In Theorem~\ref{thm:LongTimeNetworks} a description of the possible behaviors
as $t\to T_{\max}$ is given for evolving networks subjected to Navier boundary conditions.
When instead clamped boundary conditions are imposed, only short time existence is known~\cite{GaMePl1}. One would like to investigate further this flow of networks as time approaches the maximal time of existence.
Recent results~\cite{DePlPo} on the minimization of $\mathcal{E}_\mu$ among 
networks whose curves meet with fixed angles suggest that an analogous of 
Theorem~\ref{thm:LongTimeNetworks} is expected: either $T_{\max}=\infty$
or as $t\to T_{\max}$  the length of at least one curve of the network 
could go to zero.
\item
In Section~\ref{sec:LongTimeExistence} we described a couple of numerical examples by Robert N\"{u}rnberg in which some curves vanish or the amplitude of the angles at the junctions goes to zero. 
It is an open problem to explicitly find an example of an evolving network developing such phenomena. More generally, one wants to give a more accurate description of the onset of singularities during the flow.
\item In the case of the flow of networks with Navier boundary conditions
estimates of the type
\begin{equation*}
\frac{d}{dt}\int_{\mathcal{N}_t}
\left\lvert\partial_s^n k\right\rvert^2 \,\mathrm{d}s
\leq C(\mathcal{E}_\mu(\mathcal{N}_0))\,, 
\end{equation*}
are shown for $n=2+4j$ with $j\in\mathbb{N}$  only for a special choice of the tangential velocity (see~\cite{DaChPo19}).
One could ask whether the same holds true for a general tangential velocity.
\item In the last section we show that if 
$\gamma_t$ is a solution of the elastic flow of closed curves in $[0,\infty)$,
then its support stays in a compact sets of $\mathbb{R}^2$ for any time. 
The same is true for open curves and networks with some endpoint fixed in the plane.
What about compact networks?
At the moment we are not able to exclude that
if the initial network $\mathcal{N}_0$ has no fixed endpoints (as in the case of a Theta)  as $t\to T_{\max}$ the entire configuration $\mathcal{N}_t$ ``escapes'' to infinity. 
\item Another related question asked by G. Huisken is the following: suppose that the support of an initial closed curve
$\gamma_0$ lies in the upper halfplane, is it possible to prove that 
there is no time $\tau$ such that the support of the solution 
at time $\tau$ lies completely in the lower halfplane?
\item Are there self--similar (for instance translating or rotating) solutions of the elastic flow?
\item Several variants of the elastic flow have been investigated, but an analysis of the 
elastic flow of closed curves that encloses a fixed (signed) area is missing.
\item At the moment no stability results are shown for the elastic flow of networks. More generally, one would understand whether an elastic flow of a general network defined for all times converges smoothly to a critical point, just as in the case of closed curves. Similarly, proving the stability of the flow would mean to understand whether an elastic flow of networks starting ``close to'' a critical point exists for all times and smoothly converges.
\item Is it possible to introduce a definition of weak solution (for instance by variational schemes such as minimizing movements) that is also capable to provide global existence in the case of networks?
We remark that all notions based on the maximum principle, such as viscosity solutions,
cannot work in this context, due to the high order of the evolution equation in the spatial variable.
\end{itemize}


\printbibliography[title={References}]

\typeout{get arXiv to do 4 passes: Label(s) may have changed. Rerun} 

\end{document}